\documentclass[12pt]{amsart}

\usepackage{dsfont}
\usepackage{tikz} \usepackage{tikz-cd}
\usetikzlibrary{decorations.pathreplacing,arrows}
\usetikzlibrary{decorations.pathmorphing}
\usetikzlibrary{calc}
\usepackage{amssymb}
\usepackage{amsmath}
\usepackage{enumerate}
\usepackage[all]{xy}

\usepackage{mathabx,epsfig}

\DeclareMathOperator\Loc{Loc}
\DeclareMathOperator\coker{coker}
\DeclareMathOperator\Naka{\mathcal N}
\DeclareMathOperator\cF{\mathcal F}
\DeclareMathOperator\Imm{Im}

\DeclareMathOperator\M{\mathcal M}
\DeclareMathOperator\Gr{Gr}
\DeclareMathOperator\PPP{\mathbb P}
\DeclareMathOperator\h{\hbar}
\DeclareMathOperator\End{End}
\DeclareMathOperator\DD{\mathcal D}
\def\tDD{\widetilde{\mathcal D}}

\DeclareMathOperator\A{\mathbb A}
\DeclareMathOperator\R{\mathbb R}

\DeclareMathOperator\MM{\mathbb M}
\DeclareMathOperator\NN{\mathbb N}
\DeclareMathOperator\C{\mathbb C}
\DeclareMathOperator\Z{\mathbb Z}

\DeclareMathOperator\T{\mathbb T}

\DeclareMathOperator\HH{\mathcal H}
\DeclareMathOperator\KK{\mathbb K}

\DeclareMathOperator\G{\mathcal G}
\DeclareMathOperator\Hom{\mathrm Hom}

\DeclareMathOperator\pt{pt}

\DeclareMathOperator{\GL}{GL}

\DeclareMathOperator\Fl{\mathcal F}

\DeclareMathOperator\Stab{Stab}

\def\ttt#1{\text{\tt #1}}
\newcommand\bs{{\color{blue} \char`\\}}
\newcommand\fs{{\color{red}/}}
\newcommand\charge{\text{charge}}
\newcommand\mult{d}
\def\rr{r}
\def\cc{c}
\DeclareMathOperator\Ch{\mathcal C}

\DeclareMathOperator\BCT{BCT}
\def\uu{\mathbf u}
\def\vv{\mathbf v}
\def\hb{\mathbf h}
\def\HH{\mathbb H^*}
\DeclareMathOperator\Leaf{Leaf}
\DeclareMathOperator\Slope{Slope}
\DeclareMathOperator\Sym{Sym}

\newtheorem{fact}{Fact}[section]
\newtheorem{lemma}[fact]{Lemma}
\newtheorem{theorem}[fact]{Theorem}
\newtheorem{definition}[fact]{Definition}
\newtheorem{example}[fact]{Example}
\newtheorem{rremark}[fact]{Remark}
\newenvironment{remark}{\begin{rremark} \rm}{\end{rremark}}
\newtheorem{proposition}[fact]{Proposition}
\newtheorem{corollary}[fact]{Corollary}
\newtheorem{conjecture}[fact]{Conjecture}
\newtheorem{assumption}[fact]{Assumption}

\author{R. Rim\'anyi}
\address{Department of Mathematics, University of North Carolina at Chapel Hill, USA}
\email{rimanyi@email.unc.edu}

\author{Y. Shou}
\address{Department of Mathematics, University of North Carolina at Chapel Hill, USA}
\email{yshou@live.unc.edu}


\title[Bow varieties]{Bow varieties---geometry, combinatorics, \\ characteristic classes}

\begin{document}

\begin{abstract} 
Cherkis bow varieties are believed to be the set of spaces where 3d mirror symmetry for characteristic classes can be observed. We describe geometric structures on a large class of Cherkis bow varieties by developing the necessary combinatorial presentations, including binary contingency tables and skein diagrams. We make the first steps toward the sought after statement for 3d mirror symmetry for characteristic classes by conjecturing a formula for cohomological stable envelopes. Additionally we provide an account of the full statement, with examples, for elliptic stable envelopes.
\end{abstract}

\maketitle

\tableofcontents

\section{Introduction}
Assigning cohomology classes to geometrically relevant subsets of a variety is an effective technique in enumerative algebraic geometry. The prototype is the assignment of Schubert classes to the Schubert varieties of a Grassmannian. The study of these classes---e.g. their multiplicative properties, their combinatorics, their relation to representation theory and algebraic combinatorics---is an important part of Schubert Calculus.

The notion of Schubert class has natural generalizations. First, the Grassmannian can be replaced with other homogeneous spaces $G/P$, and even further to Nakajima quiver varieties. Second, the underlying spaces come with a torus action, and the Schubert classes can be defined in equivariant cohomology theory. Third, the cohomology theory itself can be generalized to extraordinary cohomologies. The ones proven to be useful are K theory and elliptic cohomology (these are the ones whose corresponding formal group law is in fact a one dimensional algebraic group). Fourth, the notion of {\em fundamental class of the closure} of a, say, Schubert cell, has a one-parameter ($\h$) deformation. The resulting notion in cohomology is called the Chern-Schwartz-MacPherson class (CSM), in K theory the Motivic Chern class (MC), and in elliptic cohomology the elliptic class---see eg. the survey \cite{Rh} and references therein.  

Recently, motivated by relations with quantum integrable systems, Okounkov and his coauthors introduced another characteristic class assigned to torus fixed points of Nakajima quiver varieties: the {\em stable envelope} (class) \cite{MO, O, AO}. Stable envelopes are defined axiomatically, and they exist in cohomology, K theory, elliptic cohomology. They depend on other parameters, and varying those parameters one finds geometric quantum group actions on cohomology, K theory, elliptic cohomology algebras. They govern solutions of differential and q-difference equations associated with the space. 

In certain settings both the CSM, MC, elliptic classes on the one hand, and the three versions of stable envelopes on the other hand, are defined. It turns out that in these situations the two are the same (after appropriate identifications). These coincidence results not only made the study of $\h$-deformed characteristic classes even more interesting than before, but it brought fresh ideas, and new ways to calculate them \cite{RV, FRW, AMSS1, AMSS2, RW1, KRW}.

As a result, a previously hidden structure emerged among characteristic classes. Due to some relations to mirror symmetry in $N=4, d=3$ supersymmetric gauge theory, the new structure is named {\em 3d mirror symmetry}---although maybe {\em S-duality} and {\em symplectic duality} are also related physical phenomena \cite{BFN, BDGH, GW, GMMS, IS}. The 3d mirror symmetry for characteristic classes is the following: for a space $X$ there is a 3d mirror dual space $X'$, such that their torus fixed points are in natural bijection, and the characteristic classes of corresponding fixed points match. The matching is sophisticated. Let $E(\omega,\sigma)$ denote the elliptic class (stable envelope) associated to the fixed point $\omega$, restricted to the fixed point $\sigma$ on $X$, and prime versions are the corresponding objects on $X'$. Then {\em 3d mirror symmetry for elliptic characteristic classes} mean the identity 
\[
E(\omega,\sigma)(u,v,\h)=\pm E(\sigma',\omega')(v,u,\h^{-1}).
\]
Here $u$ is the set of equivariant parameters, and $v$ is the set of new parameters called K\"ahler (or dynamical) parameters. 
Some examples for such statements are proven in \cite{RSVZ1, RSVZ2, RW2, SZ}---with ad hoc methods---but a general theory is not known yet. 

There are indications that the right collection of spaces where 3d mirror symmetry for characteristic classes is well displayed is the collection of Cherkis bow varieties. Bow varieties generalize quiver varieties. For bow varieties the equivariant and K\"ahler parameters emerge on an equal footing. Moreover, bow varieties are closed under a 3d mirror symmetry operation that switches the objects (5-branes) that the equivariant and K\"ahler parameters are assigned to. For bow varieties there is an extra operation---which does not exist for the smaller set of quiver varieties---called Hanany-Witten transition.

The first goal of this paper is to introduce Cherkis bow varieties to enumerative algebraic geometry. We work out many of the combinatorial structures necessary for their enumerative analysis: the tautological and tangent bundles, the combinatorics of torus fixed points, fixed point restriction maps in K theory and cohomology; and we relate all these data with the two new operations: 3d mirror symmetry and Hanany-Witten transition.  The appearance of knot diagram-like figures, as well as {\em binary contingency tables} of statistics are surprising fresh phenomena in Schubert calculus. 

The second goal of the paper is the study of stable envelopes for bow varieties, that we begin in Section \ref{sec:char}. We define cohomological stable envelopes and present a conjectured formula. This formula---while has some new features---generalizes {\em weight functions} of \cite{RTV}. Weight functions were originally constructed in relation with hypergeometric solutions of Knizhnik-Zamolodchikov-type equations \cite{TV1, TV2}. We explore the notion of elliptic stable envelopes through an example, and illustrate their 3d mirror symmetry. We emphasize how the pool of bow varieties is the natural setting for 3d mirror symmetry for characteristic classes.

\bigskip

The authors are indebted to Lev Rozansky, who explained to us bow varieties and many of the known and expected features of them. Without him this paper would not have been born. We also thank A. Okounkov, A. Smirnov, A. Varchenko, A. Weber, A. Yong for discussions on related subjects, and H. Nakajima for valuable comments. The first author was partially supported by the Simons Foundation grant 523882.

\section{Brane combinatorics}

\subsection{Brane diagrams}\label{sec:diagram}
Objects like this

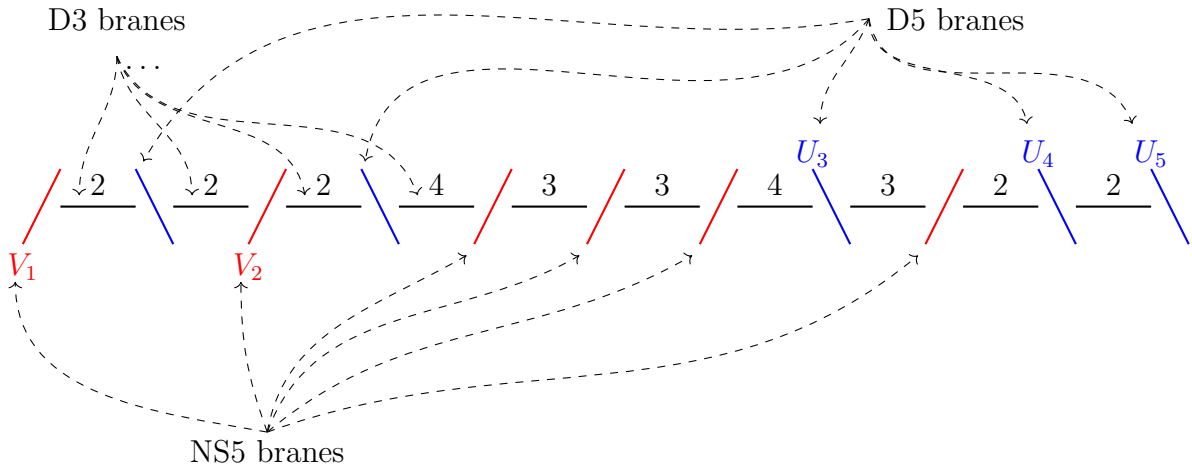
\begin{tikzpicture}[scale=.5]
\draw [thick,red] (0.5,0) --(1.5,2); 
\node[red] at (.5,-.6) {$V_1$};  %%
\draw[thick] (1.5,1)--(2.5,1) node [above] {$2$} -- (3.5,1);
\draw [thick,blue](4.5,0) --(3.5,2);  
\draw [thick](4.5,1)--(5.5,1) node [above] {$2$} -- (6.5,1);
\draw [thick,red](6.5,0) -- (7.5,2);  %%
\node[red] at (6.5,-.6) {$V_2$};
\draw [thick](7.5,1) --(8.5,1) node [above] {$2$} -- (9.5,1); 
\draw[thick,blue] (10.5,0) -- (9.5,2);  
\draw[thick] (10.5,1) --(11.5,1) node [above] {$4$} -- (12.5,1); 
\draw [thick,red](12.5,0) -- (13.5,2);   %%
\draw [thick](13.5,1) --(14.5,1) node [above] {$3$} -- (15.5,1);
\draw[thick,red] (15.5,0) -- (16.5,2);  %%
\draw [thick](16.5,1) --(17.5,1) node [above] {$3$} -- (18.5,1);  
\draw [thick,red](18.5,0) -- (19.5,2);  %%
\draw [thick](19.5,1) --(20.5,1) node [above] {$4$} -- (21.5,1);
\draw [thick,blue](22.5,0) -- (21.5,2);
\node[blue] at (21.5,2.4) {$U_3$};
\draw [thick](22.5,1) --(23.5,1) node [above] {$3$} -- (24.5,1);  
\draw[thick,red] (24.5,0) -- (25.5,2); 
\node[blue] at (27.5,2.4) {$U_4$};
\node[blue] at (30.5,2.4) {$U_5$};
\draw[thick] (25.5,1) --(26.5,1) node [above] {$2$} -- (27.5,1);
\draw [thick,blue](28.5,0) -- (27.5,2);  %%
\draw [thick](28.5,1) --(29.5,1) node [above] {$2$} -- (30.5,1);   
\draw [thick,blue](31.5,0) -- (30.5,2);   %%

\draw [black,dashed,->](7,-5) to [out=170,in=270] (0.3,-1);
\draw [black,dashed,->](7,-5) to [out=110,in=270] (6.3,-1);
\draw [black,dashed,->](7,-5) to [out=80,in=210] (12.3,-.3);
\draw [black,dashed,->](7,-5) to [out=70,in=210] (15.3,-.3);
\draw [black,dashed,->](7,-5) to [out=40,in=210] (18.3,-.3);
\draw [black,dashed,->](7,-5) to [out=20,in=220] (24.3,-.3);
\node at (7,-5.5) {NS5 branes};

\draw [black,dashed,->](23,6) to [out=190,in=60] (3.7,2.2);
\draw [black,dashed,->](23,6) to [out=230,in=80] (9.6,2.2);
\draw [black,dashed,->](23,6) to [out=250,in=90] (21.7,3.2);
\draw [black,dashed,->](23,6) to [out=270,in=100] (27.2,3.2);
\draw [black,dashed,->](23,6) to [out=280,in=110] (30,3.2);
\node at (25.3, 6) {D5 branes};

\node at (3,6) {D3 branes};
\draw [black,dashed,->](3,5) to [out=-90,in=90] (2,1.3);
\draw [black,dashed,->](3,5) to [out=-80,in=90] (5,1.3);
\draw [black,dashed,->](3,5) to [out=-70,in=100] (8,1.3);
\draw [black,dashed,->](3,5) to [out=-65,in=110] (11,1.3);
\node [color=black] at (3.7,4.7) {$\ldots$};
\end{tikzpicture}

\noindent will be called (type A) brane diagrams. That is, a brane diagram is a finite sequence, drawn from left to right, of 5-branes. A 5-brane is either a D5 brane\footnote{Dirichlet 5-brane} (we draw a SE-NW line \ttt{\bs}), or an NS5 brane\footnote{Neveu-Schwarz 5-brane} (we draw a SW-NE skew line \ttt{{\fs}}). Consecutive 5-branes are connected by D3 branes (we draw a horizontal line). Each D3 brane is decorated with its multiplicity, a non-negative integer. The list of multiplicities can be called the {\em dimension vector} of the diagram.

\subsection{Notations, convention}
In text we will write \ttt{{\fs}2\bs 2{\fs}2\bs 4{\fs}3{\fs}3{\fs}4\bs 3{\fs}2\bs 2\bs} for the brane diagram above. In figures we can also just draw one continuous line for the union of D3 branes. 

We can formally add D3 branes to the left and to the right of the brane diagram with multiplicities 0; we consider this extended brane diagram the same.

The D5 branes will be denoted by $U$, and we number them from left to right as $U_1, U_2, \ldots, U_n$. The NS5 branes will be denoted by $V$, and we number them also from left to right as $V_1, V_2, \ldots, V_m$. The multiplicity of a D3 brane $X$ is denoted by $\mult_X$. The D3 brane directly to the left and right of $U$ (or $V$) will be denoted by $U^-, U^+$ (or $V^-, V^+$). Similarly the 5-branes directly to the left and right of a D3 brane $X$ will be called $X^-$, $X^+$. We will write $U^\ddag$ for $(U^+)^+$ and $U^=$ for $(U^-)^-$. 

\subsection{Different views of the brane diagram}\label{sec:3dimview}
We introduced a brane diagram as a combinatorial code. Yet, the following two points of views are relevant. 

We can regard the brane diagram in the three-dimensional 
\begin{tikzpicture}[baseline=0pt,scale=.4]
\draw[->] (0,0) -- (1,0); 
\node at (1.3,0) {\tiny $x$};
\draw[->] (0,0) -- (.4,.5);
\node at (.7,.5) {\tiny $y$};
\draw[->] (0,0) -- (-.3,1); 
\node at (-0.6,0.8) {\tiny $z$};
\end{tikzpicture} space.
The D3 branes are lines parallel with the $x$ axis, the D5 branes are parallel with the $z$ axis, and the NS5 branes are lines parallel with the $y$ axis. 
In fact, in superstring theory, branes are not 1-dimensional lines, but higher dimensional mem{\em branes}, and the three different types of branes are parallel to complimentary, say orthogonal, subspaces (though we will not need the higher dimensions in this paper).

We will also regard the brane diagram as a diagram of segments in the plane. The union of the D3 branes is the $x$-axis, and the 5-branes indicate permitted positions where some other curves may intersect it. The slope of a 5-brane indicates the slope of the curve that is permitted to intersect there (see Figure \ref{YiyansButterfly} for an example).

\subsection{Brane charge, margin vectors} \label{sec:margin}
A brane diagram determines a non-negative integer for each 5-brane, the {\em brane charge} of that 5-brane. 
\begin{definition}
For an NS5 brane $V$ let
\[ 
\charge(V)=
(\mult_{V^+}-\mult_{V^-})+\#\{\text{D5 branes left of $V$}\}.
\]
For a D5 brane $U$ let
\[ 
\charge(U)=
(\mult_{U^-}-\mult_{U^+})+\#\{\text{NS5 branes right of $U$}\}.
\]
The integer vector of NS5 brane charges (from left to right) is denoted by $\rr$ (``row'') and the integer vector of D5 brane charges (from left to right) is denoted by $\cc$ (``column'').
We call $\rr$ and $\cc$ the margin vectors and we associate components of $\rr$ and $\cc$ to the rows and columns of a matrix of size $m\times n$---where $m$ and $n$ are the numbers of NS5 and D5 branes. 
\end{definition}

\begin{remark} \rm In string theory the brane charge is defined as an integral, and is calculated to be 
$(\mult_{W^+}-\mult_{W^-})+\#\{\text{opposite type 5-branes left of $W$}\}$ 
for both types of 5-branes. Hence, for D5 branes our definition (which will be convenient for us for many reasons, the first is Lemma~\ref{lem:marginsum} below) is not identical with the physics definition, rather it is a simple linear function of it. In physics this can be interpreted as integrating on a different cycle.  
\end{remark}  

For the brane diagram in Section \ref{sec:diagram} we have $r=(2,1,1,2,3,2)$ and $c=(5,2,2,0,2)$ and hence we obtain the `margins' associated with a $6\times 5$ matrix in Figure \ref{fig:margins}(a). 
\begin{figure}
\begin{tikzpicture}[baseline=1.4 cm, scale=.4]
\draw[ultra thin] (0,0) -- (5,0);
\draw[ultra thin]  (0,1) -- (5,1);
\draw[ultra thin]  (0,2) -- (5,2);
\draw[ultra thin]  (0,3) -- (5,3);
\draw[ultra thin]  (0,4) -- (5,4);
\draw[ultra thin]  (0,5) -- (5,5);
\draw[ultra thin]  (0,6) -- (5,6);
\draw[ultra thin]  (0,0) -- (0,6);
\draw[ultra thin]  (1,0) -- (1,6);
\draw[ultra thin]  (2,0) -- (2,6);
\draw[ultra thin]  (3,0) -- (3,6);
\draw[ultra thin]  (4,0) -- (4,6);
\draw[ultra thin]  (5,0) -- (5,6);
\draw[ultra thick] (0,6) -- (0,5) -- (1,5) -- (1,4) -- (2,4) -- (2,1) -- (3,1) -- (3,0) -- (5,0);
\node at (.5,6.4) {\tiny $5$}; \node at (1.5,6.4) {\tiny $2$}; \node at (2.5,6.4) {\tiny $2$}; \node at (3.5,6.4) {\tiny $0$}; \node at (4.5,6.4) {\tiny $2$};
\node at (.5,7.4) {\tiny $U_1$}; \node at (1.5,7.4) {\tiny $U_2$}; \node at (2.5,7.4) {\tiny $U_3$}; \node at (3.5,7.4) {\tiny $U_4$}; \node at (4.5,7.4) {\tiny $U_5$};
\node at (-.5,0.5) {\tiny $2$}; \node at (-.5,1.5) {\tiny $3$}; \node at (-.5,2.5) {\tiny $2$}; \node at (-.5,3.5) {\tiny $1$}; \node at (-.5,4.5) {\tiny $1$}; \node at (-.5,5.5) {\tiny $2$};  
\node at (-1.8,0.5) {\tiny $V_6$}; \node at (-1.8,1.5) {\tiny $V_5$}; \node at (-1.8,2.5) {\tiny $V_4$}; \node at (-1.8,3.5) {\tiny $V_3$}; \node at (-1.8,4.5) {\tiny $V_2$}; \node at (-1.8,5.5) {\tiny $V_1$};  
\node at (2.5,-1.5) {(a)};
\end{tikzpicture}
\qquad
\begin{tikzpicture}[baseline=1.1 cm, scale=.4]
\draw[ultra thin] (0,0) -- (0,5);
\draw[ultra thin]  (1,0) -- (1,5);
\draw[ultra thin]  (2,0) -- (2,5);
\draw[ultra thin]  (3,0) -- (3,5);
\draw[ultra thin]  (4,0) -- (4,5);
\draw[ultra thin]  (5,0) -- (5,5);
\draw[ultra thin]  (6,0) -- (6,5);
\draw[ultra thin]  (0,0) -- (6,0);
\draw[ultra thin]  (0,1) -- (6,1);
\draw[ultra thin]  (0,2) -- (6,2);
\draw[ultra thin]  (0,3) -- (6,3);
\draw[ultra thin]  (0,4) -- (6,4);
\draw[ultra thin]  (0,5) -- (6,5);
\draw[ultra thick] (0,5) -- (1,5) -- (1,4) -- (2,4) -- (2,3) -- (5,3) -- (5,2) -- (6,2) -- (6,0);
\node at (-.5,.5) {\tiny $4$}; \node at (-.5,1.5) {\tiny $6$}; \node at (-.5,2.5) {\tiny $4$}; \node at (-.5,3.5) {\tiny $4$}; \node at (-.5,4.5) {\tiny $1$};
\node at (-1.8,.5) {\tiny $V'_5$}; \node at (-1.8,1.5) {\tiny $V'_4$}; \node at (-1.8,2.5) {\tiny $V'_3$}; \node at (-1.8,3.5) {\tiny $V'_2$}; \node at (-1.8,4.5) {\tiny $V'_1$};
\node at (0.5,5.4) {\tiny $3$}; \node at (1.5,5.4) {\tiny $4$}; \node at (2.5,5.4) {\tiny $4$}; \node at (3.5,5.4) {\tiny $3$}; \node at (4.5,5.4) {\tiny $2$}; \node at (5.5,5.4) {\tiny $3$}; 
\node at (0.5,6.4) {\tiny $U'_1$}; \node at (1.5,6.4) {\tiny $U'_2$}; \node at (2.5,6.4) {\tiny $U'_3$}; \node at (3.5,6.4) {\tiny $U'_4$}; \node at (4.5,6.4) {\tiny $U'_5$}; \node at (5.5,6.4) {\tiny $U'_6$};
\node at (2.5,-2.2) {(b)};
\end{tikzpicture}
\qquad
\begin{tikzpicture}[baseline=1.4 cm, scale=.4]
\draw[ultra thin] (0,0) -- (5,0);
\draw[ultra thin]  (0,1) -- (5,1);
\draw[ultra thin]  (0,2) -- (5,2);
\draw[ultra thin]  (0,3) -- (5,3);
\draw[ultra thin]  (0,4) -- (5,4);
\draw[ultra thin]  (0,5) -- (5,5);
\draw[ultra thin]  (0,6) -- (5,6);
\draw[ultra thin]  (0,0) -- (0,6);
\draw[ultra thin]  (1,0) -- (1,6);
\draw[ultra thin]  (2,0) -- (2,6);
\draw[ultra thin]  (3,0) -- (3,6);
\draw[ultra thin]  (4,0) -- (4,6);
\draw[ultra thin]  (5,0) -- (5,6);
\draw[ultra thick] (0,6) -- (0,5) -- (1,5) -- (1,4) -- (2,4) -- (2,2) -- (3,2) -- (3,0) -- (5,0);
\node at (.5,6.4) {\tiny $5$}; \node at (1.5,6.4) {\tiny $2$}; \node at (2.5,6.4) {\tiny $2$}; \node at (3.5,6.4) {\tiny $0$}; \node at (4.5,6.4) {\tiny $2$};
\node at (.5,7.4) {\tiny $U_1$}; \node at (1.5,7.4) {\tiny $U_2$}; \node at (2.5,7.4) {\tiny $U_3$}; \node at (3.5,7.4) {\tiny $U_4$}; \node at (4.5,7.4) {\tiny $U_5$};
\node at (-.5,0.5) {\tiny $2$}; \node at (-.5,1.5) {\tiny $3$}; \node at (-.5,2.5) {\tiny $2$}; \node at (-.5,3.5) {\tiny $1$}; \node at (-.5,4.5) {\tiny $1$}; \node at (-.5,5.5) {\tiny $2$};  
\node at (-1.8,0.5) {\tiny $V_6$}; \node at (-1.8,1.5) {\tiny $V_5$}; \node at (-1.8,2.5) {\tiny $V_4$}; \node at (-1.8,3.5) {\tiny $V_3$}; \node at (-1.8,4.5) {\tiny $V_2$}; \node at (-1.8,5.5) {\tiny $V_1$};  
\draw[->,  thin, red] (2.1,1.1) -- (2.8,1.8); \draw[->,  thin, red] (2.4,1.1) -- (2.8,1.5); \draw[->,  thin, red] (2.1,1.4) -- (2.5,1.8); 
\node at (2.5,-1.5) {(c)};
\end{tikzpicture}
\qquad
\begin{tikzpicture}[baseline=1.4 cm, scale=.4]
\draw[ultra thin] (0,0) -- (5,0);
\draw[ultra thin]  (0,1) -- (5,1);
\draw[ultra thin]  (0,2) -- (5,2);
\draw[ultra thin]  (0,3) -- (5,3);
\draw[ultra thin]  (0,4) -- (5,4);
\draw[ultra thin]  (0,5) -- (5,5);
\draw[ultra thin]  (0,6) -- (5,6);
\draw[ultra thin]  (0,0) -- (0,6);
\draw[ultra thin]  (1,0) -- (1,6);
\draw[ultra thin]  (2,0) -- (2,6);
\draw[ultra thin]  (3,0) -- (3,6);
\draw[ultra thin]  (4,0) -- (4,6);
\draw[ultra thin]  (5,0) -- (5,6);
\draw[ultra thick] (0,6) -- (0,5) -- (1,5) -- (1,4) -- (2,4) -- (2,1) -- (3,1) -- (3,0) -- (5,0);
\node at (.5,6.4) {\tiny $5$}; \node at (1.5,6.4) {\tiny $2$}; \node at (2.5,6.4) {\tiny $2$}; \node at (3.5,6.4) {\tiny $0$}; \node at (4.5,6.4) {\tiny $2$};
\node at (.5,7.4) {\tiny $U_1$}; \node at (1.5,7.4) {\tiny $U_2$}; \node at (2.5,7.4) {\tiny $U_3$}; \node at (3.5,7.4) {\tiny $U_4$}; \node at (4.5,7.4) {\tiny $U_5$};
\node at (-.5,0.5) {\tiny $2$}; \node at (-.5,1.5) {\tiny $3$}; \node at (-.5,2.5) {\tiny $2$}; \node at (-.5,3.5) {\tiny $1$}; \node at (-.5,4.5) {\tiny $1$}; \node at (-.5,5.5) {\tiny $2$};  
\node at (-1.8,0.5) {\tiny $V_6$}; \node at (-1.8,1.5) {\tiny $V_5$}; \node at (-1.8,2.5) {\tiny $V_4$}; \node at (-1.8,3.5) {\tiny $V_3$}; \node at (-1.8,4.5) {\tiny $V_2$}; \node at (-1.8,5.5) {\tiny $V_1$};  
\node[violet] at (0.5,0.5) {\tiny $1$};\node[violet] at (1.5,0.5) {\tiny $0$};\node[violet] at (2.5,0.5) {\tiny $0$};\node[violet] at (3.5,0.5) {\tiny $0$};\node[violet] at (4.5,0.5) {\tiny $1$};
\node[violet] at (0.5,1.5) {\tiny $1$};\node[violet] at (1.5,1.5) {\tiny $1$};\node[violet] at (2.5,1.5) {\tiny $0$};\node[violet] at (3.5,1.5) {\tiny $0$};\node[violet] at (4.5,1.5) {\tiny $1$};
\node[violet] at (0.5,2.5) {\tiny $1$};\node[violet] at (1.5,2.5) {\tiny $0$};\node[violet] at (2.5,2.5) {\tiny $1$};\node[violet] at (3.5,2.5) {\tiny $0$};\node[violet] at (4.5,2.5) {\tiny $0$};
\node[violet] at (0.5,3.5) {\tiny $0$};\node[violet] at (1.5,3.5) {\tiny $0$};\node[violet] at (2.5,3.5) {\tiny $1$};\node[violet] at (3.5,3.5) {\tiny $0$};\node[violet] at (4.5,3.5) {\tiny $0$};
\node[violet] at (0.5,4.5) {\tiny $1$};\node[violet] at (1.5,4.5) {\tiny $0$};\node[violet] at (2.5,4.5) {\tiny $0$};\node[violet] at (3.5,4.5) {\tiny $0$};\node[violet] at (4.5,4.5) {\tiny $0$};
\node[violet] at (0.5,5.5) {\tiny $1$};\node[violet] at (1.5,5.5) {\tiny $1$};\node[violet] at (2.5,5.5) {\tiny $0$};\node[violet] at (3.5,5.5) {\tiny $0$};\node[violet] at (4.5,5.5) {\tiny $0$};
\node at (2.5,-1.5) {(d)};
\end{tikzpicture}
\caption{} \label{fig:margins}
\end{figure}
We added one more decoration to the shape of the matrix: a down-and-right monotone {\em ``separating line''} connecting the NW corner with the SE corner defined as follows: the $(V_i,U_j)$ box of the matrix is to the right (equivalently, above) of the separating line if and only if the brane $V_i$ is left of the brane $U_j$ in the brane diagram. 
Note that there are $\binom{m+n}{m}$ possible separating lines. 

\begin{lemma} \label{lem:marginsum} We have
\begin{itemize}
\item $\sum r_i=\sum c_j$ for any brane diagram.
\item The combinatorial data consisting of $\rr, \cc$ (with $\sum r_i=\sum c_j$) and the separating line uniquely determine the brane diagram. 
\end{itemize}
\end{lemma}

\begin{proof} The first statement follows from the definition of charges. To prove the second statement notice that the separating line sets the order of 5-branes. Then reading the charges from $r$ and $c$ we can fill the multiplicities of the D3 branes one by one from left to right. The $\sum r_i=\sum c_j$ condition guarantees that at the right end of the diagram we arrive at multiplicity 0.
\end{proof}

The size of the matrix, the margin vectors, and the separating line together will be called the {\em table-with-margins} code of the brane diagram.

%\begin{remark} \rm
As our vocabulary ``margin vectors'' suggests we will soon consider matrices with the given row and column sums $r, c$. In fact we will consider 01-matrices only, for example the one in Figure~\ref{fig:margins}(d). The D3 brane multiplicities can be read from such a matrix the following way: the D3 branes correspond to the integer points of the separating line (in the order from top down), and the multiplicity corresponding to such an integer point $P$ is $\#\{\text{1's NE of P}\}+\#\{\text{0's SW of P}\}$. Interestingly, this holds for {\em any} 01-matrix with row and column sums $r$, $c$. (For the margins in Figure \ref{fig:margins}(a) there are 123 such matrices, one of them is Figure \ref{fig:margins}(d).) This fact is an alternative proof for the second part of Lemma \ref{lem:marginsum}---in the case when a matrix with row and column sums  $r, c$ {\em exists}. However, we will restrict our attention to those brane diagrams: 
%\end{remark}

\begin{assumption} \label{assume}
In the whole paper we will assume that part of the definition of {\em brane diagram} is that there exists at least one 01-matrix whose row and column sums are $r$ and $c$. % of its associated table-with-margins. 
\end{assumption}

\noindent The geometric significance of this assumpion will be given in Proposition \ref{prop:existsfix}.

\subsection{3d mirror symmetry of brane diagrams} \label{sec:3d}
Replacing NS5-branes with D5-branes, and D5-branes with NS5-branes in a brane diagram we get its 3d mirror dual brane diagram (referring to the mirror symmetry in $N=4$, $d=3$ superstring theory). 

In the three dimensional view of the brane diagram (of Section \ref{sec:3dimview}) the mirror is obtained by rotating the diagram around the $x$ axis by $90^\circ$ (or, reflection about the $y=z$ plane.) In the two dimensional view of the brane diagram (of Section \ref{sec:3dimview}) the 3d mirror is obtained by reflecting the diagram across the $x$-axis.%(and re-coloring).

The change of the table-with-margins combinatorial code under 3d mirror symmetry is also natural: The table gets transposed, together with the separating line, and the new margin vectors $r', c '$ are obtained from the old ones $r, c$ by
\begin{equation}\label{eq:rc}
c_i'=n-r_i, \qquad r_i'=m-c_i.
\end{equation}
The verification of these statements is straightforward. Figure~\ref{fig:margins}(b) is the table-with-margins associated with the 3d mirror dual of the brane diagram of Section \ref{sec:diagram}.

\subsection{Hanany-Witten transition}\label{sec:HW}

Assume $\mult_2+\tilde{\mult}_2=\mult_1+\mult_3+1$. Carrying out the local change 
\begin{equation}\label{fig:HW}
\begin{tikzpicture}[baseline=(current  bounding  box.center), scale=.5]
\draw[thick] (0,1)--(1,1) node [above] {$\mult_1$} -- (2,1);
\draw[thick,blue] (3,0)--(2,2);
\draw[thick] (3,1)--(4,1) node [above] {$\mult_2$}--(5,1);
\draw[thick,red] (5,0)--(6,2);
\draw[thick] (6,1)--(7,1) node [above] {$\mult_3$}--(8,1);
\draw[ultra thick, <->] (10,1)--(11.5,1) node[above]{HW} -- (13,1);
\draw[thick] (15,1)--(16,1) node [above] {$\mult_1$} -- (17,1);
\draw[thick,red] (17,0)--(18,2);
\draw[thick] (18,1)--(19,1) node [above] {$\tilde{\mult}_2$}--(20,1);
\draw[thick,blue] (21,0)--(20,2);
\draw[thick] (21,1)--(22,1) node [above] {$\mult_3$}--(23,1);
\end{tikzpicture}
\end{equation}
(in either direction) in a brane diagram is called a {\em Hanany-Witten transition} \cite{HW}, \cite[\S7]{NT}. 

The charge of each 5-brane remains the same before and after a Hanany-Witten (HW) transition. Indeed, assume there are $k$ NS5 branes to the right of the portion of the diagrams displayed. Then the charge of the D5 brane on the left is $\mult_1-\mult_2+(k+1)$, and the charge of the D5 brane on the right is $\tilde{\mult}_2-\mult_3+k$. Our assumption $\mult_2+\tilde{\mult}_2=\mult_1+\mult_3+1$ implies that these two expressions are equal. The calculation for the NS5 brane is analogous. In fact, the condition $\mult_2+\tilde{\mult}_2=\mult_1+\mult_3+1$ is set so that the brane charges are invariant under HW transition.

Let us see how a table-with-margins changes under Hanany-Witten transition. The number of NS5 and D5 branes, or their order do not change, and their charges do not change either. What changes is that the separating line will transition from one side of the $(V_i, U_j)$ box to its other side. Figure~\ref{fig:margins}(c) corresponds to 
\ttt{{\fs}2\bs 2{\fs}2\bs 4{\fs}3{\fs}3\bs 3{\fs}3{\fs}2\bs 2\bs},  
which is obtained from the brane diagram of Section \ref{sec:diagram} by a Hanany-Witten transition switching $V_5$ and $U_3$. (The reader may verify that the new D3 brane multiplicity between these two branes involved is indeed $(3+3+1)-4=3$.)

\begin{lemma} \label{lem:noneg}
For any consecutive pair of different 5-branes the HW transition can be carried out, that is, the $\tilde{\mult}_2=\mult_1+\mult_3+1-\mult_2$ assignment does not violate $\tilde{d}_2\geq 0$.
\end{lemma}

One of the reasons we made Assumption \ref{assume} is to make this lemma true. The reader may construct a simple combinatorial proof at this point, but we prefer to prove the lemma by just one picture in Section \ref{sec:HWfix}. 

\medskip

Given a brane diagram with $m$ NS5 branes and $n$ D5 branes one can carry out Hanany-Witten moves to re-order them to any of their $\binom{m+n}{m}$ orders (recall that we do not allow switching the same type of 5-branes). The codes of these $\binom{m+n}{m}$ brane diagrams of this Hanany-Witten equivalence class correspond to the $\binom{m+n}{m}$ ways of drawing the separating line that we mentioned in Section~\ref{sec:3dimview}.

\begin{remark}\label{rem:OtherTransitions}
It is tempting to consider the analogous transitions interchanging two 5-branes of the same type in such a way that their brane charges do not change, that is,  
\begin{equation*}
\begin{tikzpicture}[baseline=(current  bounding  box.center), scale=.35]
\draw[thick] (0,1)--(1,1) node [above] {$d_1$} -- (2,1);
\draw[thick,blue] (3,0)--(2,2);
\draw[thick] (3,1)--(4,1) node [above] {$d_2$}--(5,1);
\draw[thick,blue] (6,0)--(5,2);
\draw[thick] (6,1)--(7,1) node [above] {$d_3$}--(8,1);
\draw[ultra thick, <->] (10,1)--(11.5,1) node[above]{} -- (13,1);
\draw[thick] (15,1)--(16,1) node [above] {$d_1$} -- (17,1);
\draw[thick,blue] (18,0)--(17,2);
\draw[thick] (18,1)--(19,1) node [above] {$\tilde{d}_2$}--(20,1);
\draw[thick,blue] (21,0)--(20,2);
\draw[thick] (21,1)--(22,1) node [above] {$d_3$}--(23,1);
\begin{scope}[yshift=-3cm]
\draw[thick] (0,1)--(1,1) node [above] {$d_1$} -- (2,1);
\draw[thick,red] (2,0)--(3,2);
\draw[thick] (3,1)--(4,1) node [above] {$d_2$}--(5,1);
\draw[thick,red] (5,0)--(6,2);
\draw[thick] (6,1)--(7,1) node [above] {$d_3$}--(8,1);
\draw[ultra thick, <->] (10,1)--(11.5,1) node[above]{} -- (13,1);
\draw[thick] (15,1)--(16,1) node [above] {$d_1$} -- (17,1);
\draw[thick,red] (17,0)--(18,2);
\draw[thick] (18,1)--(19,1) node [above] {$\tilde{d}_2$}--(20,1);
\draw[thick,red] (20,0)--(21,2);
\draw[thick] (21,1)--(22,1) node [above] {$d_3$}--(23,1);
\end{scope}
\end{tikzpicture}
\qquad\quad \text{for } d_2+\tilde{d}_2=d_1+d_3.
\end{equation*}
In fact, these transitions make sense and are important in physics, but not for our purposes, until Section \ref{sec:OtherTransitions}.  
\end{remark}

\section{Bow varieties}
Associated with a brane diagram $\DD$ there is a smooth holomorphic symplectic variety $\Ch(\DD)$, called the Cherkis bow variety. It is equipped with rich structures, such as a torus action, tautological complex vector bundles, holomorphic symplectic form.

There are some equivalent definitions of $\Ch(\DD)$. In the original definition \cite{Ch09, Ch10, Ch11} \cite[Sec 2.1]{NT} $\Ch(\DD)$ is the moduli spaces of $U(n)$-instantons on multi-Taub-NUT spaces. That is, points of $\Ch(\DD)$ are gauge equivalence classes of so-called bow-solutions, tuples of linear maps and a Hermitian connection satisfying some constraints (in particular, one called Nahm's equation). Another definition is through quivers \cite[Sec 2.2]{NT}.

\subsection{The Nakajima-Takayama description of bow varieties} \label{sec:bowdef}
In this section we repeat the quiver-based construction of Nakajima-Takayama \cite[Section 2.2]{NT} for bow varieties, with the only novelty of introducing an extra $\C^{\times}$-action that will be convenient for us later. Readers familiar with \cite{NT} can use this picture 
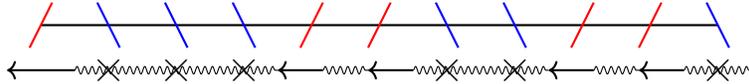
\[
\begin{tikzpicture}[scale=.3]
\draw[thick] (0,1)--(30,1) ;
\draw[thick,red] (-.5,0)--(.5,2);
\draw[thick,blue] (3.5,0)--(2.5,2);
\draw[thick,blue] (6.5,0)--(5.5,2);
\draw[thick,blue] (9.5,0)--(8.5,2);
\draw[thick,red] (11.5,0)--(12.5,2);
\draw[thick,red] (14.5,0)--(15.5,2);
\draw[thick,blue] (18.5,0)--(17.5,2);
\draw[thick,blue] (21.5,0)--(20.5,2);
\draw[thick,red] (23.5,0)--(24.5,2);
\draw[thick,red] (26.5,0)--(27.5,2);
\draw[thick,blue] (30.5,0)--(29.5,2);

\draw[<-, thick] (-1.5,-1) -- (1.5,-1);
\draw[decorate, decoration={snake, segment length=1mm, amplitude=.5mm}] (1.5,-1) -- (10.5,-1);
\draw[<-, thick] (10.5,-1) -- (12.5,-1);
\draw[decorate, decoration={snake, segment length=1mm, amplitude=.5mm}] (12.5,-1) -- (14.5,-1);
\draw[<-, thick] (14.5,-1) -- (16.5,-1);
\draw[decorate, decoration={snake, segment length=1mm, amplitude=.5mm}] (16.5,-1) -- (22.5,-1);
\draw[<-,thick] (22.5,-1) -- (24.5,-1);
\draw[decorate, decoration={snake, segment length=1mm, amplitude=.5mm}] (24.5,-1) -- (26.5,-1);
\draw[<-,thick] (26.5,-1) -- (28.5,-1);
\draw[decorate, decoration={snake, segment length=1mm, amplitude=.5mm}] (28.5,-1) -- (31.5,-1);

\node at (3,-1) {$\bigtimes$};\node at (6,-1) {$\bigtimes$};\node at (9,-1) {$\bigtimes$};
\node at (18,-1) {$\bigtimes$};\node at (21,-1) {$\bigtimes$};\node at (30,-1) {$\bigtimes$};

\end{tikzpicture}
\]
for comparison between diagrams of this paper (top) and diagrams of \cite{NT} (bottom). We will not use the diagram in the bottom any further.

Let $\DD$ be a brane diagram. For each D3 brane $X$ consider a vector space $W_X$ of dimension $\mult_{X}$. Also, for each D5 brane $U$ we fix a one-dimensional vector space $\C_U$. 

We are going to define some vector spaces associated with branes and brane diagrams. These vector spaces in fact will be representations of $\C^{\times}$, which we denote by $\C^{\times}_{\h}$. For a vector space $W$ we write $\hb W$ if $\C^{\times}_{\h}$ acts on it by multiplication (dilation), or $\hb^kW$ if $\C^{\times}_{\h}$ acts through the $k$'th power.

\begin{itemize}
\item For a D5 brane $U$ define the ``three-way part''
\begin{align*}
\MM_U= &\Hom(W_{U^+},W_{U^-}) \oplus
\hb\Hom(W_{U^+},\C_U) \oplus \Hom(\C_U,W_{U^-}) \\
&  \oplus\hb\End(W_{U^-}) \oplus \hb\End(W_{U^+}),
\end{align*} 
whose elements will  be denoted by $(A_U, b_U, a_U, B_U, B'_U)$, see the diagram
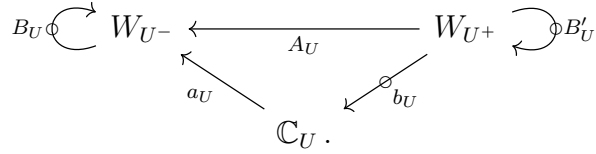
\[
\begin{tikzcd}
W_{U^-} \arrow[loop,out=200,in=160,distance=2em, "B_U", "\circ" marking] & & W_{U^+} \arrow[ld, "\circ" marking, "b_U" ] \arrow[ll, "A_U"] \arrow[loop,out=20,in=-20,distance=2em, "B'_U", "\circ" marking] \\
& \C_U.\arrow[ul, "a_U"] & 
\end{tikzcd}
\]
\noindent The $\circ$ marking on an arrow refers to the $\C^\times_{\h}$ action.
\item For an NS5 brane $V$ define the ``two-way part''
\[
\MM_V= \hb\Hom(W_{V^+},W_{V^-}) \oplus
\Hom(W_{V^-},W_{V^+}),
\]
whose elements will  be denoted by $(C_V, D_V)$, see the diagram
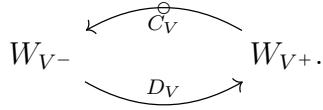
\[
\begin{tikzcd}
W_{V^-} \ar[rr,bend right, "D_V"] & & W_{V^+}. \ar[ll, "C_V", bend right, "\circ" marking] 
\end{tikzcd}
\]
\item For a D5 brane $U$ define 
\[
\NN_U=\hb\Hom(W_{U^+},W_{U^-}), \qquad \text{illustrated by }
\begin{tikzcd}
W_{U^-} & W_{U^+}. \ar[l,  "\circ" marking] 
\end{tikzcd}
\]
\item For a D3 brane $X$ define
\[
\NN_X=\hb\End(W_X), \qquad \text{illustrated by }
\begin{tikzcd}
W_X \ar[loop,out=200,in=160,distance=1.6em, "\circ" marking]
\end{tikzcd}.
\]
\end{itemize}

\noindent For the brane diagram let
\[
\MM=\bigoplus_{U \ \text{D5}} \MM_U \ \oplus\  \bigoplus_{V \ \text{NS5}} \MM_V,
\qquad\qquad
\NN=\bigoplus_{U \ \text{D5}} \NN_U \ \oplus \bigoplus_{X \  \text{D3}} \NN_X.
\]
We define the moment map $\mu:\MM\to\NN$ component-wise as follows. 
\begin{itemize}
\item 
The $\NN_U$-component ($U$ is D5) of $\mu$ is 
$B_U A_U -A_U B_U+a_U b_U$.
\item
The $\NN_X$-components ($X$ is D3) of $\mu$ depend on the diagram:
\begin{itemize}
\item[\ttt{\bs -\bs}] If $X$ is in between two D5 branes then it is $B'_{X^-}-B_{X^+}$. 
\item[\ttt{{\fs}-{\fs}}] If $X$ is in between two NS5 branes then it is $C_{X^+}D_{X^+}-D_{X^-}C_{X^-}$.
\item[\ttt{{\fs}-\bs}] If $X^-$ is an NS5 brane and $X^+$ is a D5 brane then it is $-D_{X^-}C_{X^-}-B_{X^+}$.
\item[\ttt{\bs -{\fs}}] If $X^-$ is a  D5 brane and $X^-$ is an NS5 brane then it is $C_{X^+}D_{X^+}+B'_{X^-}$.
\end{itemize}
\end{itemize}

\begin{remark}\label{nuC}
In \cite{NT} more general bow varieties are considered, with parameters $\nu^{\C}$ (associated with segments of the brane diagram separated by NS5 branes). For those general bow varieties the second and forth line above would read $C_{X^+}D_{X^+}-D_{X^-}C_{X^-}-\nu^{\C}$, and $C_{X^+}D_{X^+}+B'_{X^-}-\nu^{\C}$ respectively. We choose these $\nu^{\C}$ parameters to be 0 in the whole paper, so that we have $\C^{\times}_{\h}$-equivariance of $\mu$.  
\end{remark}

\begin{lemma} 
The map $\mu:\MM \to \NN$ is $\C^{\times}_{\h}$-equivariant.
\end{lemma}

\begin{proof} Straightforward calculation: e.g. for the $\NN_X$-component in the \ttt{\bs -{\fs}} case we have (with economical notation)
\[
\mu\left( \hb.(C,D,B') \right) =
\mu\left( (\hb C, D, \hb B') \right) =
\hb CD+\hb B' =
\hb(CD+B') =
\hb.\left(\mu(C,D,B')\right).
\]
\end{proof}

\noindent Let $\widetilde{\M}$ consist of points of $\mu^{-1}(0)$ for which the stability conditions  
\begin{itemize}
\item[(S1)]  If $S\leq W_{U^+}$ is a subspace for a D5 brane $U$ with $B'_U(S)\subset S$, $A_U(S)=0$, $b_U(S)=0$ then $S=0$. 
\item[(S2)]  If $T \leq W_{U^-}$ is a subspace for a D5 brane $U$ with $B_U(T)\subset T$, $\Imm(A_U)+\Imm(a_U)\subset T$ then $T=W_{U^-}$.
\end{itemize}
hold. Let $\G=\prod_{X \text{ D3}} \GL(W_X)$ and consider the character 
\begin{equation}\label{eq:character}
\chi: \G \to \C^{\times}, 
\qquad\qquad
(g_X)_X \mapsto \prod_{X'} \det(g_{X'})
\end{equation}
where the product runs for D3 branes $X'$ such that ${X'}^-$ is an NS5 brane (in picture \ttt{{\fs}X'}). Let $\G$ act on $\widetilde{\M} \times \C$ by $g.(m,z)=(gm,\chi^{-1}(g)z)$. We say that $m\in \widetilde{\M}$ is stable (notation $m\in \widetilde{\M}^s$) if the orbit $\G(m,z)$ is closed and the stabilizer of $(m,z)$ is finite for $z\not=0$. Define the bow variety $\Ch(\DD)$ associated with the brane diagram  $\DD$ to be $\widetilde{\M}^s/\G$. By this definition $\Ch(\DD)$ is an orbifold, but in fact the stabilizers of stable points are trivial \cite[Lemma 2.10]{NT}, hence $\Ch(\DD)$ is a smooth variety.

\begin{remark} \label{nuR}
Again, in \cite{NT} more general bow varieties are considered, with parameters $\nu^{\R}$ (associated with segments of the brane diagram separated by NS5 branes, just like the $\nu^{\C}$ parameters from Remark \ref{nuC}). For those bow varieties the character $\chi$ in \eqref{eq:character} is $\prod_{X'} \det(g_{X'})^{-\nu^{\R}}$. For the purposes of this paper we chose all $\nu^{\R}=-1$.
\end{remark}

The smooth varieties $\Ch(\DD)$ are the main objects of this paper. They come with the following structures:
\begin{description}
\item[tautological bundles] For D3 branes $X$ the vector spaces $W_X$ induce ``tautological'' vector bundles $\xi_X$ of rank $\mult_{X}$ over $\Ch(\DD)$.
\item[torus action] For D5 branes $U$ the reparametrizations of the lines $\C_U$ induce an action of the torus $\A=(\C^{\times})^{\{\text{D5 branes}\}}$ on $\Ch(\DD)$.
\item[extra $\C^{\times}$ action] The action of $\C^{\times}_{\h}$ on $\MM$ descends to an action on $\Ch(\DD)$. Define the torus $\T=\A\times \C^{\times}_{\h}$.
\item[holomorphic symplectic structure] $\Ch(\DD)$ is equipped with a holomorphic symplectic form \cite[Section 5]{NT} preserved by $\A$ and scaled by $\C^\times_{\h}$.
\end{description}

The listed structures are also present for cotangent bundles of (type A) homogeneous spaces, and more generally for (type A) Nakajima quiver varieties. In fact, for those spaces the listed structures (together with the combinatorial description of the tangent bundle, fixed points, and cohomological fixed point restrictions) are the basis of their algebraic combinatorial study, as well as the study of their characteristic classes. 

\subsection{Tangent bundle in terms of tautological bundles}\label{sec:tangent}

Recall that tautological bundles $\xi_X$ of rank $\mult_{X}$ are associated with the D3 branes. For a bundle $\xi$ we will use the notation $\hb \xi$ to denote the bundle with $\C^\times_{\h}$ acting by multiplication in the fibers. 

\begin{definition} (Cf. three-way and two-way parts in Section \ref{sec:bowdef}.)
\begin{itemize}
\item For a D5 brane $U$ define  
\begin{align*}
T_U =& 
\Hom(\xi_{U^+},\xi_{U^-})\oplus 
\hb\Hom(\xi_{U^+},\C_U) \oplus
\Hom(\C_U,\xi_{U^-}) 
\\
& \oplus \hb \End(\xi_{U^-})\oplus \hb\End(\xi_{U^+}).
\end{align*}
\item
For an NS5 brane $V$ define  
$
T_V =\hb\Hom(\xi_{V^+},\xi_{V^-}) \oplus \Hom(\xi_{V^-},\xi_{V^+})
$.
\end{itemize}
\end{definition}
From the definition of bow varieties we obtain that the tangent bundle of $\Ch(\DD)$ %(when it is not empty) 
is expressed in terms of the tautological bundles as 
\begin{align*}
T\Ch(\DD)=&
\left(\bigoplus_{U \text{ D5}} T_U \right)\oplus 
\left( \bigoplus_{V \text{ NS5}} T_V \right) \\
& 
\ominus \left( \bigoplus_{U \text{ D5}} \hb \Hom(\xi_{U^+},\xi_{U^-})  \right)
\ominus \left(\bigoplus_{X \text{ D3}} (1+\hb)\End(\xi_{X}) \right).
\end{align*}
As a consequence, we have a formula for the dimension of $\Ch(\DD)$ 
\begin{align} \label{eq:dimensionformula}
\dim(\Ch(\DD))=&
\sum_{U\text{ D5}} \left( (\mult_{U^-} +1)\mult_{U^-} + (\mult_{U^+}+1)\mult_{U^+} \right)
\\
& 
+ 2\sum_{V\text{ NS5}}  \mult_{V^+} \mult_{V^-} 
- 2\sum_{X\text{ D3}} \mult_{X}^2. \notag
\end{align}
As expected, this expression is even for any dimension vector.

\begin{remark} \label{rem:negatives}
The appearance of $\ominus$ in the formula for the tangent bundle in terms of the tautological bundles has the following classical analogue. Let $S$ and $Q$ be the tautological sub- and quotient bundles on a Grassmannian. Then the tangent bundle is $\Hom(S,Q)$ which expression has no negative sign. However, if $Q$ is not among our ``permitted'' tautological bundles then we need to express the tangent bundle as $\Hom(S,\C^n \ominus  S)=\Hom(S,\C^n) \ominus \End(S)$. 
\end{remark}

\begin{example}\rm \label{ex:P1a} 
Let the three tautological line bundles over $\Ch( \ttt{{\fs}1\bs 1\bs 1{\fs}} )$
be called $\alpha,\beta,\gamma$ (from left to right). Then the tangent bundle in $K_{\T}(\Ch(\DD))$ is 
\begin{equation}\label{eq:TP1}
(1-\hb)\left( \frac{\alpha}{\beta}+\frac{\beta}{\gamma}\right)+\frac{\alpha}{\uu_1}+\frac{\hb \uu_1}{\beta} +\frac{\beta}{\uu_2}+\frac{\hb \uu_2}{\gamma}+\hb-3,
\end{equation}
where $\uu_1, \uu_2$ are the line bundles induced by $\C_{U_1}$, $\C_{U_2}$ (wit the action of $\A$), and $\hb$ is the trivial line bundle with $\C^\times_{\h}$ acting by multiplication. We have $\dim(\Ch(\DD))=2$ (cf. Example \ref{ex:P1b}.)
\end{example}

\begin{example} \rm For the brane diagram \ttt{{\fs}a{\fs}b\bs c\bs} let us denote the tautological bundles of rank $a,b,c$ by $\alpha, \beta, \gamma$. The tangent bundle of the associated bow variety can be illustrated by 
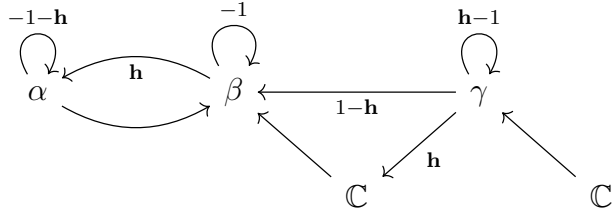
\[
\begin{tikzcd}
\alpha \ar[rr, bend right]  \arrow[loop,out=120,in=60,distance=2em, "-1-\hb"]
& & 
\beta \ar[ll, bend right, "\hb"]  \arrow[loop,out=120,in=60,distance=2em, "-1"]
& &  
\gamma \ar[ll,"1-\hb"] \ar[ld, "\hb"]  \arrow[loop,out=120,in=60,distance=2em, "\hb-1"]
\\
&&&\C \ar[ul] && \C\ar[ul]
\end{tikzcd}
\]
where each arrow represents a $\Hom$-bundle, and the decoration is its coefficient (eg. $(\hb-1)\End(\gamma)$ is one of the terms).
Hence the dimension of the variety is $2(ab+c-a^2)+b-b^2$ (One may verify combinatorially that our Assumption \ref{assume} makes this number non-negative; alternatively it will follow trivially from Proposition \ref{prop:existsfix}.) 
\end{example}

\begin{example}  \rm
The bow variety associated with the brane diagram \ttt{{\fs}2\bs 2{\fs}2\bs 4{\fs}3{\fs}3{\fs}4\bs 3{\fs}2\bs 2\bs} from Section \ref{sec:diagram} has dimension 16.
\end{example}

\subsection{Hanany-Witten transition induces isomorphism}\label{sec:HWproof}

Let $\mult_2+\tilde{\mult}_2=\mult_1+\mult_3+1$ and consider the brane diagrams $\DD$ and $\tDD$ which only differ locally as in the picture
\begin{equation}\label{fig:HW2}
\begin{tikzpicture}[baseline=(current  bounding  box.center), scale=.5]
\draw[thick] (0,1)--(1,1) node [above] {$\mult_1$} -- (2,1);
\draw[thick,blue] (3,0)--(2,2);
\draw[thick] (3,1)--(4,1) node [above] {$\mult_2$}--(5,1);
\draw[thick,red] (5,0)--(6,2);
\draw[thick] (6,1)--(7,1) node [above] {$\mult_3$}--(8,1);
\node at (2.7,-.45) {$U$}; \node at (5.3,-.45) {$V$};
\node at (4,.3) {$X$};
%\draw[ultra thick, <->] (10,1)--(11.5,1) node[above]{HW} -- (13,1);
\draw[thick] (15,1)--(16,1) node [above] {$\mult_1$} -- (17,1);
\draw[thick,red] (17,0)--(18,2);
\draw[thick] (18,1)--(19,1) node [above] {$\tilde{\mult}_2$}--(20,1);
\draw[thick,blue] (21,0)--(20,2);
\draw[thick] (21,1)--(22,1) node [above] {$\mult_3$}--(23,1);
\node at (20.7,-.4) {$\tilde{U}$}; \node at (17.3,-.45) {$\tilde{V}$};
\node at (19,.3) {$\tilde{X}$};
\end{tikzpicture}.
\end{equation}
The three tautological bundles  over $\Ch(\DD)$ and $\Ch(\tDD)$ corresponding to the pictured D3 branes will be denoted by $\xi_1, \xi_2, \xi_3$ and $\xi_1, \tilde{\xi_2}, \xi_3$. 
Recall that the tori 
\[
\T=(\C^{\times})^{\text{D5 branes in } \DD} \times \C_{\h}^\times
\qquad\qquad\text{and}\qquad\qquad
\tilde{\T}=(\C^{\times})^{\text{D5 branes in } \tDD} \times \C_{\h}^\times
\] 
act on $\Ch(\DD)$ and $\Ch(\tDD)$ respectively. Define $\rho: \T \to \tilde{\T}$ to be the identity in most components, except in the components
\begin{equation}\label{eq:reparametrize}
\C^\times_{U} \times \C^{\times}_{\h} \to \C^{\times}_{\tilde{U}} \times \C^{\times}_{\h}
\qquad\qquad\text{by}\qquad\qquad
(x,y)\mapsto (xy, y).
\end{equation}
The trivial line bundles over the two spaces with the multiplication action of $\C^{\times}_U$ and $\C^{\times}_{\tilde{U}}$ will be denoted by $\C_U$ and $\C_{\tilde{U}}$. 

\begin{theorem} \cite[Proposition 8.1]{NT} \label{thm:HW}
We have 
\begin{enumerate}
\item \label{egy} $\Ch(\DD)$ isomorphic to $\Ch(\tDD)$.
\item The isomorphism in \eqref{egy} is $\rho$-equivariant between the $\T$-variety $\Ch(\DD)$ and the $\tilde{\T}$-variety $\Ch(\tDD)$. 
\item After identification of the two spaces, we have the exact sequence of  bundles
\[
0 \to \xi_2 \to \xi_3\oplus\xi_1\oplus \C_U \to \tilde{\xi}_2 \to 0.
\]
\end{enumerate}
\end{theorem}

\begin{proof} The isomorphism (without the torus actions) described in the proof of \cite[Prop. 8.1]{NT} satisfies all the assertions of the theorem. For completeness, we repeat the main construction, for details see \cite{NT}.

Consider a representative of $\Ch(\DD)$ in $\MM_{\DD}$, in particular the components corresponding to the portion of $\DD$ in \eqref{fig:HW2}, namely
\begin{equation}\label{hw-left}
\begin{tikzcd}[scale=.6]
W_1 \arrow[loop,out=120,in=60,distance=2em, "B_U"]
& &  
W_2 \ar[ll,"A_U"] \ar[ld, "b_U"]  \arrow[loop,out=120,in=60,distance=2em, "B'_U"]  \ar[rr, bend right, "D_V"] 
& &
W_3. \ar[ll, bend right, "C_V"]
\\
&\C_U \ar[ul, "a_U"] &&
\end{tikzcd}
\end{equation}
By construction these components satisfy four relations
\begin{enumerate}[(i)]
\item \label{i} $B_UA_U-A_UB'_U+a_Ub_U=0$,
\item \label{ii} $C_VD_V+B'_U=0$,
\item \label{iii} $D_VC_V=\begin{cases} 
-B_{V^{\ddag}} & \text{ if } V^{\ddag} \text{ is a D5 brane} \\
C_{V^{\ddag}}D_{V^{\ddag}} & \text{ if } V^{\ddag} \text{ is an NS5 brane},
\end{cases}$
\item \label{iv} $B_U=
\begin{cases} 
-B_{U^{=}} & \text{ if } U^{=} \text{ is a D5 brane} \\
-D_{U^{=}}C_{U^{=}} & \text{ if } U^{=} \text{ is an NS5 brane},
\end{cases}$
\end{enumerate}
and some stability conditions.

Define $\C_{\tilde{U}}=\hb^{-1}\!\C_U$ and consider the sequence of maps
\[ 
\begin{tikzcd}[scale=.6]
W_2 \ar[rr, "\alpha"] & & W_1\oplus W_3 \oplus \C_{\tilde{U}} \ar[rr, "\beta"] & & W_1,
\end{tikzcd}
\]
where the components of $\alpha$ are $A_U, D_V, b_U$ and those of $\beta$ are $B_U, A_UC_V, a_U$. 

One can verify that $\alpha$ is injective, and $\beta\alpha=0$. Define $\tilde{W}_2=\coker(\alpha)$, and consider 
\begin{equation}\label{hw-right}
\begin{tikzcd}[scale=.6]
W_1   \ar[rr, bend right, "D_{\tilde{V}}"] & & \tilde{W}_2 \arrow[loop,out=120,in=60,distance=2em, "B_{\tilde{U}}"]\ar[ll, bend right, "C_{\tilde{V}}"]
& &  W_3, \arrow[loop,out=120,in=60,distance=2em, "B'_{\tilde{U}}"] \ar[ld,"b_{\tilde{U}}"] \ar[ll,"A_{\tilde{U}}"]
& &
\\
& & &\C_{\tilde{U}} \ar[ul, "a_{\tilde{U}}"] &&
\end{tikzcd}
\end{equation}
where $A_{\tilde{U}}, a_{\tilde{U}}, D_{\tilde{U}}$ are the compositions of the natural embedding to $W_1\oplus W_3 \oplus \C_U$ and the quotient map to $\tilde{W}_2$, except for $D_{\tilde{U}}$ we multiply this composition with $-1$. Let $C_{\tilde{V}}$ be induced by $\beta$, and set $B_{\tilde{U}}=-D_{\tilde{V}}C_{\tilde{V}}$, $b_{\tilde{U}}=b_UC_V$. Define $B'_{\tilde{U}}$ to be the right hand side of (\ref{iii}), ie. its value depends whether $V^\ddag=\tilde{U}^\ddag$ is a D5 brane or an NS5 brane. [As a ``reality check'' one can verify that the $\C^\times_{\h}$-degree of each defined map is the required one, for example $\deg(b_{\tilde{U}})=$(degree at target)-(degree at source)+(degree of $b_UC_V$)$=(-1)-0+2=1$.]

Replacing the linear maps in \eqref{hw-left} with those in \eqref{hw-right} we get an element of $\MM_{\tDD}$. This represents an element in $\Ch(\tDD)$ if the new maps satisfy
\begin{enumerate}[(i')]
\item \label{ip} $B_{\tilde{U}}A_{\tilde{U}}-A_{\tilde{U}}B'_{\tilde{U}}+a_{\tilde{U}}b_{\tilde{U}}=0$,
\item \label{iip} $D_{\tilde{V}}C_{\tilde{V}}+B_{\tilde{U}}=0$,
\item \label{iiip} $C_{\tilde{V}}D_{\tilde{V}}=$RHS of (\ref{iv}),
\item \label{ivp} $B'_{\tilde{U}}=$RHS of (\ref{iii}),
\end{enumerate}
as well as some stability conditions. Equations (ii'), (iv'), and (iii') hold by the definitions of $B_{\tilde{U}}$, $B'_{\tilde{U}}$, and those of $C_{\tilde{V}}, D_{\tilde{V}}$. 
Equation (i') is proved by studying the sequence  of maps
\[ 
\begin{tikzcd}[scale=.6]
W_3 \ar[rr, "\begin{pmatrix} C_{\tilde{V}}A_{\tilde{U}} \\ -B'_{\tilde{U}} \\ b_{\tilde{U}} \end{pmatrix}"] & & W_1\oplus W_3 \oplus \C_U 
\ar[rrr, "\begin{pmatrix} -D_{\tilde{V}}\ \  A_{\tilde{U}} \ \  a_{\tilde{U}} \end{pmatrix}"] & & & \tilde{W}_2,
\end{tikzcd}
\]
and the needed stability conditions can be proved using the stability conditions for \eqref{hw-left}, see details in \cite{NT}. The inverse construction is analogous.

All the statements of the theorem can be read from the construction of \eqref{hw-right}.
\end{proof}

Observe that since $\rho$ is not the identity, the identification of the equivariant K theory  (or cohomology) of the isomorphic varieties includes an $\hb$-shift in one of the equivariant parameters.

\subsection{The  $\C^\times_{\h}$ action of \cite{NT}---comparison}
In this section we show the torus reparametrization that would identify the torus action of this paper with that of \cite[\S 6.9.3]{NT}. Assume that the generator of $K_{\C^\times_{\h}}(\pt)$, denoted by $\hb$, has a formal square root and denote it by $\hb^{1/2}$. Let $k_X$ be the number of NS5 branes left of the D3 brane~$X$. 
If we reparameterized 
\begin{itemize}
\item $W_X$ by $\hb^{k_X/2}$ ($X$ is a D3 brane), and
\item $\C_U$ by $\hb^{(-1+d_{U^-}-d_{U^+}+k_{U^+})/2}$ ($U$ is a D5 brane),
\end{itemize}
then the $\C^\times_{\h}$-degrees would change from
\[
\deg_{\hb}(A)=0, \qquad
\deg_{\hb}(B)=1, \qquad
\deg_{\hb}(B')=1, \qquad
\deg_{\hb}(C)=1, \qquad
\deg_{\hb}(D)=0,
\]
\[\deg_{\hb}(a)=0, \qquad
\deg_{\hb}(b)=1 \qquad \text{(the values of this paper)}
\]
to the values 
\[
\deg_{\hb^{1/2}}(A)=0, \qquad
\deg_{\hb^{1/2}}(B)=2, \qquad
\deg_{\hb^{1/2}}(B')=2, \qquad
\deg_{\hb^{1/2}}(C)=1, \qquad
\deg_{\hb^{1/2}}(D)=1,
\]
\[\deg_{\hb^{1/2}}(a_U)=1+d_{U^+}-d_{U^-}, \qquad
\deg_{\hb^{1/2}}(b_U)=1+d_{U^-}-d_{U^+},
\]
that agree with the degrees in \cite[\S 6.9.3]{NT}.  Although the \cite{NT} convention has conceptual advantages, in this paper we will stick with our convention.

\section{Torus fixed points}
The fixed points of the torus $\T=\A\times \C^\times=(\C^\times)^{\{\text{D5 branes}\}} \times \C_{\h}^\times$ acting on $\Ch(\DD)$ can be described by combinatorial codes.

\subsection{First combinatorial code: tie-diagram}\label{sec:FirstComb}

For a pair of 5-branes let us say that it {\em covers} the D3 branes that are in between the two 5-branes. Our first combinatorial code of a $\T$ fixed point of $\Ch(\DD)$ is a set of pairs of 5-branes such that 
\begin{itemize}
\item each pair consists of a D5 brane and an NS5 brane,
\item each D3 brane is covered by these pairs as many times as its multiplicity.
\end{itemize}
Let us emphasize that we mean a {\em set} of pairs, that is, a pair cannot be repeated.

We will depict a pair by a dotted curve connecting the two 5-branes in the diagram. For aesthetic reasons we draw this curve above the brane diagram if the NS5 brane is on the left, and below the brane diagram if the NS5 brane is on the right. A fixed point is thus encoded by a collection of such dotted curves that cover each D3 brane as many times as its multiplicity. This collection of curves we will call the tie-diagram of the fixed point.

\begin{example} \label{ex:P3} \rm
Consider the brane diagram $\DD=$\ttt{\bs 1\bs 2{\fs}2\bs 2\bs 2{\fs}}. The corresponding bow variety has six fixed points, the tie-diagram of two of them are 
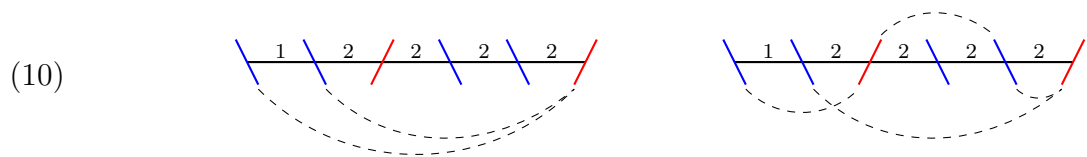
\begin{equation}\label{eq:2fixes}
\begin{tikzpicture}[baseline=0pt,scale=.3]
\draw[thick] (0,1)--(15,1) ;
\draw[thick,blue] (.5,0)--(-.5,2);
\draw[thick,blue] (3.5,0)--(2.5,2);
\draw[thick,red] (5.5,0)--(6.5,2);
\draw[thick,blue] (9.5,0)--(8.5,2);
\draw[thick,blue] (12.5,0)--(11.5,2);
\draw[thick,red] (14.5,0)--(15.5,2);
\draw [dashed, black](0.5,-.2) to [out=-45,in=225] (14.5,-.2);
\draw [dashed, black](3.5,-.2) to [out=-45,in=220] (14.5,-.2);
\node at (1.5,1.5) {\tiny 1};\node at (4.5,1.5) {\tiny 2};\node at (7.5,1.5) {\tiny 2};\node at (10.5,1.5) {\tiny 2};\node at (13.5,1.5) {\tiny 2};
\end{tikzpicture}
\qquad\qquad
\begin{tikzpicture}[baseline=0pt,scale=.3]
\draw[thick] (0,1)--(15,1) ;
\draw[thick,blue] (.5,0)--(-.5,2);
\draw[thick,blue] (3.5,0)--(2.5,2);
\draw[thick,red] (5.5,0)--(6.5,2);
\draw[thick,blue] (9.5,0)--(8.5,2);
\draw[thick,blue] (12.5,0)--(11.5,2);
\draw[thick,red] (14.5,0)--(15.5,2);
\draw [dashed, black](0.5,-.2) to [out=-45,in=225] (5.5,-.2);
\draw [dashed, black](3.5,-.2) to [out=-45,in=220] (14.5,-.2);
\draw [dashed, black](6.5,2.2) to [out=45,in=-220] (11.5,2.2);
\draw [dashed, black](12.5,-.2) to [out=-45,in=220] (14.5,-.2);
\node at (1.5,1.5) {\tiny 1};\node at (4.5,1.5) {\tiny 2};\node at (7.5,1.5) {\tiny 2};\node at (10.5,1.5) {\tiny 2};\node at (13.5,1.5) {\tiny 2};
\end{tikzpicture}
\end{equation}
and the reader is invited to find the other four. For a more conceptual solution of this exercise, see Example \ref{ex:marginsVSties}. We will also see later that $\Ch(\DD)$ is isomorphic to $T^*\!\Gr(2,4)$.
\end{example}

\begin{remark} \rm
Let us explain the physical intuition of the combinatorial codes. Consider the three dimensional view of brane diagrams (see Section \ref{sec:3dimview}). Imagine that the 5-branes can be moved in their respective planes orthogonal to the $x$-axis. We want to realize the D3 branes as ``ties'', parallel with the $x$-axis, connecting exactly two 5-branes of opposite type (the `rails'). (If the tie connected 5-branes of the same type, it would not be ``tight'', it could move in the direction of the 5-branes.) Ties may span more than one D3 brane positions, but a D3 position needs to be covered by as many ties as its multiplicity.  The combinatorial type of such tie arrangements is the same as a tie diagram. The tie arrangement corresponding to the picture on the right of~\eqref{eq:2fixes} is shown below.
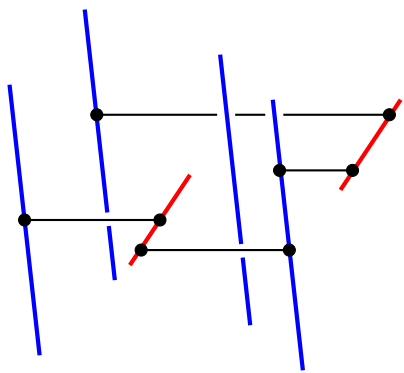
\begin{equation*}
\begin{tikzpicture}[scale=.2]
\draw[ultra thick,blue] (2,1)--(0,19);
\draw[ultra thick,blue] (7,6)--(6.6,9.6); \draw[ultra thick,blue] (6.5,10.5)--(5,24); 
\draw[ultra thick,red] (8,7)--(12,13);
\draw[ultra thick,blue] (16,3)--(15.5,7.5);\draw[ultra thick,blue] (15.4,8.4)--(14,21);
\draw[ultra thick,blue] (19.5,0)--(17.5,18);
\draw[ultra thick,red] (22,12)--(26,18);

\draw[thick, black] (1,10)--(10,10);
\draw[thick, black] (8.8,8)--(18.8,8);
\draw[thick, black] (6,17)--(13.8,17);\draw[thick, black] (15,17)--(17,17);\draw[thick, black] (18.2,17)--(25.2,17);
\draw[thick, black] (17.7,13.3)--(23,13.3); 

\draw[fill] (1,10) circle [radius=.4];
\draw[fill] (10,10) circle [radius=.4];
\draw[fill] (8.75,8) circle [radius=.4];
\draw[fill] (18.6,8) circle [radius=.4];
\draw[fill] (17.95,13.3) circle [radius=.4];
\draw[fill] (25.25,17) circle [radius=.4];
\draw[fill] (5.8,17) circle [radius=.4];
\draw[fill] (22.8,13.3) circle [radius=.4];

\end{tikzpicture}
\end{equation*}
\end{remark}

\subsection{Second combinatorial code: binary contingency tables}\label{sec:SecondComb}

\begin{definition} For given margin vectors $r\in \Z_{\geq 0}^m$, $c\in \Z_{\geq 0}^n$ an $m\times n$ matrix $M$ is called a {\em binary contingency table} (BCT) with margins $r,c$, if 
\[
M_{ij}\in \{0,1\}\text{ for all $i,j$, and }\qquad\qquad \sum_j M_{ij}=r_i, \qquad \sum_i M_{ij}=c_j.
\]
\end{definition}

\begin{remark} \rm
The set $\BCT(r,c)$ of BCTs with margins $r,c$ is studied in combinatorics and statistics. Here is a list of some relations between BCTs and other algebraic combinatorial notions.  
Elements of $\BCT(r,c)$ can be interpreted as $(m,n)$-bipartite graphs with degrees $(r,c)$ an vertices ordered in the parts. 
Clearly
\[
\prod_{i=1}^m \prod_{j=1}^n (1+x_iy_j)=\sum_{r,c} \#\BCT(r,c) \cdot x_1^{r_1}x_2^{r_2}\ldots x_m^{r_m} y_1^{c_1}y_2^{c_2}\ldots y_n^{c_n}.
\]
The number $\#\BCT(r,c)$ can be evaluated in terms of the Kostka numbers associated with Young tableaux. The Gale-Ryser theorem is a simple numerical criterion on $r,c$ determining whether $\BCT(r,c)$ is empty (cf. Assumption \ref{assume}). The Robinson-Schensted-Knuth correspondence establishes a bijection between $\BCT(r,c)$ and pairs of certain Young tableaux. For exact statements and more on the relevance of BCTs see e.g.  \cite{brualdi, barvinok} and references therein. 
\end{remark}

\begin{proposition}
There is a natural bijection between 
\begin{itemize} 
\item tie-diagrams of $\DD$;
\item BCTs of the table-with-margins corresponding to $\DD$.
\end{itemize}
\end{proposition}

\begin{proof}
Suppose we have a tie diagram. If the NS5 brane $V_i$ is left of the D5 brane $U_j$ (ie. if their potential tie would be drawn above the diagram) then put
\[
M_{ij}=\begin{cases}
1 & \text{ if } V_i, U_j \text{ are connected by a tie} \\
0 & \text{ if } V_i, U_j \text{ are not connected by a tie,}
\end{cases}
\]
and if the NS5 brane $V_i$ is right of the D5 brane $U_j$ (ie. if their potential tie would be drawn below the diagram) then put
\[
M_{ij}=\begin{cases}
0 & \text{ if } V_i, U_j \text{ are connected by a tie} \\
1 & \text{ if } V_i, U_j \text{ are not connected by a tie.}
\end{cases}
\]
It is straightforward to verify that the obtained table $M_{ij}$ is in $\BCT(r,c)$. Conversely, given a BCT for $r,c$ together with the separating line (recall that the separating line is part of the combinatorial code {\em table-with-margins}) then interpret 1s (0s) above the separating line as tying (not tying) the corresponding branes, and interpret 0s (1s) below the separating line as tying (not tying) the corresponding branes. The obtained tie-diagram is consistent with the D3 brane multiplicities. 

The two operations are clearly inverses of each other.
\end{proof}

\begin{example}\rm \label{ex:marginsVSties}
The BCTs corresponding to the two tie diagrams in \eqref{eq:2fixes} are
\[
\begin{tikzpicture}[scale=.4]
\draw[ultra thin] (0,0) -- (4,0);
\draw[ultra thin]  (0,1) -- (4,1);
\draw[ultra thin]  (0,2) -- (4,2);
\draw[ultra thin]  (0,0) -- (0,2);
\draw[ultra thin]  (1,0) -- (1,2);
\draw[ultra thin]  (2,0) -- (2,2);
\draw[ultra thin]  (3,0) -- (3,2);
\draw[ultra thin]  (4,0) -- (4,2);
\draw[ultra thick] (0,2) -- (2,2) -- (2,1) -- (4,1) -- (4,0);
\node at (.5,2.4) {\tiny $1$}; \node at (1.5,2.4) {\tiny $1$}; \node at (2.5,2.4) {\tiny $1$}; \node at (3.5,2.4) {\tiny $1$}; 
\node at (.5,3.4) {\tiny $U_1$}; \node at (1.5,3.4) {\tiny $U_2$}; \node at (2.5,3.4) {\tiny $U_3$}; \node at (3.5,3.4) {\tiny $U_4$}; 
\node at (-.5,0.5) {\tiny $2$}; \node at (-.5,1.5) {\tiny $2$};
\node at (-1.8,0.5) {\tiny $V_2$}; \node at (-1.8,1.5) {\tiny $V_1$};
\node[violet] at (0.5,0.5) {\tiny $0$};\node[violet] at (1.5,0.5) {\tiny $0$};\node[violet] at (2.5,0.5) {\tiny $1$};\node[violet] at (3.5,0.5) {\tiny $1$};
\node[violet] at (0.5,1.5) {\tiny $1$};\node[violet] at (1.5,1.5) {\tiny $1$};\node[violet] at (2.5,1.5) {\tiny $0$};\node[violet] at (3.5,1.5) {\tiny $0$};
\end{tikzpicture}
\qquad\qquad
\begin{tikzpicture}[scale=.4]
\draw[ultra thin] (0,0) -- (4,0);
\draw[ultra thin]  (0,1) -- (4,1);
\draw[ultra thin]  (0,2) -- (4,2);
\draw[ultra thin]  (0,0) -- (0,2);
\draw[ultra thin]  (1,0) -- (1,2);
\draw[ultra thin]  (2,0) -- (2,2);
\draw[ultra thin]  (3,0) -- (3,2);
\draw[ultra thin]  (4,0) -- (4,2);
\draw[ultra thick] (0,2) -- (2,2) -- (2,1) -- (4,1) -- (4,0);
\node at (.5,2.4) {\tiny $1$}; \node at (1.5,2.4) {\tiny $1$}; \node at (2.5,2.4) {\tiny $1$}; \node at (3.5,2.4) {\tiny $1$}; 
\node at (.5,3.4) {\tiny $U_1$}; \node at (1.5,3.4) {\tiny $U_2$}; \node at (2.5,3.4) {\tiny $U_3$}; \node at (3.5,3.4) {\tiny $U_4$}; 
\node at (-.5,0.5) {\tiny $2$}; \node at (-.5,1.5) {\tiny $2$};
\node at (-1.8,0.5) {\tiny $V_2$}; \node at (-1.8,1.5) {\tiny $V_1$};
\node[violet] at (0.5,0.5) {\tiny $1$};\node[violet] at (1.5,0.5) {\tiny $0$};\node[violet] at (2.5,0.5) {\tiny $1$};\node[violet] at (3.5,0.5) {\tiny $0$};
\node[violet] at (0.5,1.5) {\tiny $0$};\node[violet] at (1.5,1.5) {\tiny $1$};\node[violet] at (2.5,1.5) {\tiny $0$};\node[violet] at (3.5,1.5) {\tiny $1$};
\end{tikzpicture},
\]
and the tie-diagram corresponding to Figure \ref{fig:margins}(d) is 
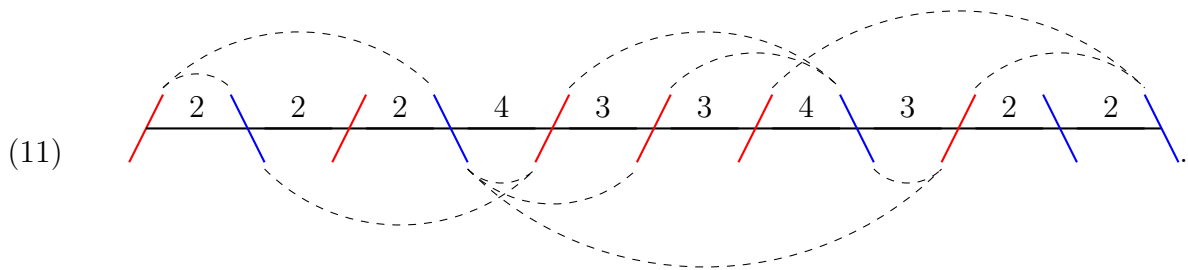
\begin{equation}\label{YiyansExample}
\begin{tikzpicture}[baseline=0,scale=.45]
\draw [thick,red] (0.5,0) --(1.5,2); 
\draw[thick] (1,1)--(2.5,1) node [above] {$2$} -- (31,1);
\draw [thick,blue](4.5,0) --(3.5,2);  
\draw [thick](4.5,1)--(5.5,1) node [above] {$2$} -- (6.5,1);
\draw [thick,red](6.5,0) -- (7.5,2);  %%
\draw [thick](7.5,1) --(8.5,1) node [above] {$2$} -- (9.5,1); 
\draw[thick,blue] (10.5,0) -- (9.5,2);  
\draw[thick] (10.5,1) --(11.5,1) node [above] {$4$} -- (12.5,1); 
\draw [thick,red](12.5,0) -- (13.5,2);   %%
\draw [thick](13.5,1) --(14.5,1) node [above] {$3$} -- (15.5,1);
\draw[thick,red] (15.5,0) -- (16.5,2);  %%
\draw [thick](16.5,1) --(17.5,1) node [above] {$3$} -- (18.5,1);  
\draw [thick,red](18.5,0) -- (19.5,2);  %%
\draw [thick](19.5,1) --(20.5,1) node [above] {$4$} -- (21.5,1);
\draw [thick,blue](22.5,0) -- (21.5,2);
\draw [thick](22.5,1) --(23.5,1) node [above] {$3$} -- (24.5,1);  
\draw[thick,red] (24.5,0) -- (25.5,2); 
\draw[thick] (25.5,1) --(26.5,1) node [above] {$2$} -- (27.5,1);
\draw [thick,blue](28.5,0) -- (27.5,2);  %%
\draw [thick](28.5,1) --(29.5,1) node [above] {$2$} -- (30.5,1);   
\draw [thick,blue](31.5,0) -- (30.5,2);   %%

\draw [dashed, black](4.5,-.2) to [out=-45,in=225] (12.5,-.2);
\draw [dashed, black](10.5,-.2) to [out=-45,in=225] (12.5,-.2);
\draw [dashed, black](10.5,-.2) to [out=-45,in=225] (15.5,-.2);
\draw [dashed, black](10.5,-.2) to [out=-45,in=225] (24.5,-.2);
\draw [dashed, black](22.5,-.2) to [out=-45,in=225] (24.5,-.2);

\draw [dashed, black](1.5,2.2) to [out=45,in=-225] (3.5,2.2);
\draw [dashed, black](1.5,2.2) to [out=45,in=-225] (9.5,2.2);
\draw [dashed, black](13.5,2.2) to [out=45,in=-225] (21.5,2.2);
\draw [dashed, black](16.5,2.2) to [out=45,in=-225] (21.5,2.2);
\draw [dashed, black](19.5,2.2) to [out=45,in=-225] (30.5,2.2);
\draw [dashed, black](25.5,2.2) to [out=45,in=-225] (30.5,2.2);
\end{tikzpicture}.
\end{equation}
\end{example}

\subsection{Third combinatorial code: butterfly diagram} \label{sec:butterfly}
In Sections \ref{sec:FirstComb}, \ref{sec:SecondComb} we just named combinatorial objects, but we have not identified them with representatives in $\MM^s$. If we want to name representatives in $\MM^s$, namely of the $\T$ fixed points,  we need to name linear maps $A_U, B_U, B'_U, a_U, b_U, C_V, D_V$ between appropriate $W_X$ or $\C_U$ vector spaces. %; see notations in Section~\ref{sec:bowdef}.

Our first remark is that there is no need to name both $B$ and $B'$ maps. At a given D3 brane $X$ we have the following endomorphisms of $W_X$: (i) both $B$ and $B'$ if \ttt{\bs X\bs}, (ii) only $B$ if \ttt{{\fs}X\bs}, (iii) only $B'$ if \ttt{\bs X{\fs}}, (iv) no $B$ or $B'$ if \ttt{{\fs}X{\fs}}. In case (i) $B$ and $B'$ are identified (see the $\NN_X$ component of the moment map) so we can call the common value $B$. In cases (ii) and (iii) we just call the only endomorphism there (we call it $B$), and in (iv) there is no $B$ or $B'$.

Our second remark is that we will drop the subscript $U, V$ if they are clear from the context.

Hence, to name a representative of, say a $\T$ fixed point, we need to name an appropriate tuple of $A, B, a, b, C, D$ linear maps. That can be cumbersome, hence, instead, we illustrate such a tuple with diagrams similar to the type A quiver lace diagrams of \cite{abeasis} see also \cite[\S2]{BFR}. 
These multiply laced diagrams will be called butterfly diagrams.

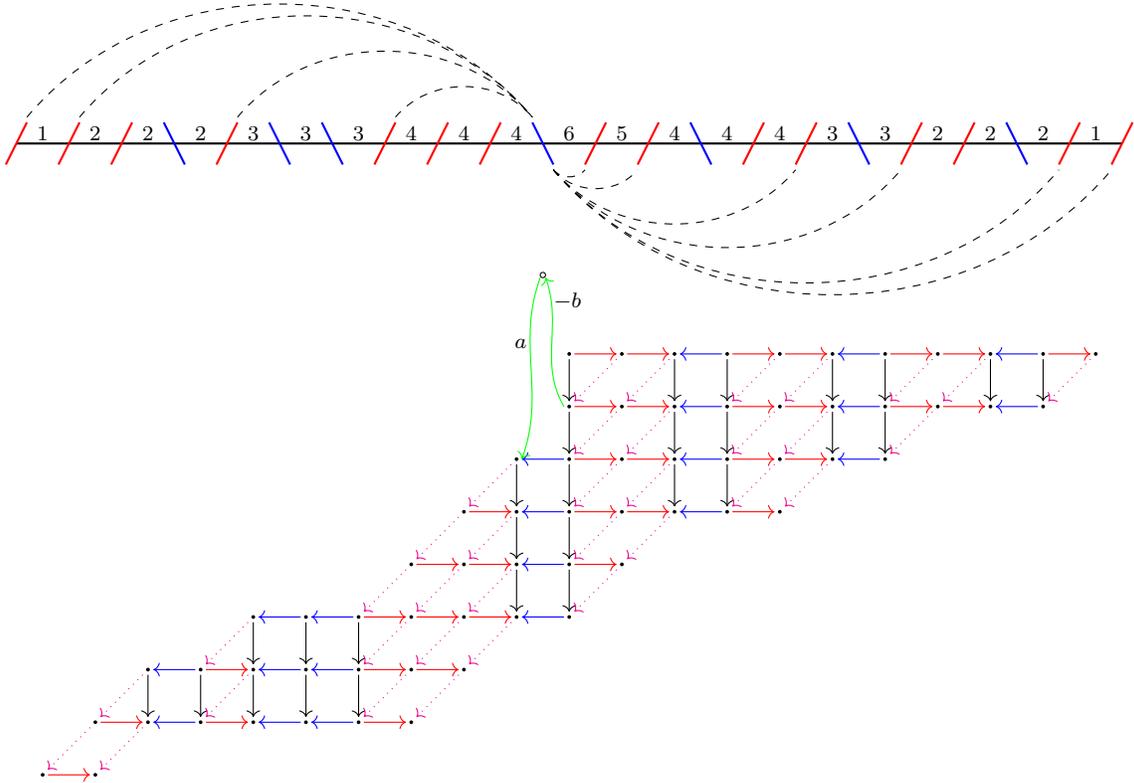
\begin{figure}
\begin{tikzpicture}[scale=.7]
\draw[fill] (0,-3) circle [radius=.025];
\draw[fill] (1,-3) circle [radius=.025]; \draw[fill] (1,-2) circle [radius=.025];
\draw[fill] (2,-2) circle [radius=.025]; \draw[fill] (2,-1) circle [radius=.025];
\draw[fill] (3,-2) circle [radius=.025]; \draw[fill] (3,-1) circle [radius=.025];
\draw[fill] (4,0) circle [radius=.025]; \draw[fill] (4,-1) circle [radius=.025]; \draw[fill] (4,-2) circle [radius=.025];
\draw[fill] (5,0) circle [radius=.025]; \draw[fill] (5,-1) circle [radius=.025]; \draw[fill] (5,-2) circle [radius=.025];
\draw[fill] (6,0) circle [radius=.025]; \draw[fill] (6,-1) circle [radius=.025]; \draw[fill] (6,-2) circle [radius=.025];
\draw[fill] (7,1) circle [radius=.025]; \draw[fill] (7,0) circle [radius=.025]; \draw[fill] (7,-1) circle [radius=.025]; \draw[fill] (7,-2) circle [radius=.025];
\draw[fill] (8,2) circle [radius=.025]; \draw[fill] (8,1) circle [radius=.025]; \draw[fill] (8,0) circle [radius=.025]; \draw[fill] (8,-1) circle [radius=.025];
\draw[fill] (9,3) circle [radius=.025]; \draw[fill] (9,2) circle [radius=.025]; \draw[fill] (9,1) circle [radius=.025]; \draw[fill] (9,0) circle [radius=.025];
\draw[fill] (10,0) circle [radius=.025]; \draw[fill] (10,1) circle [radius=.025]; \draw[fill] (10,2) circle [radius=.025]; \draw[fill] (10,3) circle [radius=.025]; \draw[fill] (10,4) circle [radius=.025]; \draw[fill] (10,5) circle [radius=.025]; 
\draw[fill] (11,1) circle [radius=.025]; \draw[fill] (11,2) circle [radius=.025]; \draw[fill] (11,3) circle [radius=.025]; \draw[fill] (11,4) circle [radius=.025]; \draw[fill] (11,5) circle [radius=.025]; 
\draw[fill] (12,2) circle [radius=.025]; \draw[fill] (12,3) circle [radius=.025]; \draw[fill] (12,4) circle [radius=.025]; \draw[fill] (12,5) circle [radius=.025]; 
\draw[fill] (13,2) circle [radius=.025]; \draw[fill] (13,3) circle [radius=.025]; \draw[fill] (13,4) circle [radius=.025]; \draw[fill] (13,5) circle [radius=.025]; 
\draw[fill] (14,2) circle [radius=.025]; \draw[fill] (14,3) circle [radius=.025]; \draw[fill] (14,4) circle [radius=.025]; \draw[fill] (14,5) circle [radius=.025]; 
\draw[fill] (15,3) circle [radius=.025]; \draw[fill] (15,4) circle [radius=.025]; \draw[fill] (15,5) circle [radius=.025]; 
\draw[fill] (16,3) circle [radius=.025]; \draw[fill] (16,4) circle [radius=.025]; \draw[fill] (16,5) circle [radius=.025]; 
\draw[fill] (17,4) circle [radius=.025]; \draw[fill] (17,5) circle [radius=.025]; 
\draw[fill] (18,4) circle [radius=.025]; \draw[fill] (18,5) circle [radius=.025]; 
\draw[fill] (19,4) circle [radius=.025]; \draw[fill] (19,5) circle [radius=.025]; 
\draw[fill] (20,5) circle [radius=.025]; 

\draw [->] (2,-1.1) -- (2,-1.9);
\draw [->] (3,-1.1) -- (3,-1.9);
\draw [->] (4,-1.1) -- (4,-1.9); \draw [->] (4,-0.1) -- (4,-0.9); 
\draw [->] (5,-1.1) -- (5,-1.9); \draw [->] (5,-0.1) -- (5,-0.9); 
\draw [->] (6,-1.1) -- (6,-1.9); \draw [->] (6,-0.1) -- (6,-0.9); 
\draw [->] (9,2.9) -- (9,2.1); \draw [->] (9,1.9) -- (9,1.1); \draw [->] (9,0.9) -- (9,0.1); 
\draw [->] (10,4.9) -- (10,4.1); \draw [->] (10,3.9) -- (10,3.1); \draw [->] (10,2.9) -- (10,2.1); \draw [->] (10,1.9) -- (10,1.1); \draw [->] (10,0.9) -- (10,0.1); 
\draw [->] (12,4.9) -- (12,4.1); \draw [->] (12,3.9) -- (12,3.1); \draw [->] (12,2.9) -- (12,2.1);
\draw [->] (13,4.9) -- (13,4.1); \draw [->] (13,3.9) -- (13,3.1); \draw [->] (13,2.9) -- (13,2.1);
\draw [->] (15,4.9) -- (15,4.1); \draw [->] (15,3.9) -- (15,3.1); 
\draw [->] (16,4.9) -- (16,4.1); \draw [->] (16,3.9) -- (16,3.1); 
\draw [->] (18,4.9) -- (18,4.1);
\draw [->] (19,4.9) -- (19,4.1);

\draw[red, ->] (10.1,5)--(10.9,5); \draw[red, ->] (10.1,4)--(10.9,4); \draw[red, ->] (10.1,3)--(10.9,3); \draw[red, ->] (10.1,2)--(10.9,2); \draw[red, ->] (10.1,1)--(10.9,1);
\draw[red, ->] (11.1,5)--(11.9,5); \draw[red, ->] (11.1,4)--(11.9,4); \draw[red, ->] (11.1,3)--(11.9,3); \draw[red, ->] (11.1,2)--(11.9,2); 
\draw[blue, ->] (12.9,5)--(12.1,5);\draw[blue, ->] (12.9,4)--(12.1,4); \draw[blue, ->] (12.9,3)--(12.1,3); \draw[blue, ->] (12.9,2)--(12.1,2);
\draw[red, ->] (13.1,5)--(13.9,5);\draw[red, ->] (13.1,4)--(13.9,4); \draw[red, ->] (13.1,3)--(13.9,3); \draw[red, ->] (13.1,2)--(13.9,2); 
\draw[red, ->] (14.1,5)--(14.9,5);\draw[red, ->] (14.1,4)--(14.9,4); \draw[red, ->] (14.1,3)--(14.9,3);
\draw[blue, ->] (15.9,5)--(15.1,5);\draw[blue, ->] (15.9,4)--(15.1,4); \draw[blue, ->] (15.9,3)--(15.1,3); 
\draw[red, ->] (16.1,5)--(16.9,5);\draw[red, ->] (16.1,4)--(16.9,4); 
\draw[red, ->] (17.1,5)--(17.9,5);\draw[red, ->] (17.1,4)--(17.9,4); 
\draw[blue, ->] (18.9,5)--(18.1,5);\draw[blue, ->] (18.9,4)--(18.1,4); 
\draw[red, ->] (19.1,5)--(19.9,5);

\draw [magenta, dotted,->] (10.9,4.9) -- (10.1,4.1);\draw [magenta, dotted,->] (10.9,3.9) -- (10.1,3.1);\draw [magenta, dotted,->] (10.9,2.9) -- (10.1,2.1);\draw [magenta, dotted,->] (10.9,1.9) -- (10.1,1.1);\draw [magenta, dotted,->] (10.9,0.9) -- (10.1,0.1);
\draw [magenta, dotted,->] (11.9,4.9) -- (11.1,4.1);\draw [magenta, dotted,->] (11.9,3.9) -- (11.1,3.1);\draw [magenta, dotted,->] (11.9,2.9) -- (11.1,2.1);\draw [magenta, dotted,->] (11.9,1.9) -- (11.1,1.1);
\draw [magenta, dotted,->] (13.9,4.9) -- (13.1,4.1);\draw [magenta, dotted,->] (13.9,3.9) -- (13.1,3.1);\draw [magenta, dotted,->] (13.9,2.9) -- (13.1,2.1);
\draw [magenta, dotted,->] (14.9,4.9) -- (14.1,4.1);\draw [magenta, dotted,->] (14.9,3.9) -- (14.1,3.1);\draw [magenta, dotted,->] (14.9,2.9) -- (14.1,2.1);
\draw [magenta, dotted,->] (16.9,4.9) -- (16.1,4.1);\draw [magenta, dotted,->] (16.9,3.9) -- (16.1,3.1);
\draw [magenta, dotted,->] (17.9,4.9) -- (17.1,4.1);
\draw [magenta, dotted,->] (19.9,4.9) -- (19.1,4.1);

\draw[blue, ->] (9.9,0)--(9.1,0); \draw[blue, ->] (9.9,1)--(9.1,1); \draw[blue, ->] (9.9,2)--(9.1,2); \draw[blue, ->] (9.9,3)--(9.1,3); 
\draw[blue, ->] (5.9,0)--(5.1,0); \draw[blue, ->] (5.9,-1)--(5.1,-1); \draw[blue, ->] (5.9,-2)--(5.1,-2); 
\draw[blue, ->] (4.9,0)--(4.1,0); \draw[blue, ->] (4.9,-1)--(4.1,-1); \draw[blue, ->] (4.9,-2)--(4.1,-2); 
\draw[blue, ->] (2.9,-2)--(2.1,-2); \draw[blue, ->] (2.9,-1)--(2.1,-1); 

\draw [magenta, dotted,->] (8.9,0.9) -- (8.1,0.1);\draw [magenta, dotted,->] (8.9,-0.1) -- (8.1,-0.9);\draw [magenta, dotted,->] (8.9,1.9) -- (8.1,1.1);\draw [magenta, dotted,->] (8.9,2.9) -- (8.1,2.1);
\draw [magenta, dotted,->] (7.9,0.9) -- (7.1,0.1);\draw [magenta, dotted,->] (7.9,-0.1) -- (7.1,-0.9);\draw [magenta, dotted,->] (7.9,-1.1) -- (7.1,-1.9);\draw [magenta, dotted,->] (7.9,1.9) -- (7.1,1.1);
\draw [magenta, dotted,->] (6.9,0.9) -- (6.1,0.1);\draw [magenta, dotted,->] (6.9,-0.1) -- (6.1,-0.9);\draw [magenta, dotted,->] (6.9,-1.1) -- (6.1,-1.9);
\draw [magenta, dotted,->] (3.9,-0.1) -- (3.1,-0.9);\draw [magenta, dotted,->] (3.9,-1.1) -- (3.1,-1.9);
\draw [magenta, dotted,->] (1.9,-2.1) -- (1.1,-2.9);\draw [magenta, dotted,->] (1.9,-1.1) -- (1.1,-1.9);
\draw [magenta, dotted,->] (0.9,-2.1) -- (0.1,-2.9);

\draw[red, ->] (0.1,-3)--(0.9,-3); \draw[red, ->] (1.1,-2)--(1.9,-2); \draw[red, ->] (3.1,-2)--(3.9,-2);\draw[red, ->] (3.1,-1)--(3.9,-1); 
\draw[red, ->] (6.1,0)--(6.9,0);\draw[red, ->] (6.1,-1)--(6.9,-1);\draw[red, ->] (6.1,-2)--(6.9,-2);
\draw[red, ->] (7.1,1)--(7.9,1);\draw[red, ->] (7.1,0)--(7.9,0);\draw[red, ->] (7.1,-1)--(7.9,-1);
\draw[red, ->] (8.1,2)--(8.9,2);\draw[red, ->] (8.1,1)--(8.9,1);\draw[red, ->] (8.1,0)--(8.9,0);

\draw (9.5,6.5) circle [radius=.05];
\draw[<-, green] (9.55,6.45) to [out=-70,in=120] (9.9,4); \draw[->,green] (9.45,6.45) to [out=-110,in=70] (9.1,3);
\node [left] at (9.4,5.2) {\tiny $a$};\node [left] at (10.43,6) {\tiny $-b$};

\draw [thick] (-.5,9)--(20.5,9);
\draw [thick, blue] (2.7,8.6)--(2.3,9.4);\draw [thick,blue] (4.7,8.6)--(4.3,9.4);\draw [blue,thick] (5.7,8.6)--(5.3,9.4);
\draw [thick,blue] (9.7,8.6)--(9.3,9.4);\draw [thick,blue] (12.7,8.6)--(12.3,9.4);\draw [thick,blue] (15.7,8.6)--(15.3,9.4);
\draw [thick,blue] (18.7,8.6)--(18.3,9.4);
\draw [thick,red] (-.7,8.6)--(-0.3,9.4);\draw [thick,red] (0.3,8.6)--(0.7,9.4);
\draw [thick,red] (1.3,8.6)--(1.7,9.4);\draw [thick,red] (3.3,8.6)--(3.7,9.4);\draw [thick,red] (6.3,8.6)--(6.7,9.4);
\draw [thick,red] (7.3,8.6)--(7.7,9.4);\draw [thick,red] (8.3,8.6)--(8.7,9.4);\draw [thick,red] (10.3,8.6)--(10.7,9.4);
\draw [thick,red] (11.3,8.6)--(11.7,9.4);\draw [thick,red] (13.3,8.6)--(13.7,9.4);\draw [thick,red] (14.3,8.6)--(14.7,9.4);
\draw [thick,red] (16.3,8.6)--(16.7,9.4);\draw [thick,red] (17.3,8.6)--(17.7,9.4);\draw [thick,red] (19.3,8.6)--(19.7,9.4);\draw [thick,red] (20.3,8.6)--(20.7,9.4);

\draw [black,dashed]( 9.7,8.5) to [out=-50,in=230] (20.3,8.5);
\draw [black,dashed]( 9.7,8.5) to [out=-50,in=230] (19.3,8.5);
\draw [black,dashed]( 9.7,8.5) to [out=-50,in=230] (16.3,8.5);
\draw [black,dashed]( 9.7,8.5) to [out=-50,in=230] (14.3,8.5);
\draw [black,dashed]( 9.7,8.5) to [out=-50,in=230] (11.3,8.5);
\draw [black,dashed]( 9.7,8.5) to [out=-50,in=230] (10.3,8.5);

\draw [black,dashed]( -.3,9.5) to [out=50,in=130] (9.3,9.5);
\draw [black,dashed]( 0.7,9.5) to [out=50,in=130] (9.3,9.5);
\draw [black,dashed]( 3.7,9.5) to [out=50,in=130] (9.3,9.5);
\draw [black,dashed]( 6.7,9.5) to [out=50,in=130] (9.3,9.5);

\node at (0,9.2) {\tiny 1};\node at (1,9.2) {\tiny 2};\node at (2,9.2) {\tiny 2};\node at (3,9.2) {\tiny 2};\node at (4,9.2) {\tiny 3};\node at (5,9.2) {\tiny 3};\node at (6,9.2) {\tiny 3};
\node at (7,9.2) {\tiny 4};\node at (8,9.2) {\tiny 4};\node at (9,9.2) {\tiny 4};\node at (10,9.2) {\tiny 6};\node at (11,9.2) {\tiny 5};\node at (12,9.2) {\tiny 4};\node at (13,9.2) {\tiny 4};
\node at (14,9.2) {\tiny 4};\node at (15,9.2) {\tiny 3};\node at (16,9.2) {\tiny 3};\node at (17,9.2) {\tiny 2};\node at (18,9.2) {\tiny 2};\node at (19,9.2) {\tiny 2};\node at (20,9.2) {\tiny 1};
 \end{tikzpicture}
 \caption{A butterfly corresponding to the D5 brane $U_4$. The numbers in the brane diagram mean the number of times ties connected to $U_4$ (the only ties drawn in this picture) cover that D3 brane. Dots represent basis vectors. Action of $-B$: vertical arrows. Action of $A$: blue arrows. Action of $C$: skew magenta dotted arrows. Action of $D$: red arrows. }\label{fig:butterly}
 \end{figure}

Now we describe a butterfly diagram associated with a tie-diagram. A butterfly diagram has components---called butterflies---corresponding to the D5 branes. To describe the butterfly of the D5 brane $U$ consider the ties with one end at $U$, e.g. the one on top of Figure \ref{fig:butterly}. Below every D3 brane we create a column of dots whose height is the number of times that D3 brane is covered by the ties adjacent to $U$. Align these columns of dots the following way:
\begin{itemize}
\item If the 5-brane in between two consecutive columns is a D5 brane, align the columns of dots from the bottom.
\item Right of $U$: if the 5-brane in between two consecutive columns is an NS5 brane, align the columns of dots from the top.
\item Left of $U$: if the 5-brane in between two consecutive columns is an NS5 brane, align the columns of dots from the top {\em in a $45^\circ$ angle}. (That is, for a pair of adjacent columns, the top dot for the column on the left is one unit below the top dot for the column on the right.) 
\end{itemize}
Having this diagram of dots, now we add arrows. 
\begin{itemize}
\item[$A$] For a pair of consecutive columns under a D5 brane we add horizontal left pointing length 1 (blue) arrows. 
\item[$B$] If on either side of a column there is an D5 brane then we add downward pointing length 1 (black) arrows between dots in that column. 
\item[$C$] For a pair of consecutive columns under an NS5 brane we add (magenta dotted) arrows pointing left and down (ie. representing the vector $(-1,-1)$). 
\item[$D$] For a pair of consecutive columns under an NS5 brane we add horizontal right pointing length 1 (red) arrows. 
\end{itemize}
The obtained diagram represents a sequence of vector spaces and linear maps among them as follows. Consider a vector space for each D3 brane with basis the dots in the column below that D3 brane. The linear maps $A,-B,C,D$ are maps among these spaces, and they map the basis vectors according to the arrows itemized in $A,B,C,D$ above. If for a particular color there is no such arrow pointing out of a dot, then that basis vector is mapped to 0. 

We still need to describe the linear maps $a_U, b_U$.  The map $a_U$ maps $\C_U$ to the basis vector represented by the {\em top} dot under $U^-$. The map $-b_U$ is represented by an arrow from the dot under $U^+$ which is one higher then the top dot under $U^-$ (if there is such dot, otherwise $b_U=0$), see the figure below.
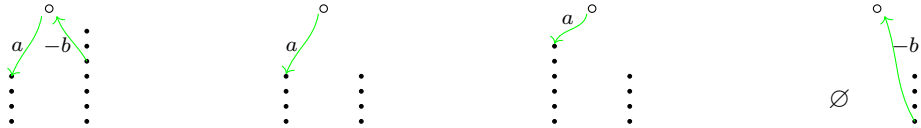
\[
\begin{tikzpicture}
\draw[fill] (0,0) circle [radius=.025]; \draw[fill] (0,.2) circle [radius=.025]; \draw[fill] (0,.4) circle [radius=.025]; \draw[fill] (0,0.6) circle [radius=.025]; 
\draw[fill] (1,0) circle [radius=.025]; \draw[fill] (1,.2) circle [radius=.025]; \draw[fill] (1,.4) circle [radius=.025]; \draw[fill] (1,0.6) circle [radius=.025]; \draw[fill] (1,0.8) circle [radius=.025]; \draw[fill] (1,1) circle [radius=.025]; \draw[fill] (1,1.2) circle [radius=.025]; 
\draw (0.5,1.5) circle [radius=.05];
\draw[<-,green] (0.6,1.4) to [out=-70,in=120] (1,0.8); 
\draw[<-,green] (0,0.6) to [out=70,in=-100] (0.4,1.4);
\node [left] at (0.3,1) {\tiny $a$};\node [left] at (0.93,1) {\tiny $-b$};
\end{tikzpicture}
\qquad\qquad\qquad
\begin{tikzpicture}
\draw[fill] (0,0) circle [radius=.025]; \draw[fill] (0,.2) circle [radius=.025]; \draw[fill] (0,.4) circle [radius=.025]; \draw[fill] (0,0.6) circle [radius=.025]; 
\draw[fill] (1,0) circle [radius=.025]; \draw[fill] (1,.2) circle [radius=.025]; \draw[fill] (1,.4) circle [radius=.025]; \draw[fill] (1,0.6) circle [radius=.025]; 
\draw (0.5,1.5) circle [radius=.05];
\draw[<-,green] (0,0.63) to [out=70,in=-100] (0.43,1.43);
\node [left] at (0.3,1) {\tiny $a$};
\end{tikzpicture}
\qquad\qquad\qquad
\begin{tikzpicture}
\draw[fill] (0,0) circle [radius=.025]; \draw[fill] (0,.2) circle [radius=.025]; \draw[fill] (0,.4) circle [radius=.025]; \draw[fill] (0,0.6) circle [radius=.025]; \draw[fill] (0,0.8) circle [radius=.025]; \draw[fill] (0,1) circle [radius=.025]; 
\draw[fill] (1,0) circle [radius=.025]; \draw[fill] (1,.2) circle [radius=.025]; \draw[fill] (1,.4) circle [radius=.025]; \draw[fill] (1,0.6) circle [radius=.025]; 
\draw (0.5,1.5) circle [radius=.05];
\draw[<-,green] (0,1.03) to [out=70,in=-100] (0.43,1.43);
\node [left] at (0.4,1.35) {\tiny $a$};
\end{tikzpicture}
\qquad\qquad\qquad
\begin{tikzpicture}
\node[] at (0,.3) {\tiny $\emptyset$};
%\draw[fill] (0,0) circle [radius=.025]; \draw[fill] (0,.2) circle [radius=.025]; \draw[fill] (0,.4) circle [radius=.025]; \draw[fill] (0,0.6) circle [radius=.025]; \draw[fill] (0,0.8) circle [radius=.025]; \draw[fill] (0,1) circle [radius=.025]; 
\draw[fill] (1,0) circle [radius=.025]; \draw[fill] (1,.2) circle [radius=.025]; \draw[fill] (1,.4) circle [radius=.025]; \draw[fill] (1,0.6) circle [radius=.025]; 
\draw (0.5,1.5) circle [radius=.05];
\draw[<-,green] (0.6,1.4) to [out=-70,in=120] (1,0); 
\node [left] at (1.2,1) {\tiny $-b$};
\end{tikzpicture}
\]

The disjoint union of butterflies for each D5 brane is the butterfly diagram associated with the tie diagram. The direct sum of the corresponding vector spaces and linear maps is a representative in $\MM^s$ of the fixed point (proved in the next section).

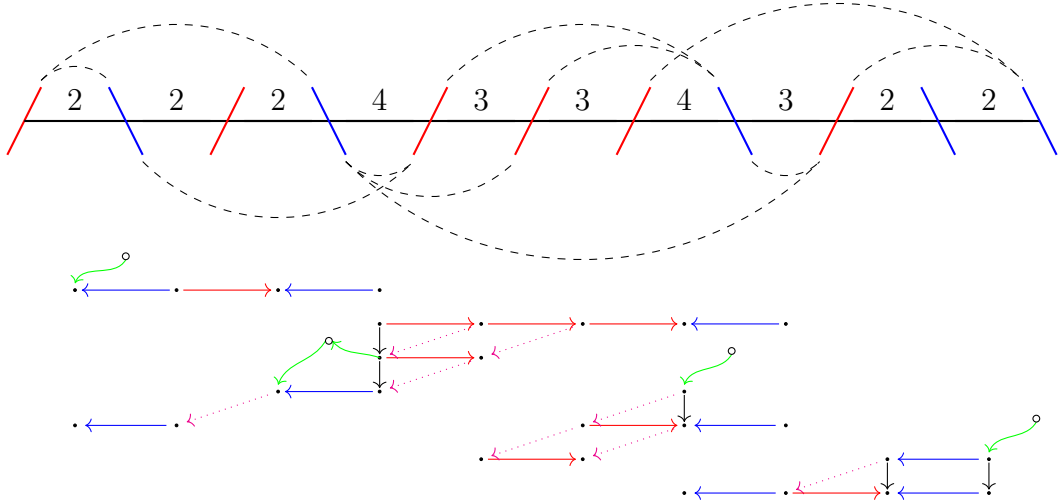
\begin{figure}
\[
\begin{tikzpicture}[scale=.45]
\draw [thick, red] (0.5,0) --(1.5,2); 
\draw[thick] (1,1)--(2.5,1) node [above] {$2$} -- (31,1);
\draw [thick,blue](4.5,0) --(3.5,2);  
\draw [thick](4.5,1)--(5.5,1) node [above] {$2$} -- (6.5,1);
\draw [thick,red](6.5,0) -- (7.5,2);  
\draw [thick](7.5,1) --(8.5,1) node [above] {$2$} -- (9.5,1); 
\draw[thick,blue] (10.5,0) -- (9.5,2);  
\draw[thick] (10.5,1) --(11.5,1) node [above] {$4$} -- (12.5,1); 
\draw [thick,red](12.5,0) -- (13.5,2);   
\draw [thick](13.5,1) --(14.5,1) node [above] {$3$} -- (15.5,1);
\draw[thick,red] (15.5,0) -- (16.5,2); 
\draw [thick](16.5,1) --(17.5,1) node [above] {$3$} -- (18.5,1);  
\draw [thick,red](18.5,0) -- (19.5,2); 
\draw [thick](19.5,1) --(20.5,1) node [above] {$4$} -- (21.5,1);
\draw [thick,blue](22.5,0) -- (21.5,2);
\draw [thick](22.5,1) --(23.5,1) node [above] {$3$} -- (24.5,1);  
\draw[thick,red] (24.5,0) -- (25.5,2); 
\draw[thick] (25.5,1) --(26.5,1) node [above] {$2$} -- (27.5,1);
\draw [thick,blue](28.5,0) -- (27.5,2); 
\draw [thick](28.5,1) --(29.5,1) node [above] {$2$} -- (30.5,1);   
\draw [thick,blue](31.5,0) -- (30.5,2);   

\draw [dashed, black](4.5,-.2) to [out=-45,in=225] (12.5,-.2);
\draw [dashed, black](10.5,-.2) to [out=-45,in=225] (12.5,-.2);
\draw [dashed, black](10.5,-.2) to [out=-45,in=225] (15.5,-.2);
\draw [dashed, black](10.5,-.2) to [out=-45,in=225] (24.5,-.2);
\draw [dashed, black](22.5,-.2) to [out=-45,in=225] (24.5,-.2);

\draw [dashed, black](1.5,2.2) to [out=45,in=-225] (3.5,2.2);
\draw [dashed, black](1.5,2.2) to [out=45,in=-225] (9.5,2.2);
\draw [dashed, black](13.5,2.2) to [out=45,in=-225] (21.5,2.2);
\draw [dashed, black](16.5,2.2) to [out=45,in=-225] (21.5,2.2);
\draw [dashed, black](19.5,2.2) to [out=45,in=-225] (30.5,2.2);
\draw [dashed, black](25.5,2.2) to [out=45,in=-225] (30.5,2.2);

\draw[fill] (2.5,-4) circle [radius=.04]; \draw[fill] (5.5,-4) circle [radius=.04]; \draw[fill] (8.5,-4) circle [radius=.04]; \draw[fill] (11.5,-4) circle [radius=.04];
\draw[blue, <-] (2.7,-4)--(5.3,-4); \draw[red, ->] (5.7,-4)--(8.3,-4); \draw[blue, <-] (8.7,-4)--(11.3,-4);
\draw[<-,green] (2.5,-3.8) to [out=70,in=-120] (4,-3.1); \draw[] (4,-3) circle [radius=.1];

\draw[fill] (11.5,-5) circle [radius=.04]; \draw[fill] (14.5,-5) circle [radius=.04]; \draw[fill] (17.5,-5) circle [radius=.04]; \draw[fill] (20.5,-5) circle [radius=.04]; \draw[fill] (23.5,-5) circle [radius=.04];
\draw[red, ->] (11.7,-5)--(14.3,-5);\draw[red, ->] (14.7,-5)--(17.3,-5);\draw[red, ->] (17.7,-5)--(20.3,-5);\draw[blue,<-] (20.7,-5)--(23.3,-5);
\draw[fill] (11.5,-6) circle [radius=.04]; \draw[fill] (14.5,-6) circle [radius=.04];
\draw[red, ->] (11.7,-6)--(14.3,-6);
\draw[fill] (11.5,-7) circle [radius=.04];
\draw [magenta, dotted,->] (14.2,-5.1) -- (11.8,-5.9);\draw [magenta, dotted,->] (17.2,-5.1) -- (14.8,-5.9);\draw [magenta, dotted,->] (14.2,-6.1) -- (11.8,-6.9);
\draw[blue,<-] (8.7,-7)--(11.3,-7);
\draw[fill] (8.5,-7) circle [radius=.04];\draw[fill] (5.5,-8) circle [radius=.04];\draw[fill] (2.5,-8) circle [radius=.04];
\draw [magenta, dotted,->] (8.2,-7.1) -- (5.8,-7.9);\draw [blue, ->] (5.2,-8) -- (2.8,-8);
\draw [->] (11.5,-6.1) -- (11.5,-6.9);\draw [->] (11.5,-5.1) -- (11.5,-5.9);
\draw[<-,green] (8.5,-6.8) to [out=70,in=-120] (9.9,-5.5); \draw[<-,green] (10.1,-5.5) to [out=-50,in=160] (11.5,-6); \draw[] (10,-5.5) circle [radius=.1];

\draw[fill] (17.5,-8) circle [radius=.04]; \draw[fill] (20.5,-8) circle [radius=.04]; \draw[fill] (23.5,-8) circle [radius=.04];\draw[fill] (20.5,-7) circle [radius=.04];\draw[fill] (17.5,-9) circle [radius=.04];\draw[fill] (14.5,-9) circle [radius=.04];
\draw[red, ->] (14.7,-9)--(17.3,-9);\draw[red, ->] (17.7,-8)--(20.3,-8);\draw [blue, ->] (23.2,-8) -- (20.8,-8);
\draw [magenta, dotted,->] (20.2,-7.1) -- (17.8,-7.9);\draw [magenta, dotted,->] (20.2,-8.1) -- (17.8,-8.9);\draw [magenta, dotted,->] (17.2,-8.1) -- (14.8,-8.9);
\draw [->] (20.5,-7.1) -- (20.5,-7.9);
\draw[<-,green] (20.5,-6.8) to [out=70,in=-120] (21.9,-5.9);\draw[] (21.9,-5.8) circle [radius=.1];

\draw[fill] (20.5,-10) circle [radius=.04];\draw[fill] (23.5,-10) circle [radius=.04];\draw[fill] (26.5,-10) circle [radius=.04];\draw[fill] (29.5,-10) circle [radius=.04];\draw[fill] (26.5,-9) circle [radius=.04];\draw[fill] (29.5,-9) circle [radius=.04];
\draw [->] (29.5,-9.1) -- (29.5,-9.9);\draw [->] (26.5,-9.1) -- (26.5,-9.9);
\draw [blue, ->] (29.2,-10) -- (26.8,-10);\draw [blue, ->] (29.2,-9) -- (26.8,-9);\draw [blue, ->] (23.2,-10) -- (20.8,-10);
\draw [magenta, dotted,->] (26.2,-9.1) -- (23.8,-9.9);
\draw[red, ->] (23.7,-10)--(26.3,-10);
\draw[<-,green] (29.5,-8.8) to [out=70,in=-120] (30.9,-7.9);\draw[] (30.9,-7.8) circle [radius=.1];
\end{tikzpicture}
\]
\caption{The tie and the butterfly diagrams of a fixed point.}\label{YiyansButterfly}
\end{figure}

\begin{example} \rm 
The butterfly diagram of the fixed point whose tie diagram is \eqref{YiyansExample} is in Figure \ref{YiyansButterfly}.
\end{example}

\subsection{The proof of the combinatorial codes for fixed points}

So far we have presented three different combinatorial objects (tie diagrams, BCTs, butterfly diagrams) associated with a brane diagram $\DD$, and we showed the natural bijections among them. We still need to show that they are in bijection with the $\T$ fixed points of $\Ch(\DD)$. We will need an alternative formulation of the $\nu^{\R}$ stability condition for this purpose. By \cite[Proposition 2.8]{NT}, $(A,B,C,D,a,b)\in \widetilde{\M}$ is stable if and only if the only $A,B,C,D$-invariant subspace $S \subset\mathbf{W}=\bigoplus_{X\text{ D3}}W_X$ such that $\mathrm{im}(a)\subset S$ and $A$ induces isomorphisms on $\mathbf{W}/S$ is the whole space $\mathbf{W}$.

\begin{theorem}
Every $\T$ fixed point in $\Ch(\DD)$ is represented by a butterfly diagram for $\DD$. Every butterfly diagram is a stable bow representation (ie. an element of $\MM^s$) representing a fixed point.
\end{theorem}

\begin{proof}
Let $f$ be a fixed point and identify each vector space $W$ with the fibre of the corresponding tautological bundle over $f$. This endows each $W$ with the structure of an $\mathbb T = \mathbb{A} \oplus \mathbb{C}_{\hbar}$ representation. The direct sum of all $W$ spaces $\mathbf{W}$ also inherits a $\mathbb T$ action. Hence, each of these spaces decomposes into weight spaces. The $A, B, C, D$ maps send weight spaces to weight spaces. In particular, the $B$ and $C$ maps decrease the $\hbar$-weights by 1, while the $A$ and $D$ maps are homomorphisms of $\mathbb{T}$-representations. The $\nu^{\R}$ stability condition implies that the $a$ maps are all nonzero. Given a D5 brane $U$, the vector $a_U(1)$ has weight $u$ under the action of $\mathbb{A}$. It is clear from the structure of $\mathbb{M}$ that the weights of the $\mathbb{A}$-action on $\mathbf{W}$ are all of the form $u$ for some D5 brane $U$. We will show that each of these weight spaces corresponds to a butterfly.

Fix a D5 brane $U$. Let $U_0\neq U$ be another D5 brane. By comparing weights, we see that $a_{U_0}$ and $b_{U_0}$ vanish on the $u$-weight space of $\mathbf{W}$ under the action of $\mathbb{A}$. The 0-momentum condition reduces to $A_{U_0} B'_{U_0} = B_{U_0} A_{U_0}$ on this weight space. This relation implies that the portions of $\ker A_{U_0}$ and $\mathrm{im } A_{U_0}$ in the $u$-weight space are $B$-invariant. The (S1) and (S2) conditions imply that $A_{U_0}$ is an isomorphism on the $u$-weight space. Let $l$ be the dimension of the $u$-weight space of $W_{U^-}$ and $r$ be the dimension of the $u$-weight space of $W_{U^+}$.

We first consider the case where $l \geq r$. Since $A_U$ is full rank, $A_U$ is injective on the $u$-weight space of $W_{U^+}$ under the $\mathbb{A}$-action. Let $w = a_U(1)$. Let $E\subset W_{U^-}$ be the subspace spanned by vectors of the form $B_U^i(w)$ for $i\geq 0$ and $E' = A_U^{-1}(E)$. From the 0-momentum condition, we have the relation $ A_U B'_U - a_U b_U = B_U A_U$. This relation implies that $\mathrm{im} A_U + E$ is $B_U$-invariant. The (S2) condition implies that $W_{U^-} = \mathrm{im} A_U + E$. Hence, $A_U$ induces an isomorphism $W_{U^+}{/}E' \to W_{U^-}{/}E$. We may extend $E\oplus E'$ to a $A,A^{-1},B,C,D$-invariant subspace $\hat{E} \subset \mathbf{W}$ by acting by all available maps. Let $\mathbf{E}$ be the direct sum of $\hat{E}$ with all $u_0$-weight spaces of the $\mathbb{A}$-action where $u_0\neq u$. The $A$ maps induce isomorphisms on the relevant subspaces of $\mathbf{W}/\mathbf{E}$, so the $\nu^{\R}$ stability condition implies that $\mathbf{W} = \mathbf{E}$.

The $u$-weight spaces of $W_{U^-}$ and $W_{U^+}$ under the $\mathbb{A}$-action are, therefore, $E$ and $E'$ respectively. It follows that the $(u+i\h)$-weight spaces of $W_{U^-}$ and $W_{U^+}$ are 1-dimensional and connected by $B$ maps. Moreover, the highest $\hbar$-weight of $E$ is 0, and $w$ is a vector of highest $\hbar$-weight. By comparing weights on both sides of the 0-momentum relation $a_U b_U = A_U B'_U - B_U A_U$, we see that $b_U$ vanishes on all weight spaces except the $(u+\h)$-weight space. The space $W_{U^+}$ does not have $u+\h$ as a weight, so $b_U = 0$. By the argument above, $\mathrm{im} A_U$ is $B_U$-invariant, so the lowest $\hbar$-weight of $E$ is the same as that of $E'$. Representing the $\hbar$-weight spaces of $E$ and $E'$ by dots aligned at the bottom and acting by all available $A, A^{-1}, B, C, D$ maps generates a butterfly.

Next, we consider the case $l <  r$. Then, $A_U$ is surjective on the $u$-weight spaces of the $\mathbb{A}$-action. Like before, we have a nonzero vector $w = a(1)$ of weight $u$ and $b_U$ vanishes on all weight spaces except the $(u+\h)$-weight space. Let $E\subset W_{U^-}$ be the subspace above generated by acting upon $w$ repeatedly by $B_U$. Let $E''\subset W_{U^+}$ be the direct sum of $(u+i\h)$-weight spaces where $i\leq 0$. The map $b_U$ vanishes on vectors of $\hbar$-weight greater than 1, so the 0-momentum condition implies that the portion of $\ker A_U$ with $\hbar$-weight greater than 1 is $B'_U$-invariant. The portion of $\ker A_U$ with weight $u+i\hbar$ where $i\leq 1$ is mapped to $E''$ by $B_U'$. Therefore, $E' = \ker A_U + E''$ is $B'_U$-invariant. As before, we extend $E'$ to a $A,A^{-1}, B,C, D$-invariant subspace $\mathbf{E} \subset \mathbf{W}$ with the property that $\mathbf{E}$ contains the $u_0$-weight space of the $\mathbb{A}$-action for all $u_0\neq u$, and the $A$ maps induce isomorphisms on $\mathbf{W}{/}\mathbf{E}$. The $\nu^{\R}$ stability condition implies that $E'$ is the $u$-weight space of $W_{U^+}$ under the $\mathbb{A}$-action. It follows that $A_U$ vanishes on vectors of weight $u+i\h$ where $i>0$. Since $A_U$ is surjective, $w\in E$ is a vector of highest $\hbar$-weight, and $E$ is the $u$-weight space of $W_{U^-}$ under the $\mathbb{A}$-action. The remainder of the argument is similar to the previous paragraph.

The converse direction follows easily by verifying the commutativity relations and stability conditions for the representative illustrated by a butterfly diagram.
\end{proof}

\subsection{The geometry of Assumption \ref{assume}}
Since torus fixed points on $\Ch(\DD)$ are in bijection with BCTs of the corresponding table-with-margins, we obtained

\begin{proposition} \label{prop:existsfix}
Assumption \ref{assume} on $\DD$ is equivalent with $\Ch(\DD)$ having at least one torus fixed point. \qed
\end{proposition}

\noindent As a consequence we have that the dimension formula \eqref{eq:dimensionformula} is always non-negative.

\subsection{Specializations of Chern roots of tautological bundles at fixed points}\label{sec:fixrest}

Let $\uu_j$ denote $\C$ acted upon by the $j$'s $\C^\times$-component of $\A\subset \T$, or its class in $K_{\C^{\times}}(\pt)\subset K_{\T}(\pt)$. Its first Chern class will be denoted by $u_j\in H^*_{\C^\times}(\pt) \subset H^*_{\T}(\pt)$. Let $\hb$ denote $\C$ acted upon by $\C^\times_{\h}$, or its class in $K_{\C^{\times}}(\pt)\subset K_{\T}(\pt)$. Its first Chern class will be denoted by $\h\in H^*_{\C^\times_{\h}}(\pt)\subset H^*_{\T}(\pt)$.
For a $\T$ fixed point $f\in \Ch(\DD)$ we have the restriction homomorphisms
\begin{align*}
\Loc^{K}_f: K_{\T}(\Ch(\DD)) & \to K_{\T}(f)=\C[\uu_1^{\pm 1},\uu_2^{\pm 1},\ldots,\uu_n^{\pm 1},\hb^{\pm 1}] \\
\Loc_f: H_{\T}^*(\Ch(\DD)) & \to H_{\T}^*(f)=\C[u_1,u_2,\ldots,u_n,\h]
\end{align*}
induced by the $\T$ equivariant inclusion of $f$ into $\Ch(\DD)$. 

Consider the butterfly diagram of a fixed point, and the non-empty butterfly associated with the D5 brane $U_j$. We will associate a {\em base level} to this butterfly. If there are dots below $U_j^-$ then the vertical position of the top dot in that column (that is, the dot representing $a(\C_{U_j})$) is the base level. If there are no dots under $U_j^-$ then one position below the bottom dot under $U_j^+$ is the base level. Associate the monomial $\uu_j\hb^s$ ($\in K_{\T}(\pt)$)  to every dot in this butterfly that is $s$ units above the base level. Figure~\ref{fig:butterflyMONS} is an example of how to associate monomials to the dots of Figure~\ref{fig:butterly}.

\newcommand{\ufha}{\scalebox{.6}[.6]{${\uu}_4\hb^2$}}
\newcommand{\ufhb}{\scalebox{.6}[.6]{${\uu}_4\hb$}}
\newcommand{\ufhc}{\scalebox{.6}[.6]{${\uu}_4$}}
\newcommand{\ufhd}{\scalebox{.6}[.6]{${\uu}_4\hb^{-1}$}}
\newcommand{\ufhe}{\scalebox{.6}[.6]{${\uu}_4\hb^{-2}$}}
\newcommand{\ufhf}{\scalebox{.6}[.6]{${\uu}_4\hb^{-3}$}}
\newcommand{\ufhg}{\scalebox{.6}[.6]{${\uu}_4\hb^{-4}$}}
\newcommand{\ufhh}{\scalebox{.6}[.6]{${\uu}_4\hb^{-5}$}}
\newcommand{\ufhi}{\scalebox{.6}[.6]{${\uu}_4\hb^{-6}$}}
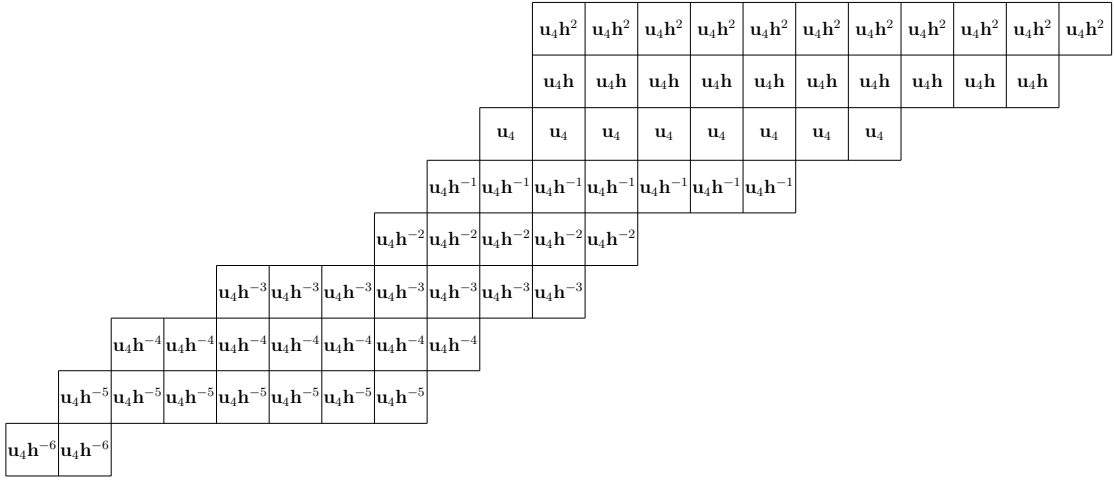
\begin{figure}
\[
\begin{tikzpicture}[scale=.7]
\node at  (0,-3) {\ufhi};
\node at (1,-3) {\ufhi}; \node at (1,-2) {\ufhh};
\node at (2,-2)  {\ufhh}; \node at (2,-1)  {\ufhg};
\node at (3,-2)  {\ufhh}; \node at (3,-1)  {\ufhg};
\node at (4,0)  {\ufhf}; \node at (4,-1)  {\ufhg}; \node at (4,-2)  {\ufhh};
\node at (5,0)  {\ufhf}; \node at (5,-1)  {\ufhg}; \node at (5,-2)  {\ufhh};
\node at (6,0)  {\ufhf}; \node at (6,-1)  {\ufhg}; \node at (6,-2)  {\ufhh};
\node at (7,1)  {\ufhe}; \node at (7,0)  {\ufhf}; \node at (7,-1)  {\ufhg}; \node at (7,-2)  {\ufhh};
\node at (8,2)  {\ufhd}; \node at (8,1)  {\ufhe}; \node at (8,0)  {\ufhf}; \node at (8,-1)  {\ufhg};
\node at (9,3)  {\ufhc}; \node at (9,2)  {\ufhd}; \node at (9,1)  {\ufhe}; \node at (9,0)  {\ufhf};
\node at (10,0)  {\ufhf}; \node at (10,1)  {\ufhe}; \node at (10,2)  {\ufhd}; \node at (10,3)  {\ufhc}; \node at (10,4)  {\ufhb}; \node at (10,5)  {\ufha}; 
\node at (11,1)  {\ufhe}; \node at (11,2)  {\ufhd}; \node at (11,3)  {\ufhc}; \node at (11,4)  {\ufhb}; \node at (11,5)  {\ufha}; 
\node at (12,2)  {\ufhd}; \node at (12,3)  {\ufhc}; \node at (12,4)  {\ufhb}; \node at (12,5)  {\ufha}; 
\node at (13,2)  {\ufhd}; \node at (13,3)  {\ufhc}; \node at (13,4)  {\ufhb}; \node at (13,5)  {\ufha}; 
\node at (14,2)  {\ufhd}; \node at (14,3)  {\ufhc}; \node at (14,4)  {\ufhb}; \node at (14,5)  {\ufha}; 
\node at (15,3)  {\ufhc}; \node at (15,4)  {\ufhb}; \node at (15,5)  {\ufha}; 
\node at (16,3)  {\ufhc}; \node at (16,4)  {\ufhb}; \node at (16,5)  {\ufha}; 
\node at (17,4)  {\ufhb}; \node at (17,5)  {\ufha}; 
\node at (18,4)  {\ufhb}; \node at (18,5)  {\ufha}; 
\node at (19,4)  {\ufhb}; \node at (19,5)  {\ufha}; 
\node at (20,5)  {\ufha}; 

\draw [ultra thin] (-0.5,-3.5) --(1.5,-3.5);\draw [ultra thin] (-0.5,-2.5) --(7.5,-2.5);\draw [ultra thin] (0.5,-1.5) --(8.5,-1.5);\draw [ultra thin] (1.5,-0.5) --(10.5,-0.5); \draw [ultra thin] (3.5,0.5) --(11.5,0.5);\draw [ultra thin] (6.5,1.5) --(14.5,1.5);\draw [ultra thin] (7.5,2.5) --(16.5,2.5);\draw [ultra thin] (8.5,3.5) --(19.5,3.5);\draw [ultra thin] (9.5,4.5) --(20.5,4.5);\draw [ultra thin] (9.5,5.5) --(20.5,5.5);

\draw [ultra thin] (-0.5,-3.5) --(-0.5,-2.5);\draw [ultra thin] (0.5,-3.5) --(0.5,-1.5);\draw [ultra thin] (1.5,-3.5) --(1.5,-.5);\draw [ultra thin] (2.5,-2.5) --(2.5,-.5);\draw [ultra thin] (3.5,-2.5) --(3.5,0.5);\draw [ultra thin] (4.5,-2.5) --(4.5,0.5);\draw [ultra thin] (5.5,-2.5) --(5.5,0.5);\draw [ultra thin] (6.5,-2.5) --(6.5,1.5);\draw [ultra thin] (7.5,-2.5) --(7.5,2.5);\draw [ultra thin] (8.5,-1.5) --(8.5,3.5);\draw [ultra thin] (9.5,-0.5) --(9.5,5.5);\draw [ultra thin] (10.5,-0.5) --(10.5,5.5);\draw [ultra thin] (11.5,0.5) --(11.5,5.5);\draw [ultra thin] (12.5,1.5) --(12.5,5.5);\draw [ultra thin] (13.5,1.5) --(13.5,5.5);\draw [ultra thin] (14.5,1.5) --(14.5,5.5);\draw [ultra thin] (15.5,2.5) --(15.5,5.5);\draw [ultra thin] (16.5,2.5) --(16.5,5.5);\draw [ultra thin] (17.5,3.5) --(17.5,5.5);\draw [ultra thin] (18.5,3.5) --(18.5,5.5);\draw [ultra thin] (19.5,3.5) --(19.5,5.5);\draw [ultra thin] (20.5,4.5) --(20.5,5.5);
\end{tikzpicture}
\]
\caption{The monomials associated to the basis vectors in Figure \ref{fig:butterly}.}\label{fig:butterflyMONS}
\end{figure}

Recall that tautological vector bundles $\xi_X$ of rank $\mult_X$ are associated with the D3 branes. The Grothendieck roots, a.k.a. K theoretic Chern roots, of a bundle $\xi$ are the virtual classes in $K_{\T}(\pt)$ whose sum is $\xi$.

\begin{theorem}
Consider the butterfly diagram of a fixed point $f$, together with the associated monomials. The restriction map $K_{\T}(\Ch(\DD))\to K_{\T}(f)$ maps the Grothendieck roots of $\xi_X$ to the monomials directly below the D3 brane $X$.
\end{theorem}

The cohomological restriction map then follows: the Chern roots of $\xi_X$ are mapped to the ``logarithms'' of the named monomials. By logarithm we mean e.g. $\ln( \uu_4 \hb^{-3})=u_4-3\h$.

\begin{proof}
The butterfly diagrams are set up in such a way that an $A, B, C, D$ arrow that maps from height $l_1$ to height $l_2$ represents a map whose $\hb$ degree is $l_1-l_2$. The base level, ie. the level where the action is $\uu_i\hb^0$, is determined using the $\hb$ degree of $a$ or $b$. 
\end{proof}

\begin{example} \rm \label{ex:P1b}
Continuing Example \ref{ex:P1a}, consider the two torus fixed point codes 
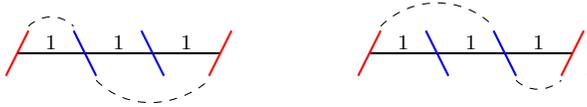
\[
\begin{tikzpicture}[baseline=0pt,scale=.3]
\draw[thick] (0,1)--(9,1) ;
\draw[thick,red] (-.5,0)--(.5,2);
\draw[thick,blue] (3.5,0)--(2.5,2);
\draw[thick,blue] (6.5,0)--(5.5,2);
\draw[thick,red] (8.5,0)--(9.5,2);
\draw [dashed, black](0.5,2.2) to [out=45,in=-225] (2.5,2.2);
\draw [dashed, black](3.5,-.2) to [out=-45,in=220] (8.5,-.2);
\node at (1.5,1.5) {\tiny 1};\node at (4.5,1.5) {\tiny 1};\node at (7.5,1.5) {\tiny 1};
\end{tikzpicture}
\qquad\qquad
\begin{tikzpicture}[baseline=0pt,scale=.3]
\draw[thick] (0,1)--(9,1) ;
\draw[thick,red] (-.5,0)--(.5,2);
\draw[thick,blue] (3.5,0)--(2.5,2);
\draw[thick,blue] (6.5,0)--(5.5,2);
\draw[thick,red] (8.5,0)--(9.5,2);
\draw [dashed, black](0.5,2.2) to [out=45,in=-225] (5.5,2.2);
\draw [dashed, black](6.5,-.2) to [out=-45,in=220] (8.5,-.2);
\node at (1.5,1.5) {\tiny 1};\node at (4.5,1.5) {\tiny 1};\node at (7.5,1.5) {\tiny 1};
\end{tikzpicture}.
\]
The corresponding decorated butterfly diagrams are 
$\begin{tabular}{|c|c|c|}
\hline
$\uu_1$ & $\uu_1$ & $\uu_1$ \\
\hline
\end{tabular}$ and $\begin{tabular}{|c|c|c|}
\hline
$\uu_2$ & $\uu_2$ & $\uu_2$ \\
\hline
\end{tabular}
$.
Therefore, for the restriction maps in K theory we have
\begin{equation*}\label{eq:themaps}
\alpha \mapsto \uu_1, \beta\mapsto \uu_1, \gamma\mapsto \uu_1,
\qquad\text{and}\qquad
\alpha \mapsto \uu_2, \beta\mapsto \uu_2, \gamma\mapsto\uu_2.
\end{equation*}
Making these substitutions into the formula for $T(\Ch(\DD))$ (see Example \ref{ex:P1a}) we find the tangent spaces at the two fixed points to be
\begin{equation}\label{eq:P1fix}
\frac{\uu_1}{\uu_2}+\frac{\uu_2}{\uu_1}\hb
\qquad\qquad \text{and} \qquad\qquad
\frac{\uu_2}{\uu_1}+\frac{\uu_1}{\uu_2}\hb.
\end{equation}
Hence, $\Ch(\DD)$ is a holomorphic symplectic manifold of dimension 2, with a $(\C^\times)^2\times \C^\times$ action having two fixed points, where the tangent spaces are \eqref{eq:P1fix}. The reader by now probably have---correctly---guessed that this variety is $T^*\PPP^1$, cf. Section \ref{sec:QB}.
\end{example}

\begin{remark} \rm 
The expression for the tangent bundle of $\Ch(\DD)$ in terms of the tautological bundles involved negative signs, cf. Remark \ref{rem:negatives}.
Yet, the restriction of the tangent bundle to a fixed point, an element of $K_{\T}(\pt)$, is a representation of $\T$---not a virtual representation. Hence it is a Laurent polynomial in $\uu_j$'s and $\hb$ {\em with non-negative coefficients}. That is, the disappearance of negative signs going from \eqref{eq:TP1} to \eqref{eq:P1fix} is a general phenomenon, and hence a useful reality check in calculations.
\end{remark}

\subsection{Torus fixed point matching under Hanany-Witten transition}\label{sec:HWfix}

In Section \ref{sec:HWproof} we described an isomorphism between $\Ch(\DD)$ and $\Ch(\tDD)$ if $\DD$ and $\tDD$ are obtained from each other by a Hanany-Witten transistion, as in  \eqref{fig:HW2}. In turn, we have a bijection between the torus fixed points of $\Ch(\DD)$ and $\Ch(\tDD)$. This bijection turns out to be very natural in terms of two of our combinatorial codes of fixed points---we will describe these combinatorial bijections now.

The bijection of fixed points on tie diagrams is illustrated in the Riedemeister-III-looking Figure \ref{fig:HWfixpoints}. It is instructive to verify the $\mult_2+\tilde{\mult}_2=\mult_1+\mult_3+1$ rule (see Sections \ref{sec:HW}, \ref{sec:HWproof}) of multiplicities.
As a consequence, Figure \ref{fig:HWfixpoints} provides a proof of Lemma \ref{lem:noneg}.

\begin{figure}
\[
\begin{tikzpicture}[scale=.5]
\draw [thick] (0,1) --(2,1);  
\draw [thick, blue] (3,0) -- (2,2);
\draw[thick] (3,1)--(5,1);
\draw[thick,red] (5,0)--(6,2);
\draw[thick] (6,1)--(8,1);
\draw [dashed](2,2) to [out=120,in=0] (0,3) ;
\draw [dashed](2,2) to [out=120,in=0] (0,3.2) node [left] {$A$} ;
\draw [dashed](2,2) to [out=120,in=0] (0,3.4) ;
\draw [dashed](5,0) to [out=240,in=0] (0,-2)  ;
\draw [dashed](5,0) to [out=240,in=0] (0,-2.2)  node [left] {$B$} ;
\draw [dashed](5,0) to [out=240,in=0] (0,-2.4);
\draw [dashed](3,0) to [out=-60,in=180] (8,-2);
\draw [dashed](3,0) to [out=-60,in=180] (8,-2.2) node [right]{$C$};
\draw [dashed](3,0) to [out=-60,in=180] (8,-2.4) ;
\draw [dashed](6,2) to [out=60,in=180] (8,3) ;
\draw [dashed](6,2) to [out=60,in=180] (8,3.2) node [right]{$D$};
\draw [dashed](6,2) to [out=60,in=180] (8,3.4);
\draw[dashed] (3,0) to [out=-40,in=-140](5,0);
\node at (4,0.2) {$E$};
\end{tikzpicture}
\qquad
\begin{tikzpicture}[scale=.5, baseline=-30pt]
\draw[ultra thick, <->] (0,1)--(2,1) node[above]{HW transition} -- (4,1);
\draw[ultra thick, <->] (0,1)--(2,1) node[below]{on fixpoints} -- (4,1);
\end{tikzpicture}
\qquad
\begin{tikzpicture}[scale=.5]
\draw [thick] (0,1) --(2,1);  
\draw [thick, red] (2,0) -- (3,2);
\draw[thick] (3,1)--(5,1);
\draw[thick, blue] (6,0)--(5,2);
\draw[thick] (6,1)--(8,1);
\draw [dashed](2,0) to [out=240,in=0] (0,-.7);
\draw [dashed](2,0) to [out=240,in=0] (0,-.9)  node [left] {$B$};
\draw [dashed](2,0) to [out=240,in=0] (0,-1.1) ;
\draw [dashed](5,2) to [out=120,in=0] (0,4)  ;
\draw [dashed](5,2) to [out=120,in=0] (0,4.2)  node [left] {$A$};
\draw [dashed](5,2) to [out=120,in=0] (0,4.4) ;
\draw [dashed](3,2) to [out=60,in=180] (8,4);
\draw [dashed](3,2) to [out=60,in=180] (8,4.2) node [right]{$D$};
\draw [dashed](3,2) to [out=60,in=180] (8,4.4) ;
\draw [dashed](6,0) to [out=-60,in=180] (8,-.7) ;
\draw [dashed](6,0) to [out=-60,in=180] (8,-.9)node [right]{$C$};
\draw [dashed](6,0) to [out=-60,in=180] (8,-1.1);
\draw[dashed] (3,2) to [out=40,in=140](5,2);
\node at (4,1.8) {$\lnot E$};
\end{tikzpicture}
\]
\caption{The bijection induced by a Hanany-Witten transition on tie diagrams of fixed points. The signs $E, \lnot E$ mean that if there is a tie at $E$ then there is no tie at $\lnot E$, and vice versa. Ties not ending in at least one of the 5-branes that are interchanged are not affected.} \label{fig:HWfixpoints}
\end{figure}
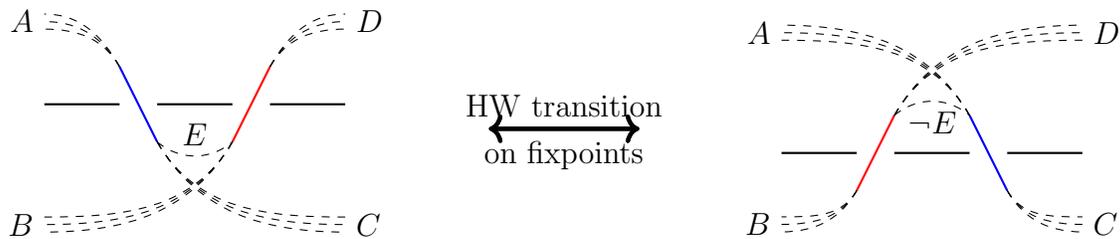

The bijection on BCT codes of fixed points is even simpler. Recall that the tables-with-margins associated with two brane diagrams differing only by a Hanany-Witten transition are essentially identical, the only difference is that the separating line is changed ``at one corner,'' see Figure~\ref{fig:margins}(a),(c). Hence a BCT for one such table is also a BCT for the other, the separating line plays no role in the BCT property. Hence the bijection between the BCT codes is: {\em identical} fillings of the tables-with-margins.

The combinatorial bijection between butterfly diagrams of fixed points can also be described, but it is less illuminating than the two we presented.

\subsection{Torus fixed point matching under 3d mirror symmetry}

\begin{theorem}
Let $\DD$ and $\DD'$ be 3d mirror dual brane diagrams.  There is a natural bijection between the torus fixed points of $\Ch(\DD)$ and $\Ch(\DD')$.
\end{theorem}

\begin{proof}
A tie diagram for $\DD$ reflected across the line of the D3 branes is a tie diagram of $\DD'$, and vice versa. %This proves the theorem. 
\end{proof}

Let us describe the same bijection in the language of BCTs. Recall the relation between the table-with-margins for $\DD$ and that of $\DD'$ from Section \ref{sec:3d}: The table of $\DD'$ is obtained by transposing the table of $\DD$, together with the separating line, and setting the new margin vectors $r', c'$ to be 
$c_i'=n-r_i, r_i'=m-c_i$, see Figure~\ref{fig:margins}(a),(b). Thus the bijection between the BCTs is obtained by transposing and {\em negating} (switching 1 and 0) the BCTs:
\[
M_{i,j}(\DD')=1-M_{j,i}(\DD).
\]

Since the tie diagram of a torus fixed point in $\Ch(\DD)$ and that of the corresponding torus fixed point in the 3d mirror dual are (essentially) the {\em same}, it is worth decorating this common tie diagram with the monomials that encode both restriction maps $K_{\T}(\Ch(\DD)) \to K_{\T}(\pt)$ and $K_{\T}(\Ch(\DD')) \to K_{\T}(\pt)$, see Figure \ref{piroskek}.

\newcommand{\vehe}{\scalebox{.6}[.6]{${\vv}_1\hb$}}
\newcommand{\vehk}{\scalebox{.6}[.6]{${\vv}_1\hb^2$}}
\newcommand{\vkhme}{\scalebox{.6}[.6]{${\vv}_2\hb^{-1}$}}
\newcommand{\vk}{\scalebox{.6}[.6]{${\vv}_2$}}
\newcommand{\vkh}{\scalebox{.6}[.6]{${\vv}_2\hb$}}
\newcommand{\vhhmk}{\scalebox{.6}[.6]{${\vv}_3\hb^{-2}$}}
\newcommand{\vhhme}{\scalebox{.6}[.6]{${\vv}_3\hb^{-1}$}}
\newcommand{\vh}{\scalebox{.6}[.6]{${\vv}_3$}}
\newcommand{\vnhmk}{\scalebox{.6}[.6]{${\vv}_4\hb^{-2}$}}
\newcommand{\vnhme}{\scalebox{.6}[.6]{${\vv}_4\hb^{-1}$}}
\newcommand{\vn}{\scalebox{.6}[.6]{${\vv}_4$}}
\newcommand{\ue}{\scalebox{.6}[.6]{${\uu}_1$}}
\newcommand{\uk}{\scalebox{.6}[.6]{${\uu}_2$}}
\newcommand{\ukh}{\scalebox{.6}[.6]{${\uu}_2\hb$}}
\newcommand{\ukhk}{\scalebox{.6}[.6]{${\uu}_2\hb^2$}}
\newcommand{\un}{\scalebox{.6}[.6]{${\uu}_4$}}
\newcommand{\uo}{\scalebox{.6}[.6]{${\uu}_5$}}
\newcommand{\uohme}{\scalebox{.6}[.6]{${\uu}_5\hb^{-1}$}}
\newcommand{\uohmk}{\scalebox{.6}[.6]{${\uu}_5\hb^{-2}$}}
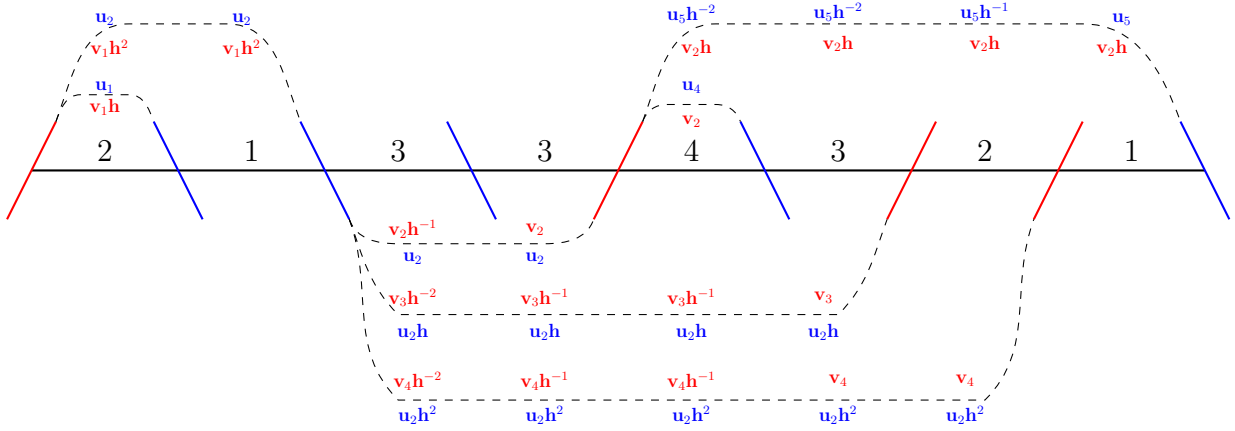
\begin{figure}
\[
\begin{tikzpicture}[scale=.65]
\draw [thick] (0,0)--(24,0);
\draw [red, thick] (-.5,-1) --(0.5,1); 
\draw [blue, thick] (3.5,-1) --(2.5,1); 
\draw [blue,thick] (6.5,-1) --(5.5,1); 
\draw [blue,thick] (9.5,-1) --(8.5,1); 
\draw [red,thick] (11.5,-1) --(12.5,1); 
\draw [blue,thick] (15.5,-1) --(14.5,1); 
\draw [red,thick] (17.5,-1) --(18.5,1); 
\draw [red,thick] (20.5,-1) --(21.5,1); 
\draw [blue,thick] (24.5,-1) --(23.5,1); 

\node at (1.5,.4) {2}; \node at (4.5,.4) {1}; \node at (7.5,.4) {3}; \node at (10.5,.4) {3}; 
\node at (13.5,.4) {4}; \node at (16.5,.4) {3}; \node at (19.5,.4) {2}; \node at (22.5,.4) {1}; 

\draw [dashed](0.5,1) to [out=70,in=180] (1,1.55) to (2,1.55) to [out=0,in=110] (2.5,1);
\node[red] at (1.5,1.3) {\vehe} ;
\node[blue] at (1.5,1.7) {\ue} ;

\draw [dashed](0.5,1) to [out=70,in=180] (2,3) to (4,3) to [out=0,in=110] (5.5,1);
\node[red] at (1.6,2.5) {\vehk} ;
\node[blue] at (1.5,3.1) {\uk} ;
\node[red] at (4.3,2.5) {\vehk} ;
\node[blue] at (4.3,3.1) {\uk} ;

\draw [dashed](6.5,-1) to [out=-70] (7.5,-4.7) to (19.5,-4.7) to [in=-110] (20.5,-1);
\node[red] at (7.9,-4.3) {\vnhmk} ;
\node[blue] at (7.9,-5) {\ukhk} ;
\node[red] at (10.5,-4.3) {\vnhme} ;
\node[blue] at (10.5,-5) {\ukhk} ;
\node[red] at (13.5,-4.3) {\vnhme} ;
\node[blue] at (13.5,-5) {\ukhk} ;
\node[red] at (16.5,-4.3) {\vn} ;
\node[blue] at (16.5,-5) {\ukhk} ;
\node[red] at (19.1,-4.3) {\vn} ;
\node[blue] at (19.1,-5) {\ukhk} ;

\draw [dashed](6.5,-1) to [out=-70]  (7.5,-2.95) to (16.5,-2.95) to [in=-110] (17.5,-1);
\node[red] at (7.8,-2.6) {\vhhmk} ;
\node[blue] at (7.8,-3.3) {\ukh} ;
\node[red] at (10.5,-2.6) {\vhhme} ;
\node[blue] at (10.5,-3.3) {\ukh} ;
\node[red] at (13.5,-2.6) {\vhhme} ;
\node[blue] at (13.5,-3.3) {\ukh} ;
\node[red] at (16.2,-2.6) {\vh} ;
\node[blue] at (16.2,-3.3) {\ukh} ;

\draw [dashed](6.5,-1) to [out=-70,in=180] (7.5,-1.5) to (10.5,-1.5) to [out=0,in=-110] (11.5,-1);
\node[red] at (7.8,-1.2) {\vkhme} ;
\node[blue] at (7.8,-1.8) {\uk} ;
\node[red] at (10.3,-1.2) {\vk} ;
\node[blue] at (10.3,-1.8) {\uk} ;

\draw [dashed](12.5,1) to [out=70,in=180]  (13,1.35) to (14,1.35) to [out=0,in=110] (14.5,1);
\node[red] at (13.5,1) {\vk} ;
\node[blue] at (13.5,1.7) {\un} ;

\draw [dashed](12.5,1) to [out=70,in=180] (14,3) to (21.5,3) to [out=0,in=110] (23.5,1);
\node[red] at (13.6,2.5) {\vkh} ;
\node[blue] at (13.5,3.2) {\uohmk} ;
\node[red] at (16.5,2.6) {\vkh} ;
\node[blue] at (16.5,3.25) {\uohmk} ;
\node[red] at (19.5,2.6) {\vkh} ;
\node[blue] at (19.5,3.25) {\uohme} ;
\node[red] at (22.1,2.5) {\vkh};
\node[blue] at (22.3,3.1) {\uo} ;
%\node[red,draw,circle,minimum size=.4cm,inner sep=0pt] at (0.3,-0.4) {\small $v_1$};
%\node[red,draw,circle,minimum size=.4cm,inner sep=0pt] at (12.3,-0.4) {\small $v_2$};
%\node[red,draw,circle,minimum size=.4cm,inner sep=0pt] at (18.3,-0.4) {\small$v_3$};
%\node[red,draw,circle,minimum size=.4cm,inner sep=0pt] at (21.3,-0.4) {\small$v_4$};
%\node[blue,draw,circle,minimum size=.4cm,inner sep=0pt] at (3.3,0.4) {\small$u_1$};
%\node[blue,draw,circle,minimum size=.4cm,inner sep=0pt] at (6.3,0.4) {\small$u_2$};
%\node[blue,draw,circle,minimum size=.4cm,inner sep=0pt] at (9.3,0.4) {\small$u_3$};
%\node[blue,draw,circle,minimum size=.4cm,inner sep=0pt] at (15.3,0.4) {\small$u_4$};
%\node[blue,draw,circle,minimum size=.4cm,inner sep=0pt] at (24.3,0.4) {\small$u_5$};
\end{tikzpicture}
\]
\caption{The blue decorations on the tie diagram indicate the monomials the Grothendieck roots of tautological bundles are mapped under the restriction map $K_{\T}(\Ch(\DD))\to K_{\T}(\pt)$. The monomials in red are analogous for the 3d mirror of $\DD$.} \label{piroskek}
\end{figure}

\section{Quiver varieties as bow varieties}

\subsection{Nakajima quiver varieties} \cite{NakajimaDuke, NakajimaBook}
First we recall the definition of the type A Nakajima quiver variety associated with dimension vector $v\in \Z_{\geq 0}^n$ and framing vector $w\in \Z_{\geq 0}^n$, equivalently, with the quiver $Q=Q(v,w)$
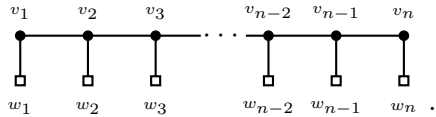
\[
\begin{tikzpicture}[baseline=-45pt, scale=.6]
 \draw[thick] (0,-1) -- (4,-1);  \draw[thick] (5,-1) -- (8.5,-1); 
 
 \draw[fill] (0,-1) circle (3pt);  
 \draw[fill] (1.5,-1) circle (3pt);  
 \draw[fill] (3,-1) circle (3pt);  
    \node at (4.5,-1) {$\cdots$};  
 \draw[fill] (5.5,-1) circle (3pt); 
 \draw[fill] (7,-1) circle (3pt);  
 \draw[fill] (8.5,-1) circle (3pt);  
 
 \draw[thick] (-0.1,-2.1) rectangle (0.1,-1.9);
 \draw[thick] (1.4,-2.1) rectangle (1.6,-1.9);
 \draw[thick] (2.9,-2.1) rectangle (3.1,-1.9);
 \draw[thick] (5.4,-2.1) rectangle (5.6,-1.9);
 \draw[thick] (6.9,-2.1) rectangle (7.1,-1.9);
 \draw[thick] (8.4,-2.1) rectangle (8.6,-1.9);

 \draw[thick] (0,-1) -- (0,-1.9);
 \draw[thick] (1.5,-1) -- (1.5,-1.9);
 \draw[thick] (3,-1) -- (3,-1.9);
 \draw[thick] (5.5,-1) -- (5.5,-1.9);
 \draw[thick] (7,-1) -- (7,-1.9);
 \draw[thick] (8.5,-1) -- (8.5,-1.9);

 \node at (0,-.5) {$\scriptscriptstyle v_1$} ;
 \node at (1.5,-.5) {$\scriptscriptstyle v_2$} ;
 \node at (3,-.5) {$\scriptscriptstyle v_3$} ;
 \node at (5.5,-.5) {$\scriptscriptstyle v_{n-2}$} ;
 \node at (7,-.5) {$\scriptscriptstyle v_{n-1}$} ;
 \node at (8.5,-.5) {$\scriptscriptstyle v_n$} ;
 
\node at (0,-2.6) {$\scriptscriptstyle w_1$} ;   \node at (1.5,-2.6) {$\scriptscriptstyle w_2$} ;   \node at (3,-2.6) {$\scriptscriptstyle w_3$} ;   \node at (5.5,-2.6) {$\scriptscriptstyle w_{n-2}$} ;   \node at (7,-2.6) {$\scriptscriptstyle w_{n-1}$} ;   \node at (8.5,-2.6) {$\scriptscriptstyle w_n$} ;
    
\end{tikzpicture}.
\]
\noindent Let 
\[
R=\bigoplus_{i=1}^{n-1}\Hom(\C^{v_i},\C^{v_{i+1}})\oplus \bigoplus_{i=1}^n \Hom(\C^{w_i},\C^{v_i})
\] 
and let the moment map
\[
\mu: R\oplus \hb R^\vee\to \bigoplus_{i=1}^n \hb \End(\C^{v_i})
\]
be defined by $\mu=[a,b]+kl$ for $(a,k)\in R$, $(b,l)\in R^\vee$. The quiver variety $\Naka(Q)=\Naka(v,w)$ is defined to be
\[
\mu^{-1}(0)^{ss} \Big/  \bigtimes_{i=1}^n \!\GL_{v_i}(\C),
\]
where an element $(a,k,b,l)$ is semistable if the images $k_i(\C^{w_i})$ generate $\oplus \C^{v_i}$ via the action of $a,b$.

Quiver varieties are smooth and they come with the following structures: (a) rank $v_i$ tautological bundles associated with the vertexes in the top row, (b) $\A=(\C^\times)^{\sum w_i}$ torus action (in fact a larger group acts as well, but we will not need it), (c) an extra $\C^\times_{\h}$ action on $R^\vee$ (as indicated above); set $\T=\A\times \C^{\times}_{\h}$, (d) a holomorphic symplectic structure. 

\smallskip

Let $\cF_{v_1,v_2,\ldots,v_{n+1}}$ denote the partial flag variety parametrizing nested subspaces 
\[
0\subset {\mathcal V}_1^{v_1} \subset {\mathcal V}_2^{v_2} \subset \ldots \subset {\mathcal V}_{n-1}^{v_{n-1}} \subset {\mathcal V}_n^{v_n}\subset \C^{v_{n+1}}.
\]
The variety $T^*\!\cF_{v_1,v_2,\ldots,v_{n+1}}$ is equipped with a standard holomorphic symplectic structure and a $\T=(\C^\times)^{v_{n+1}}\times \C^{\times}_{\h}$ action ($\C^{\times}_{\h}$ scales the fiber directions). It is easy to verify that $T^*\!\cF_{v_1,v_2,\ldots,v_{n+1}}$ with all these structures is the same as the quiver variety associated with 
%the quiver
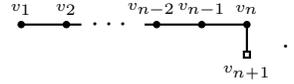
\begin{tikzpicture}[baseline=-20, scale=.4]
 \draw[thick] (1.5,-1) -- (3.5,-1);  \draw[thick] (5.6,-1) -- (9,-1); 
  \draw[fill] (1.5,-1) circle (3pt);  
 \draw[fill] (3,-1) circle (3pt);  
    \node at (4.5,-1) {$\cdots$};  
 \draw[fill] (6,-1) circle (3pt); 
 \draw[fill] (7.5,-1) circle (3pt);  
 \draw[fill] (9,-1) circle (3pt);  

 \draw[thick] (8.9,-2.1) rectangle (9.1,-1.9);
 \draw[thick] (9,-1) -- (9,-1.9);
 \node at (1.5,-.5) {$\scriptscriptstyle v_1$} ;
 \node at (3,-.5) {$\scriptscriptstyle v_2$} ;
 \node at (6,-.5) {$\scriptscriptstyle v_{n-2}\ $} ;
 \node at (7.5,-.5) {$\scriptscriptstyle v_{n-1}$} ;
 \node at (9,-.5) {$\scriptscriptstyle v_{n}$} ;
 \node at (9,-2.6) {$\scriptscriptstyle v_{n+1}$} ;
\end{tikzpicture}.

\begin{remark}\rm
The Nakajima quiver variety $\Naka\left(
\begin{tikzpicture}[baseline=-20, scale=.4]
 \draw[thick] (1.5,-1) -- (3.5,-1);  \draw[thick] (5.6,-1) -- (9,-1); 
  \draw[fill] (1.5,-1) circle (3pt);  
 \draw[fill] (3,-1) circle (3pt);  
    \node at (4.5,-1) {$\cdots$};  
 \draw[fill] (6,-1) circle (3pt); 
 \draw[fill] (7.5,-1) circle (3pt);  
 \draw[fill] (9,-1) circle (3pt);  

 \draw[thick] (1.4,-2.1) rectangle (1.6,-1.9);
 \draw[thick] (1.5,-1) -- (1.5,-1.9);
 \node at (1.5,-.5) {$\scriptscriptstyle v_n$} ;
 \node at (3,-.5) {$\scriptscriptstyle v_{n-1}$} ;
 \node at (6,-.5) {$\scriptscriptstyle v_{3}\ $} ;
 \node at (7.5,-.5) {$\scriptscriptstyle v_{2}$} ;
 \node at (9,-.5) {$\scriptscriptstyle v_{1}$} ;
 \node at (1.5,-2.6) {$\scriptscriptstyle v_{n+1}$} ;
\end{tikzpicture}\right)$
is also isomorphic with $T^*\!\cF_{v_1,v_2,\ldots,v_{n+1}}$, but here $\C^\times_{\h}$ acts in the tangent directions instead of the fiber directions; so the torus action needs to be reparametrized for an identification with  $T^*\!\cF_{v_1,v_2,\ldots,v_{n+1}}$ and its ``usual'' structures.
\end{remark}

\subsection{Quiver varieties as bow varieties} \label{sec:QB}
Starting with a quiver $Q$, replacing the $i$'th 
 \begin{tikzpicture}[baseline=-25pt,scale=.5]
 \draw[fill] (3,-1) circle (3pt);  
 \draw[thick] (2.9,-2.1) rectangle (3.1,-1.9);
 \draw[thick] (3,-1) -- (3,-1.9);
 \node at (3.6,-1) {$\scriptscriptstyle v_i$} ;
 \node at (3.6,-2) {$\scriptscriptstyle w_i$} ;
\end{tikzpicture}
portion with 
\begin{tikzpicture}[baseline=5,scale=.3]
\draw[thick] (0,1)--(7,1) ;\draw[thick] (11,1)--(18,1) ;
\draw[thick,red] (-.5,0)--(.5,2);
\draw[thick,blue] (3.5,0)--(2.5,2);
\draw[thick,blue] (6.5,0)--(5.5,2);
\node at (9.5,1) {$\cdots$};
\draw[thick,blue] (12.5,0)--(11.5,2);
\draw[thick,blue] (15.5,0)--(14.5,2);
\draw[thick,red] (17.5,0)--(18.5,2);
\node at (1.5,1.5) {\tiny $v_i$};\node at (4.5,1.5) {\tiny $v_i$};\node at (6.7,1.5) {\tiny $v_i$};
\node at (11,1.5) {\tiny $v_i$};\node at (13.5,1.5) {\tiny $v_i$};\node at (16.5,1.5) {\tiny $v_i$};
\draw  [thin](2,0) to [out=290,in=90] (9,-2);
\draw [ thin](9,-2) to [out=90,in=250] (16,0);
\node at (9,-2.5) {\tiny $w_i$};
\end{tikzpicture}
and gluing these brane diagrams along the boundary NS5 branes we obtain a brane diagram $\DD$. 

\begin{theorem} \cite[Section 2.6]{NT} We have
$\Naka(Q)=\Ch(\DD)$, and the identification respects the torus actions and the holomorphic symplectic forms on the two sides. \qed
\end{theorem}
For example we have
\[
\begin{tabular}{lclcl}
& & $\Naka\Big(
 \begin{tikzpicture}[baseline=-25pt,scale=.5]
 \draw[thick] (0,-1) -- (3,-1); 
  \draw[fill] (0,-1) circle (3pt);  
 \draw[fill] (1.5,-1) circle (3pt);  
 \draw[fill] (3,-1) circle (3pt);  
  \draw[thick] (1.4,-2.1) rectangle (1.6,-1.9);
 \draw[thick] (2.9,-2.1) rectangle (3.1,-1.9);
 \draw[thick] (1.5,-1) -- (1.5,-1.9);
 \draw[thick] (3,-1) -- (3,-1.9);
 \node at (0,-.5) {$\scriptscriptstyle 3$} ;
 \node at (1.5,-.5) {$\scriptscriptstyle 2$} ;
 \node at (3,-.5) {$\scriptscriptstyle 5$} ;
\node at (1.5,-2.6) {$\scriptscriptstyle 4$} ;   \node at (3,-2.6) {$\scriptscriptstyle 2$} ;   
\end{tikzpicture}
\Big)$&$=$ & 
$\Ch(\ttt{{\fs}3{\fs}2\bs 2\bs 2\bs 2\bs 2{\fs}5\bs 5\bs 5{\fs}})$,
\\
 $T^*\!\Gr(2,5)$ & $=$ & 

 $\Naka\Big(
  \begin{tikzpicture}[baseline=-25pt,scale=.5]
 \draw[fill] (3,-1) circle (3pt);  
 \draw[thick] (2.9,-2.1) rectangle (3.1,-1.9);
 \draw[thick] (3,-1) -- (3,-1.9);
 \node at (3.45,-1) {$\scriptscriptstyle 2$} ;
 \node at (3.45,-2) {$\scriptscriptstyle 5$} ;   
\end{tikzpicture}
\Big)$
  & $=$
 & 
 $\Ch(\ttt{{\fs}2\bs 2\bs 2\bs 2\bs 2\bs 2{\fs}})$, \\
$T^*\!\cF_{1,2,4,6}$ & $=$ &
$\Naka\Big(
   \begin{tikzpicture}[baseline=-25pt,scale=.5]
 \draw[thick] (0,-1) -- (3,-1); 
  \draw[fill] (0,-1) circle (3pt);  
 \draw[fill] (1.5,-1) circle (3pt);  
 \draw[fill] (3,-1) circle (3pt);  
  \draw[thick] (2.9,-2.1) rectangle (3.1,-1.9);
 \draw[thick] (3,-1) -- (3,-1.9);
 \node at (0,-.5) {$\scriptscriptstyle 1$} ;
 \node at (1.5,-.5) {$\scriptscriptstyle 2$} ;
 \node at (3,-.5) {$\scriptscriptstyle 4$} ;
   \node at (3.45,-2) {$\scriptscriptstyle 6$} ;   
\end{tikzpicture}
\Big)$&$=$ &
$\Ch(\ttt{{\fs}1{\fs}2{\fs}4\bs 4\bs 4\bs 4\bs 4\bs 4\bs 4{\fs}})$,
\end{tabular}
\]
and the bow variety considered in Examples \ref{ex:P1a} and \ref{ex:P1b} is $T^*\!\PPP^1$.

Observe that if $\DD$ is obtained from a quiver $Q$ as described, then it is {\em co-balanced}: $\mult_{U^+}=\mult_{U^-}$ holds for all D5 branes $U$. Most brane diagrams are not co-balanced, nor are Hanany-Witten equivalent with a co-balanced one; we will present a numerical criterion for this in Section~\ref{sec:recognize}.

Since quiver varieties are bow varieties, and for bow varieties the torus fixed points are in bijection with certain BCTs, we have a BCT-formula for the number of torus fixed points of Nakajima quiver varieties.

\begin{corollary}\label{cor:chiQ}
For $v,w\in \Z_{\geq 0}^n$ we have
\[
\chi(Q(v,w))=
\# \BCT\left( \left(v_i-v_{i-1}+\sum_{j=1}^{i-1} w_j\right)_{i=1,\ldots,n+1}, (n^{w_1},(n-1)^{w_2},(n-2)^{w_3},\ldots,1^{w_n})\right)
\]
where the notation $a^b$ means $a,a,\ldots,a$ ($b$ times) and we set $v_0=v_{n+1}=0$.\qed
\end{corollary}
\noindent The special case of $v=(v_1\leq \ldots \leq v_n), w=(0,\ldots,0,v_{n+1})\in \Z_{\geq 0}^n$ ($v_n\leq v_{n+1}$) recovers the trivial
\begin{multline*}
\chi(\cF_{v_1,\ldots,v_n,v_{n+1}}) =
\chi(T^*\!\cF_{v_1,\ldots,v_n,v_{n+1}}) = \\
\#\BCT( (v_1,v_2-v_1,\ldots,v_{n}-v_{n-1},v_{n+1}-v_n),(1^{v_{n+1}})) =
\frac{v_{n+1}!}{ v_1! (v_2-v_1)! \ldots (v_{n+1}-v_n)!}.
\end{multline*}

Formulas or various generating functions are known for some special cases of $\#\BCT(r,c)$s. 
For example, according to Corollary \ref{cor:chiQ} and \cite[Cor. 5.5.11]{stanley2} we have
\[
\sum_{n=0}^{\infty} \chi\left(\Naka\left(
\begin{tikzpicture}[baseline=-20, scale=.4]
 \draw[thick] (1.5,-1) -- (5,-1);  \draw[thick] (7.1,-1) -- (9,-1); 
  \draw[fill] (1.5,-1) circle (3pt);  
 \draw[fill] (3,-1) circle (3pt);  
    \node at (6,-1) {$\cdots$};  
 \draw[fill] (4.5,-1) circle (3pt); 
 \draw[fill] (7.5,-1) circle (3pt);  
 \draw[fill] (9,-1) circle (3pt);  

 \draw[thick] (7.4,-2.1) rectangle (7.6,-1.9);
 \draw[thick] (7.5,-1) -- (7.5,-1.9);
 \node at (1.5,-.5) {$\scriptscriptstyle 2$} ;
 \node at (3,-.5) {$\scriptscriptstyle 4$} ;
 \node at (4.5,-.5) {$\scriptscriptstyle 6$} ;
 \node at (7.5,-.5) {$\scriptscriptstyle 2n-2$} ;
 \node at (9,-.5) {$\ \scriptscriptstyle n-1$} ;
 \node at (8.4,-2) {$\scriptscriptstyle n+1$} ;
\end{tikzpicture}
\right)\right)
\frac{x^n}{(n!)^2}= 
\sum_{n=0}^{\infty} \#\BCT((2^n),(2^n)) \frac{x^n}{(n!)^2}=
\frac{ e^{-\frac{x}{2}}}{\sqrt{1-x}}.
\]

\subsection{Recognizing quiver varieties from their margin vectors} \label{sec:recognize}
We saw that among the bow varieties those are the quiver varieties whose brane diagram is co-balanced. The following question remains: given a brane diagram, is it Hanany-Witten equivalent with a co-balanced brane diagram?

\begin{theorem} \label{thm:cobalanced}
The brane diagram $\DD$ is Hanany-Witten equivalent with a co-balanced one (ie. $\Ch(\DD)$ is HW isomorphic to a quiver variety) if and only if its margin vector $c$ is weakly decreasing.
\end{theorem}

\begin{proof}
The margin vector $c$ is constant in the HW equivalence class, and for co-balanced ones $c$ is indeed weakly decreasing. Conversely, if a weakly decreasing margin vector $c$ is given then (due the Assumption \ref{assume}) we can consider the ``separating line'' illustrated by the picture
\[
\begin{tikzpicture}[scale=.4]
\draw[thin] (0,2) -- (0,9) -- (10,9) -- (10,2)--(0,2); 
\draw[ultra thick]  (4,7) -- (4,6) -- (5,6)--(5,4) -- (6,4) ; 
\draw[blue,<->] (4.5,2.1) -- (4.5,5.9);
\draw[blue,<->] (5.5,2.1) -- (5.5,3.9);
 \node[blue] at (3.9,4) {\scriptsize $c_i$};
 \node[blue] at (4.5,9.34) {\scriptsize $c_i$};
 \node[blue] at (6.4,3) {\scriptsize $c_{i+1}$};
 \node[blue]  at (5.8,9.3) {\scriptsize $c_{i+1}$};
\end{tikzpicture}
\]
which in turn corresponds to a co-balanced brane diagram.
\end{proof}

If the 3d mirror dual of $\DD$ is co-balanced, we call it {\em balanced}, that is, if for all NS5 branes $V$ we have $\mult_{V^-}=\mult_{V^+}$. 

\begin{corollary} \label{cor:balanced}
The brane diagram $\DD$ is Hanany-Witten equivalent with a balanced one (ie. the mirror dual of $\Ch(\DD)$ is HW isomorphic to a quiver variety) if and only if its margin vector $r$ is weakly increasing. \qed
\end{corollary}

It also follows that in a HW equivalance class there is at most one co-balanced diagram, and at most one balanced one.

None of the brane diagrams 
\[
\ttt{{\fs}1{\fs}3{\fs}4\bs 3\bs 1\bs},\quad  \ttt{{\fs}2{\fs}3{\fs}4\bs 3\bs 1\bs},\quad  \ttt{{\fs}2{\fs}3{\fs}5\bs 3\bs 2\bs},\quad \ttt{{\fs}2{\fs}4{\fs}5\bs 3\bs 2\bs}
\] 
are HW equivalent with a balanced or co-balanced brane diagram (however, c.f. Proposition \ref{prop:HWTUTV}). The opposite case is discussed in the next section.

\subsection{3d mirror symmetry among quiver varieties}
Consider a quiver variety, and its brane diagram $\DD$. Taking the 3d mirror dual $\DD'$ of the brane diagram is not co-balanced, so {\em a priori} its associated bow variety is not a quiver variety. However, $\DD'$ may be Hanany-Witten equivalent to a co-balanced brane diagram. In this case we found two quiver varieties that are 3d mirror duals of each other. For example, $T^*\!\Gr(2,5)=
(\ttt{{\fs}2\bs 2\bs 2\bs 2\bs 2\bs 2{\fs}})$. Carrying out the sequence of Hanany-Witten transitions 

\smallskip

\centerline{
\ttt{\bs 2{\fs}2{\fs}2{\fs}2{\fs}2{\fs}2\bs}  $\leftrightarrow$ \ttt{{\fs}1\bs 2{\fs}2{\fs}2{\fs}2{\fs}2\bs}  $\leftrightarrow$  \ttt{{\fs}1{\fs}2\bs 2{\fs}2{\fs}2{\fs}2\bs}  $\leftrightarrow$ 
\ttt{{\fs}1{\fs}2\bs 2{\fs}2{\fs}2\bs 1{\fs}}  $\leftrightarrow$ \ttt{{\fs}1{\fs}2\bs 2{\fs}2\bs 2{\fs}1{\fs}} 
}
\smallskip

\noindent on the dual diagram  we obtained a co-balanced brane diagram; and in turn that 
\[
\Naka\Big(
  \begin{tikzpicture}[baseline=-25pt,scale=.5]
 \draw[fill] (3,-1) circle (3pt);  
 \draw[thick] (2.9,-2.1) rectangle (3.1,-1.9);
 \draw[thick] (3,-1) -- (3,-1.9);
 \node at (3.45,-1) {$\scriptscriptstyle 2$} ;
 \node at (3.45,-2) {$\scriptscriptstyle 5$} ;   
\end{tikzpicture}
\Big)
\qquad\qquad 
\text{is 3d mirror dual to}
\qquad\qquad 
\Naka\Big(
 \begin{tikzpicture}[baseline=-25pt,scale=.5]
 \draw[thick] (0,-1) -- (4.5,-1); 
  \draw[fill] (0,-1) circle (3pt);  
 \draw[fill] (1.5,-1) circle (3pt);  
 \draw[fill] (3,-1) circle (3pt);  
 \draw[fill] (4.5,-1) circle (3pt);
  \draw[thick] (1.4,-2.1) rectangle (1.6,-1.9);
 \draw[thick] (2.9,-2.1) rectangle (3.1,-1.9);
 \draw[thick] (1.5,-1) -- (1.5,-1.9);
 \draw[thick] (3,-1) -- (3,-1.9);
 \node at (0,-.5) {$\scriptscriptstyle 1$} ;
 \node at (1.5,-.5) {$\scriptscriptstyle 2$} ;
 \node at (3,-.5) {$\scriptscriptstyle 2$} ;
  \node at (4.5,-.5) {$\scriptscriptstyle 1$} ;
\node at (1.5,-2.6) {$\scriptscriptstyle 1$} ;   \node at (3,-2.6) {$\scriptscriptstyle 1$} ;   
\end{tikzpicture}
\Big)
\]
up to Hanany-Witten isomorphism. More generally, the same procedure gives that for $2k\leq n$ the variety $T^*\!\Gr(k,n)=\Naka\Big(
  \begin{tikzpicture}[baseline=-25pt,scale=.5]
 \draw[fill] (3,-1) circle (3pt);  
 \draw[thick] (2.9,-2.1) rectangle (3.1,-1.9);
 \draw[thick] (3,-1) -- (3,-1.9);
 \node at (3.45,-1) {$\scriptscriptstyle k$} ;
 \node at (3.45,-2) {$\scriptscriptstyle n$} ;   
\end{tikzpicture}
\Big)$
is 3d mirror dual---up to HW isomorphism---to $\Naka(Q(v,w))$, where 
\begin{equation}\label{eq:quiverGr}
v=(1,2,\ldots,k-1,\underbrace{k,k,\ldots,k}_{n-2k+1},k-1,\ldots,2,1),
\qquad
w=e_k+e_{n-k}.
\end{equation}
(Here $e_i$ is the $i$th standard basis vector in $\Z^{n-1}$.) This pair of 3d mirror symmetric varieties is explored in terms of elliptic characteristic classes in \cite{RSVZ1}.

Similarly, $\Naka(
\begin{tikzpicture}[baseline=-20pt,scale=.4]
 \draw[fill] (0,-1) circle (3pt);  
 \draw[fill] (1,-1) circle (3pt);  
 \draw[thick] (.9,-2.1) rectangle (1.1,-1.9);  
 \draw[thick] (0,-1) -- (1,-1) -- (1,-1.9);
\node at (0,-.5) {$\scriptscriptstyle 2$} ;
\node at (1,-.5) {$\scriptscriptstyle 6$} ;
\node at (.4,-2) {$\scriptscriptstyle 10$} ;
\end{tikzpicture})$ and 
$\Naka(
\begin{tikzpicture}[baseline=-20pt,scale=.4]
 \draw[fill] (0,-1) circle (3pt);  
 \draw[fill] (1,-1) circle (3pt);  
 \draw[fill] (2,-1) circle (3pt);  
 \draw[fill] (3,-1) circle (3pt);  
 \draw[fill] (4,-1) circle (3pt);  
 \draw[fill] (5,-1) circle (3pt);  
 \draw[fill] (6,-1) circle (3pt);  
 \draw[fill] (7,-1) circle (3pt);  
 \draw[fill] (8,-1) circle (3pt);  
 \draw[thick] (3.1,-2.1) rectangle (2.9,-1.9);
 \draw[thick] (1.1,-2.1) rectangle (0.9,-1.9);
 \draw[thick] (8,-1) -- (0,-1);
 \draw[thick] (3,-1) -- (3,-1.9);
 \draw[thick] (1,-1) -- (1,-1.9);
 \node at (8,-.5) {$\scriptscriptstyle 1$} ;
 \node at (7,-.5) {$\scriptscriptstyle 2$} ;
 \node at (6,-.5) {$\scriptscriptstyle 3$} ;
 \node at (5,-.5) {$\scriptscriptstyle 4$} ;
 \node at (4,-.5) {$\scriptscriptstyle 5$} ;
 \node at (3,-.5) {$\scriptscriptstyle 6$} ;
 \node at (2,-.5) {$\scriptscriptstyle 5$} ;
 \node at (1,-.5) {$\scriptscriptstyle 4$} ;
 \node at (0,-.5) {$\scriptscriptstyle 2$} ; 
   \node at (3.5,-2) {$\scriptscriptstyle 2$} ;
   \node at (.5,-2) {$\scriptscriptstyle 1$} ;
 \end{tikzpicture}
 )$ are 3d mirror duals. This can be proved by a sequence of Hanany-Witten transitions, or alternatively, by verifying that their margin vectors
\begin{align*}
r=(2,4,4), \qquad & \qquad c=(1,1,1,1,1,1,1,1,1,1), \\
r'=(2,2,2,2,2,2,2,2,2,2), \qquad & \qquad c'=(8,6,6)
\end{align*}
satisfy \eqref{eq:rc} from Section \ref{sec:3d}.
The argument of comparing margin vectors via \eqref{eq:rc} has the following straightforward generalization. For $v,w\in \Z_{\geq 0}^n$ define
\[
r(v,w)=(v_i-v_{i-1}+\sum_{j=1}^{i-1} w_j)_{i=1,\ldots,n+1},
\qquad\qquad
c(v,w)=(n^{w_1}, (n-1)^{w_2},\ldots,2^{w_{n-1}},1^{w_n})
\]
where we set $v_0=v_{n+1}=0$. These are the margin vectors of $\Naka(Q(v,w))$ as a bow variety. Hence we obtain

\begin{theorem} \label{thm:QuiverDuals}
Suppose $v,w\in \Z_{\geq 0}^{a-1}$, $v',w'\in \Z_{\geq 0}^{b-1}$ with $\sum w_i=b, \sum w'_i=a$. The quiver varieties $\Naka(Q(v,w))$ and $\Naka(Q(v',w'))$ are 3d mirror duals---up to HW isomorphism---if and only if 
\[
r(v,w)+c(v',w')=(b^a),
\qquad\qquad
c(v,w)+r(v',w')=(a^b).
\] \qed
\end{theorem} 

As a corollary we obtain that the 3d mirror dual of $T^*\!\cF_\lambda$ (for certain $\lambda$'s) %, namely the ranks of quotient bundles form a weakly increasing sequence)
 have a 3d mirror dual that is isomorphic to a quiver variety, as follows. Let $(\lambda_1,\lambda_2,\ldots,\lambda_N)\in \Z_{\geq 0}^N$ and assume $\lambda_1\leq \lambda_2\leq \ldots \leq \lambda_N$. Consider 
\begin{equation}\label{eq:TF}
X=T^*\!\cF_{\lambda_1,\lambda_1+\lambda_2,\ldots,\lambda_1+\lambda_2+\ldots+\lambda_N}.
\end{equation}
Define 
\[
\mu=( (N-1)^{\lambda_1}, (N-2)^{\lambda_2-\lambda_1}, (N-3)^{\lambda_3-\lambda_2}, \ldots,1^{\lambda_{N-1}-\lambda_{N-2}},0^{\lambda_N-\lambda_{N-1}},
(-1)^{k-1})
\]
where $k=\sum_{j=1}^N (\lambda_j-\lambda_{j-1})(N-j)$ (that is, the sum of the first $\lambda_N$ entries of $\mu$). Define $v,w\in \Z_{\geq 0}^{\lambda_N+k-1}$ by
\[
v_i=\sum_{j=1}^i \mu_{j},
\qquad\qquad
w=e_{\lambda_1}+e_{\lambda_2}+\ldots+e_{\lambda_N},
\]
and $X'=\Naka(Q(v,w))$.

\begin{corollary} \label{cor:DTF}
The 3d mirror dual of $X$ is isomorphic to $X'$. 
\end{corollary}

\begin{proof} 
This is a special case of Theorem \ref{thm:QuiverDuals}.
\end{proof}

Another proof---instead of comparing margin vectors---would be to list the necessary Hanany-Witten transitions that connect the 3d mirror dual of the brane diagram of $X$ with the brane diagram of $X'$. For example for $\lambda=(2,3,3)$ it is 

\ttt{{\bs}2{\bs}5{\fs}5{\fs}5{\fs}5{\fs}5{\fs}5{\fs}5{\fs}5{\fs}5{\bs}}  $\leftrightarrow$ \ttt{{\bs}2{\fs}3{\bs}5{\fs}5{\fs}5{\fs}5{\fs}5{\fs}5{\fs}5{\fs}5{\bs}} $\leftrightarrow$  \ttt{{\fs}2{\bs}3{\bs}5{\fs}5{\fs}5{\fs}5{\fs}5{\fs}5{\fs}5{\fs}5{\bs}} $\leftrightarrow$

\ttt{{\fs}2{\bs}3{\fs}4{\bs}5{\fs}5{\fs}5{\fs}5{\fs}5{\fs}5{\fs}5{\bs}} $\leftrightarrow$  \ttt{{\fs}2{\fs}4{\bs}4{\bs}5{\fs}5{\fs}5{\fs}5{\fs}5{\fs}5{\fs}5{\bs}} $\leftrightarrow$ \ttt{{\fs}2{\fs}4{\bs}4{\fs}5{\bs}5{\fs}5{\fs}5{\fs}5{\fs}5{\fs}5{\bs}} $\leftrightarrow$

\ttt{{\fs}2{\fs}4{\bs}4{\fs}5{\bs}5{\fs}5{\fs}5{\fs}5{\fs}5{\bs}1{\fs}} $\leftrightarrow$ \ttt{{\fs}2{\fs}4{\bs}4{\fs}5{\bs}5{\fs}5{\fs}5{\fs}5{\bs}2{\fs}1{\fs}} $\leftrightarrow$  \ttt{{\fs}2{\fs}4{\bs}4{\fs}5{\bs}5{\fs}5{\fs}5{\bs}3{\fs}2{\fs}1{\fs}} $\leftrightarrow$

\ttt{{\fs}2{\fs}4{\bs}4{\fs}5{\bs}5{\fs}5{\bs}4{\fs}3{\fs}2{\fs}1{\fs}} $\leftrightarrow$ \ttt{{\fs}2{\fs}4{\bs}4{\fs}5{\bs}5{\bs}5{\fs}4{\fs}3{\fs}2{\fs}1{\fs}}.

\noindent A remarkable even further special case is the pair of quivers 
\begin{equation}\label{eq:fullflag}
\begin{tikzpicture}[scale=.4]
 \draw[thick] (1.5,-1) -- (3.5,-1);  \draw[thick] (5.6,-1) -- (9,-1); 
  \draw[fill] (1.5,-1) circle (3pt);  
 \draw[fill] (3,-1) circle (3pt);  
    \node at (4.5,-1) {$\cdots$};  
 \draw[fill] (6,-1) circle (3pt); 
 \draw[fill] (7.5,-1) circle (3pt);  
 \draw[fill] (9,-1) circle (3pt);  

 \draw[thick] (8.9,-2.1) rectangle (9.1,-1.9);
 \draw[thick] (9,-1) -- (9,-1.9);
 \node at (1.5,-.5) {$\scriptscriptstyle 1$} ;
 \node at (3,-.5) {$\scriptscriptstyle 2$} ;
 \node at (6,-.5) {$\scriptscriptstyle n-3\ $} ;
 \node at (7.5,-.5) {$\scriptscriptstyle n-2$} ;
 \node at (9,-.5) {$\scriptscriptstyle \ n-1$} ;
 \node at (9.5,-2) {$\scriptscriptstyle n$} ;
\end{tikzpicture}
\qquad\qquad
\begin{tikzpicture}[scale=.4]
 \draw[thick] (1.5,-1) -- (3.5,-1);  \draw[thick] (5.6,-1) -- (9,-1); 
  \draw[fill] (1.5,-1) circle (3pt);  
 \draw[fill] (3,-1) circle (3pt);  
    \node at (4.5,-1) {$\cdots$};  
 \draw[fill] (6,-1) circle (3pt); 
 \draw[fill] (7.5,-1) circle (3pt);  
 \draw[fill] (9,-1) circle (3pt);  

 \draw[thick] (1.4,-2.1) rectangle (1.6,-1.9);
 \draw[thick] (1.5,-1) -- (1.5,-1.9);
 \node at (1.5,-.5) {$\scriptscriptstyle n-1$} ;
 \node at (3,-.5) {$\scriptscriptstyle n-2$} ;
 \node at (6,-.5) {$\scriptscriptstyle 3$} ;
 \node at (7.5,-.5) {$\scriptscriptstyle 2$} ;
 \node at (9,-.5) {$\scriptscriptstyle 1$} ;
 \node at (1,-2) {$\scriptscriptstyle n$} ;
\end{tikzpicture}.
\end{equation}
According to Corollary \ref{cor:DTF} (or direct calculation of HW moves) they are 3d mirror duals of each other (up to Hanany-Witten transitions). Both of them are isomorphic to the cotangent bundle of a full flag variety of $\C^n$ (the right hand one with non-standard torus action). This 3d mirror self-symmetry is explored in terms of elliptic characteristic classes in \cite{RSVZ2, RW2}.

\begin{remark} \rm
In this section we named pairs of quiver varieties $X$ and $X'$ such that $X'$ is isomorphic to the 3d mirror dual of $X$. However, the isomorphism was through Hanany-Witten transitions, and under such the torus action gets reparametrized. Hence in future equivariant cohomological studies %of corresponding characteristic classes of the two sides 
those torus-reparameterizations need to be taken into account. In fact---without the knowledge of bow varieties or HW transition---such a reparameterization was artificially introduced in \cite{RSVZ1, RW2}.
\end{remark}

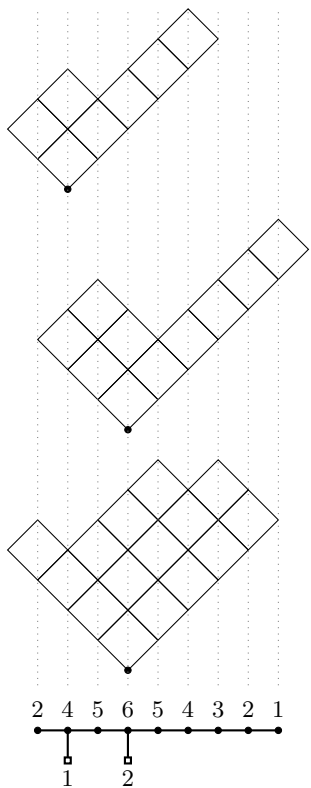
\begin{figure}
\[
\begin{tikzpicture}[scale=.4]
 \draw[fill] (0,-1) circle (3pt);  
 \draw[fill] (1,-1) circle (3pt);  
 \draw[fill] (2,-1) circle (3pt);  
 \draw[fill] (3,-1) circle (3pt);  
 \draw[fill] (4,-1) circle (3pt);  
 \draw[fill] (5,-1) circle (3pt);  
 \draw[fill] (6,-1) circle (3pt);  
 \draw[fill] (7,-1) circle (3pt);  
 \draw[fill] (8,-1) circle (3pt);  
 \draw[thick] (3.1,-2.1) rectangle (2.9,-1.9);
 \draw[thick] (1.1,-2.1) rectangle (0.9,-1.9);
 \draw[thick] (8,-1) -- (0,-1);
 \draw[thick] (3,-1) -- (3,-1.9);
 \draw[thick] (1,-1) -- (1,-1.9);
 \node at (8,-.3) {\scriptsize $1$} ;
 \node at (7,-.3) {\scriptsize $2$} ;
 \node at (6,-.3) {\scriptsize $3$} ;
 \node at (5,-.3) {\scriptsize $4$} ;
 \node at (4,-.3) {\scriptsize $5$} ;
 \node at (3,-.3) {\scriptsize $6$} ;
 \node at (2,-.3) {\scriptsize $5$} ;
 \node at (1,-.3) {\scriptsize $4$} ;
 \node at (0,-.3) {\scriptsize $2$} ; 
   \node at (3,-2.6) {\scriptsize $2$} ;
   \node at (1,-2.6) {\scriptsize $1$} ;
 
\draw[fill] (3,1) circle (3pt); 
        \coordinate (a) at (3,1);
        \draw ($(a)$) -- ($(a) + (1,1)$) -- ($(a) + (0,2)$)--($(a) + (-1,1)$) -- ($(a)$);
        \coordinate (a) at (2,2);
        \draw ($(a)$) -- ($(a) + (1,1)$) -- ($(a) + (0,2)$)--($(a) + (-1,1)$) -- ($(a)$);
        \coordinate (a) at (1,3);
        \draw ($(a)$) -- ($(a) + (1,1)$) -- ($(a) + (0,2)$)--($(a) + (-1,1)$) -- ($(a)$);
        \coordinate (a) at (0,4);
        \draw ($(a)$) -- ($(a) + (1,1)$) -- ($(a) + (0,2)$)--($(a) + (-1,1)$) -- ($(a)$);
        \coordinate (a) at (4,2);
        \draw ($(a)$) -- ($(a) + (1,1)$) -- ($(a) + (0,2)$)--($(a) + (-1,1)$) -- ($(a)$);
        \coordinate (a) at (3,3);
        \draw ($(a)$) -- ($(a) + (1,1)$) -- ($(a) + (0,2)$)--($(a) + (-1,1)$) -- ($(a)$);
        \coordinate (a) at (2,4);
        \draw ($(a)$) -- ($(a) + (1,1)$) -- ($(a) + (0,2)$)--($(a) + (-1,1)$) -- ($(a)$);
        \coordinate (a) at (5,3);
        \draw ($(a)$) -- ($(a) + (1,1)$) -- ($(a) + (0,2)$)--($(a) + (-1,1)$) -- ($(a)$);
        \coordinate (a) at (4,4);
        \draw ($(a)$) -- ($(a) + (1,1)$) -- ($(a) + (0,2)$)--($(a) + (-1,1)$) -- ($(a)$);
        \coordinate (a) at (3,5);
        \draw ($(a)$) -- ($(a) + (1,1)$) -- ($(a) + (0,2)$)--($(a) + (-1,1)$) -- ($(a)$);
        \coordinate (a) at (6,4);
        \draw ($(a)$) -- ($(a) + (1,1)$) -- ($(a) + (0,2)$)--($(a) + (-1,1)$) -- ($(a)$);
        \coordinate (a) at (5,5);
        \draw ($(a)$) -- ($(a) + (1,1)$) -- ($(a) + (0,2)$)--($(a) + (-1,1)$) -- ($(a)$);
        \coordinate (a) at (4,6);
        \draw ($(a)$) -- ($(a) + (1,1)$) -- ($(a) + (0,2)$)--($(a) + (-1,1)$) -- ($(a)$);
        \coordinate (a) at (7,5);
        \draw ($(a)$) -- ($(a) + (1,1)$) -- ($(a) + (0,2)$)--($(a) + (-1,1)$) -- ($(a)$);
        \coordinate (a) at (6,6);
        \draw ($(a)$) -- ($(a) + (1,1)$) -- ($(a) + (0,2)$)--($(a) + (-1,1)$) -- ($(a)$);
\draw[fill] (3,9) circle (3pt); 
        \coordinate (a) at (3,9);
        \draw ($(a)$) -- ($(a) + (1,1)$) -- ($(a) + (0,2)$)--($(a) + (-1,1)$) -- ($(a)$);
         \coordinate (a) at (2,10);
        \draw ($(a)$) -- ($(a) + (1,1)$) -- ($(a) + (0,2)$)--($(a) + (-1,1)$) -- ($(a)$);
        \coordinate (a) at (1,11);
        \draw ($(a)$) -- ($(a) + (1,1)$) -- ($(a) + (0,2)$)--($(a) + (-1,1)$) -- ($(a)$);
        \coordinate (a) at (4,10);
        \draw ($(a)$) -- ($(a) + (1,1)$) -- ($(a) + (0,2)$)--($(a) + (-1,1)$) -- ($(a)$);
        \coordinate (a) at (3,11);
        \draw ($(a)$) -- ($(a) + (1,1)$) -- ($(a) + (0,2)$)--($(a) + (-1,1)$) -- ($(a)$);
        \coordinate (a) at (2,12);
        \draw ($(a)$) -- ($(a) + (1,1)$) -- ($(a) + (0,2)$)--($(a) + (-1,1)$) -- ($(a)$);
        \coordinate (a) at (5,11);
        \draw ($(a)$) -- ($(a) + (1,1)$) -- ($(a) + (0,2)$)--($(a) + (-1,1)$) -- ($(a)$);
        \coordinate (a) at (6,12);
        \draw ($(a)$) -- ($(a) + (1,1)$) -- ($(a) + (0,2)$)--($(a) + (-1,1)$) -- ($(a)$);
        \coordinate (a) at (7,13);
        \draw ($(a)$) -- ($(a) + (1,1)$) -- ($(a) + (0,2)$)--($(a) + (-1,1)$) -- ($(a)$);
        \coordinate (a) at (8,14);
        \draw ($(a)$) -- ($(a) + (1,1)$) -- ($(a) + (0,2)$)--($(a) + (-1,1)$) -- ($(a)$);
\draw[fill] (1,17) circle (3pt); 
        \coordinate (a) at (1,17);
        \draw ($(a)$) -- ($(a) + (1,1)$) -- ($(a) + (0,2)$)--($(a) + (-1,1)$) -- ($(a)$);
         \coordinate (a) at (0,18);
        \draw ($(a)$) -- ($(a) + (1,1)$) -- ($(a) + (0,2)$)--($(a) + (-1,1)$) -- ($(a)$);
        \coordinate (a) at (2,18);
        \draw ($(a)$) -- ($(a) + (1,1)$) -- ($(a) + (0,2)$)--($(a) + (-1,1)$) -- ($(a)$);
        \coordinate (a) at (1,19);
        \draw ($(a)$) -- ($(a) + (1,1)$) -- ($(a) + (0,2)$)--($(a) + (-1,1)$) -- ($(a)$);
        \coordinate (a) at (3,19);
        \draw ($(a)$) -- ($(a) + (1,1)$) -- ($(a) + (0,2)$)--($(a) + (-1,1)$) -- ($(a)$);
        \coordinate (a) at (4,20);
        \draw ($(a)$) -- ($(a) + (1,1)$) -- ($(a) + (0,2)$)--($(a) + (-1,1)$) -- ($(a)$);
        \coordinate (a) at (5,21);
        \draw ($(a)$) -- ($(a) + (1,1)$) -- ($(a) + (0,2)$)--($(a) + (-1,1)$) -- ($(a)$);
        \coordinate (a) at (6,22);

\draw [dotted, gray] (0,0.5) -- (0,23);
\draw [dotted, gray] (1,0.5) -- (1,23);
\draw [dotted, gray] (2,0.5) -- (2,23);
\draw [dotted, gray] (3,0.5) -- (3,23);
\draw [dotted, gray] (4,0.5) -- (4,23);
\draw [dotted, gray] (5,0.5) -- (5,23);
\draw [dotted, gray] (6,0.5) -- (6,23);
\draw [dotted, gray] (7,0.5) -- (7,23);
\draw [dotted, gray] (8,0.5) -- (8,23);
 \end{tikzpicture}
\]
\caption{The combinatorial code for one of the 3,150 fixed points of the quiver variety $\Naka((2,4,5,6,5,4,3,2,1),(0,1,0,2,0,0,0,0,0,0))$.} \label{fig:Qfix}
\end{figure}

\subsection{Combinatorial codes for the $\T$ fixed points of $\Naka(Q)$} \label{sec:QFix}
Let us identify a partition with its Young diagram {\em standing on its corner}. That is we rotate the usual pictures of Young diagrams by $45^\circ$.  Above each position of the quiver, consider an ordered $w_i$-tuple of partitions (whose box diagonals have the same length as the segments of the quiver), drawn disjoint and ordered from bottom to top (see Figure \ref{fig:Qfix}). If  such a  $\sum w_i$-tuple of partitions is given so that above each position of the quiver the total number of partition-boxes is $v_i$ then this tuple is a combinatorial code of a $\T$ fixed point on $\Naka(Q)$---see \cite[\S 5.2]{NakajimaBook} for a special case that generalizes to type A quivers. See Figure \ref{fig:Qfix} for the code of a fixed point of a quiver variety.

\subsection{Fixed point codes on quiver vs bow varieties}
As quiver varieties are bow varieties, a fixed point code like Figure \ref{fig:Qfix} has a combinatorial code as a fixed point in a bow variety. The process of going from the sequence of partitions to the corresponding tie diagram is easy, but instead of writing complicated formulas, we leave the details to the reader---the process can be read from comparing Figures \ref{fig:Qfix} and \ref{fig:parts2ties}.

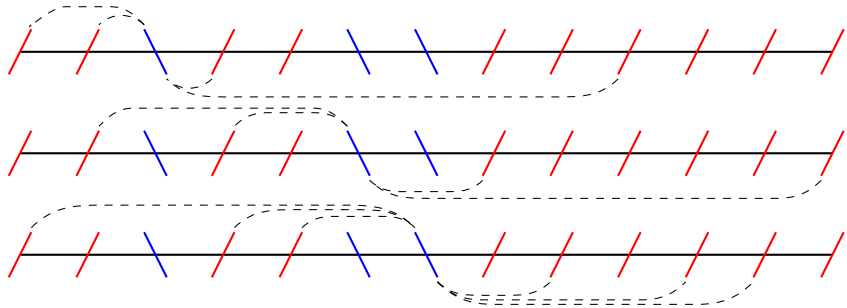
\begin{figure}
\[
\begin{tikzpicture}[scale=.3]
\draw[thick] (0,1)--(36,1) ;
\draw[thick,red] (-.5,0)--(.5,2);
\draw[thick,red] (2.5,0)--(3.5,2);
\draw[thick,blue] (6.5,0)--(5.5,2);
\draw[thick,red] (8.5,0)--(9.5,2);
\draw[thick,red] (11.5,0)--(12.5,2);
\draw[thick,blue] (15.5,0)--(14.5,2);
\draw[thick,blue] (18.5,0)--(17.5,2);
\draw[thick,red] (20.5,0)--(21.5,2);
\draw[thick,red] (23.5,0)--(24.5,2);
\draw[thick,red] (26.5,0)--(27.5,2);
\draw[thick,red] (29.5,0)--(30.5,2);
\draw[thick,red] (32.5,0)--(33.5,2);
\draw[thick,red] (35.5,0)--(36.5,2);

\draw [dashed, black](26.5,-.2) to [out=180+45,in=0] (23.5,-1) to (9.5,-1) to [out=180,in=-45] (6.5,-.2);
\draw [dashed, black](8.5,-.2) to [out=180+45,in=180-225] (6.5,-.2);
\draw [dashed, black](5.5,2.2) to [out=180-45,in=180+225] (3.5,2.2);
\draw [dashed, black](5.5,2.2) to [out=180-45,in=0]  (4,3) to (2,3) to [out=180,in=180+45] (.5,2.2);

\begin{scope}[yshift=-4.5cm]
\draw[thick] (0,1)--(36,1) ;
\draw[thick,red] (-.5,0)--(.5,2);
\draw[thick,red] (2.5,0)--(3.5,2);
\draw[thick,blue] (6.5,0)--(5.5,2);
\draw[thick,red] (8.5,0)--(9.5,2);
\draw[thick,red] (11.5,0)--(12.5,2);
\draw[thick,blue] (15.5,0)--(14.5,2);
\draw[thick,blue] (18.5,0)--(17.5,2);
\draw[thick,red] (20.5,0)--(21.5,2);
\draw[thick,red] (23.5,0)--(24.5,2);
\draw[thick,red] (26.5,0)--(27.5,2);
\draw[thick,red] (29.5,0)--(30.5,2);
\draw[thick,red] (32.5,0)--(33.5,2);
\draw[thick,red] (35.5,0)--(36.5,2);
\draw [dashed, black](35.5,-.2) to [out=180+45,in=0] (33.5,-1) to (18.5,-1) to [out=180,in=-45] (15.5,-.2);
\draw [dashed, black](20.5,-.2) to [out=180+45,in=0] (18.5,-.7) to (17.5,-.7) to [out=180,in=-45]  (15.5,-.2);
\draw [dashed, black](14.5,2.2) to [out=135,in=0] (13,2.8) to (11,2.8) to [out=180,in=45] (9.5,2.2);
\draw [dashed, black](14.5,2.2) to [out=135,in=0]  (13,3) to (5,3) to [out=180,in=45] (3.5,2.2);
\end{scope}

\begin{scope}[yshift=-9cm]
\draw[thick] (0,1)--(36,1) ;
\draw[thick,red] (-.5,0)--(.5,2);
\draw[thick,red] (2.5,0)--(3.5,2);
\draw[thick,blue] (6.5,0)--(5.5,2);
\draw[thick,red] (8.5,0)--(9.5,2);
\draw[thick,red] (11.5,0)--(12.5,2);
\draw[thick,blue] (15.5,0)--(14.5,2);
\draw[thick,blue] (18.5,0)--(17.5,2);
\draw[thick,red] (20.5,0)--(21.5,2);
\draw[thick,red] (23.5,0)--(24.5,2);
\draw[thick,red] (26.5,0)--(27.5,2);
\draw[thick,red] (29.5,0)--(30.5,2);
\draw[thick,red] (32.5,0)--(33.5,2);
\draw[thick,red] (35.5,0)--(36.5,2);
\draw [dashed, black](.5,2.2) to [out=45,in=180] (3.5,3.2) to (14.5,3.2) to [out=0,in=-225] (17.5,2.2);
\draw [dashed, black](9.5,2.2) to [out=45,in=180] (12.5,3) to (14.5,3) to [out=0,in=-225]  (17.5,2.2);
\draw [dashed, black](12.5,2.2) to [out=45,in=180] (14.5,2.7) to (15.5,2.7) to [out=0,in=-225]  (17.5,2.2);
\draw [dashed, black](18.5,-.2) to [out=-45,in=180] (20.5,-0.8) to (21.5,-0.8) to [out=0,in=225] (23.5,-.2);
\draw [dashed, black](18.5,-.2) to [out=-45,in=180] (20.5,-1) to (27.5,-1) to [out=0,in=225] (29.5,-.2);
\draw [dashed, black](18.5,-.2) to [out=-45,in=180]  (21,-1.2) to (30,-1.2) to [out=0,in=225] (32.5,-.2);
\end{scope}
\end{tikzpicture}
\]
\caption{The union of these three diagrams is the tie diagram encoding the fixed point of Figure \ref{fig:Qfix} in bow variety language.} \label{fig:parts2ties}
\end{figure}

\subsection{Fixed point identification between $T^*\!\cF_v$ and its dual}
In Corollary \ref{cor:DTF} we identified two quiver varieties---namely the cotangent bundle of a partial flag variety and its dual---up to 3d mirror symmetry and Hanany-Witten transition. Under each of these operations there is a natural matching of the torus fixed points. The net effect is a bijection between the torus fixed points of $X$ and those of $X'$. Carefully walking through the steps of the identification we get a bijection between the combinatorial codes of the fixed points. This is illustrated in Figure \ref{fig:QFixBijection}.
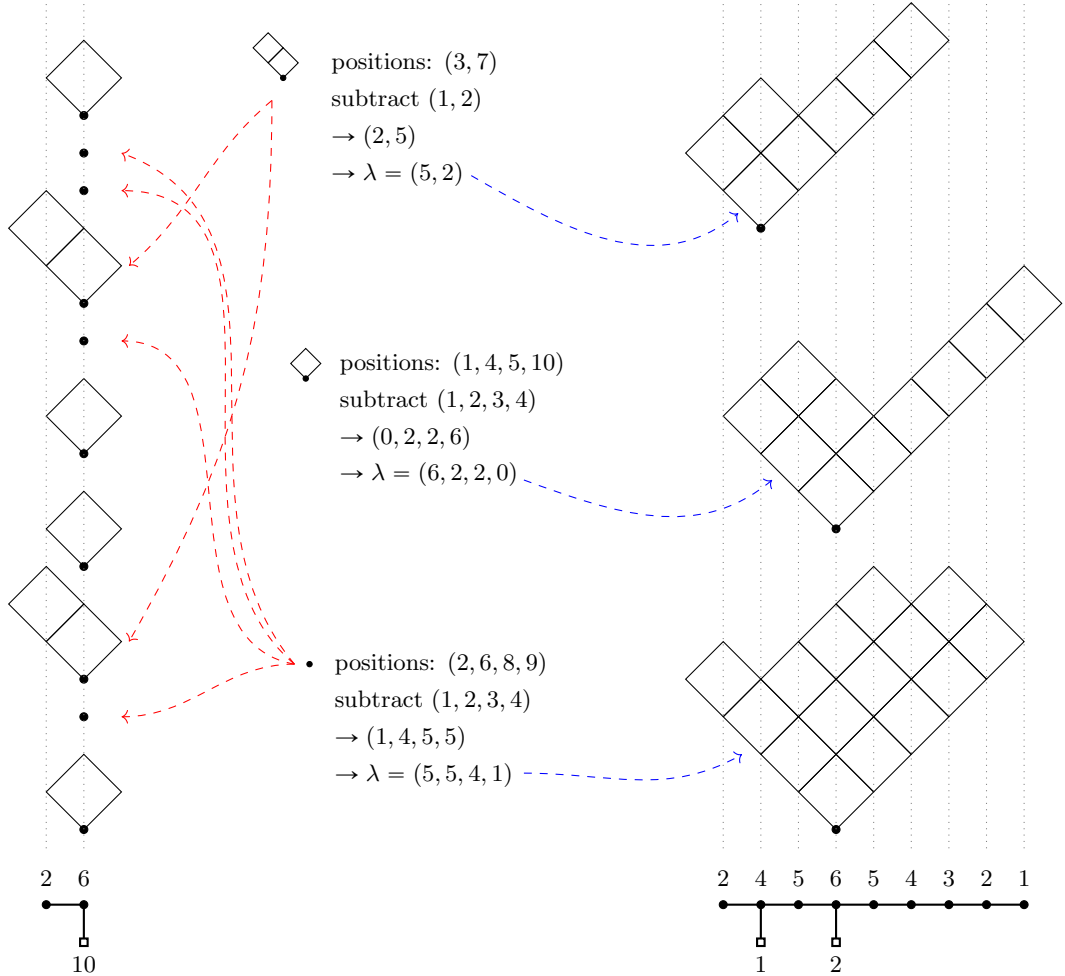
\begin{figure}
\[
 \begin{tikzpicture}[scale=.5]
  
 \draw[fill] (0,-1) circle (3pt);  
 \draw[fill] (1,-1) circle (3pt);  
 \draw[thick] (.9,-2.1) rectangle (1.1,-1.9);  
 \draw[thick] (0,-1) -- (1,-1) -- (1,-1.9);
\node at (0,-.3) {\scriptsize $2$} ;
\node at (1,-.3) {\scriptsize $6$} ;
\node at (1,-2.6) {\scriptsize $10$} ;

 \draw[fill] (1,1) circle (3pt);  
        \coordinate (a) at (1,1);
        \draw ($(a)$) -- ($(a) + (1,1)$) -- ($(a) + (0,2)$)--($(a) + (-1,1)$) -- ($(a)$);
 \draw[fill] (1,4) circle (3pt);  
 \draw[fill] (1,5) circle (3pt); 
        \coordinate (a) at (1,5);
        \draw ($(a)$) -- ($(a) + (1,1)$) -- ($(a) + (0,2)$)--($(a) + (-1,1)$) -- ($(a)$);
        \coordinate (a) at (0,6);
        \draw ($(a)$) -- ($(a) + (1,1)$) -- ($(a) + (0,2)$)--($(a) + (-1,1)$) -- ($(a)$);
 \draw[fill] (1,8) circle (3pt); 
        \coordinate (a) at (1,8);
        \draw ($(a)$) -- ($(a) + (1,1)$) -- ($(a) + (0,2)$)--($(a) + (-1,1)$) -- ($(a)$);
\draw[fill] (1,11) circle (3pt); 
        \coordinate (a) at (1,11);
        \draw ($(a)$) -- ($(a) + (1,1)$) -- ($(a) + (0,2)$)--($(a) + (-1,1)$) -- ($(a)$);
\draw[fill] (1,14) circle (3pt); 
\draw[fill] (1,15) circle (3pt); 
        \coordinate (a) at (1,15);
        \draw ($(a)$) -- ($(a) + (1,1)$) -- ($(a) + (0,2)$)--($(a) + (-1,1)$) -- ($(a)$);
        \coordinate (a) at (0,16);
        \draw ($(a)$) -- ($(a) + (1,1)$) -- ($(a) + (0,2)$)--($(a) + (-1,1)$) -- ($(a)$);
\draw[fill] (1,18) circle (3pt); 
\draw[fill] (1,19) circle (3pt); 
\draw[fill] (1,20) circle (3pt); 
        \coordinate (a) at (1,20);
        \draw ($(a)$) -- ($(a) + (1,1)$) -- ($(a) + (0,2)$)--($(a) + (-1,1)$) -- ($(a)$);

\draw [dotted, gray] (1,0.5) -- (1,23);
\draw [dotted, gray] (0,0.5) -- (0,23);

\coordinate (a) at (7,5.4);
\draw[fill] (a) circle (2pt); 
\node[align=left, below] at (10.5,6)
{\scriptsize positions: \scriptsize $(2,6,8,9)$ \\ \scriptsize subtract $(1,2,3,4)$ \\ \scriptsize$\to (1,4,5,5)$ \\ \scriptsize $\to \lambda=(5,5,4,1)$};
\draw [red,dashed,->](6.6,5.4) to [out=180,in=0] (2,4);
\draw [red,dashed,->](6.6,5.4) to [out=170,in=0] (2,14);
\draw [red,dashed,->](6.6,5.4) to [out=140,in=0] (2,18);
\draw [red,dashed,->](6.6,5.4) to [out=130,in=340] (2,19);
\draw [blue,dashed,->](12.7,2.5) to [out=0,in=210] (18.5,3);

\coordinate (a) at (6.9,13);
\draw[fill] (a) circle (2pt); 
\draw ($(a)$) -- ($(a) + (.4,.4)$) -- ($(a) + (0,.8)$)--($(a) + (-.4,.4)$) -- ($(a)$);
\node[align=left, below] at (10.8,14)
{\scriptsize positions: \scriptsize $(1,4,5,10)$ \\ \scriptsize subtract $(1,2,3,4)$ \\ \scriptsize$\to (0,2,2,6)$ \\ \scriptsize $\to \lambda=(6,2,2,0)$};
\draw [blue,dashed,->](12.7,10.3) to [out=-20,in=220] (19.3,10.3);

\coordinate (a) at (6.3,21);
\draw[fill] (a) circle (2pt); 
\draw ($(a)$) -- ($(a) + (.4,.4)$) -- ($(a) + (0,.8)$)--($(a) + (-.4,.4)$) -- ($(a)$);
\coordinate (a) at (5.9,21.4);
\draw ($(a)$) -- ($(a) + (.4,.4)$) -- ($(a) + (0,.8)$)--($(a) + (-.4,.4)$) -- ($(a)$);
\node[align=left, below] at (9.8,22)
{\scriptsize positions: \scriptsize $(3,7)$ \\ \scriptsize subtract $(1,2)$ \\ \scriptsize$\to (2,5)$ \\ \scriptsize $\to \lambda=(5,2)$};
  \draw [red,dashed,->](6,20.4) to [out=270,in=60] (2.2,6);
  \draw [red,dashed,->](6,20.4) to [out=220,in=40] (2.2,16);
  \draw [blue,dashed,->](11.3,18.4) to [out=-30,in=220] (18.4,17.4);

\begin{scope}[xshift=18cm]
 \draw[fill] (0,-1) circle (3pt);  
 \draw[fill] (1,-1) circle (3pt);  
 \draw[fill] (2,-1) circle (3pt);  
 \draw[fill] (3,-1) circle (3pt);  
 \draw[fill] (4,-1) circle (3pt);  
 \draw[fill] (5,-1) circle (3pt);  
 \draw[fill] (6,-1) circle (3pt);  
 \draw[fill] (7,-1) circle (3pt);  
 \draw[fill] (8,-1) circle (3pt);  
 \draw[thick] (3.1,-2.1) rectangle (2.9,-1.9);
 \draw[thick] (1.1,-2.1) rectangle (0.9,-1.9);
 \draw[thick] (8,-1) -- (0,-1);
 \draw[thick] (3,-1) -- (3,-1.9);
 \draw[thick] (1,-1) -- (1,-1.9);
 \node at (8,-.3) {\scriptsize $1$} ;
 \node at (7,-.3) {\scriptsize $2$} ;
 \node at (6,-.3) {\scriptsize $3$} ;
 \node at (5,-.3) {\scriptsize $4$} ;
 \node at (4,-.3) {\scriptsize $5$} ;
 \node at (3,-.3) {\scriptsize $6$} ;
 \node at (2,-.3) {\scriptsize $5$} ;
 \node at (1,-.3) {\scriptsize $4$} ;
 \node at (0,-.3) {\scriptsize $2$} ; 
   \node at (3,-2.6) {\scriptsize $2$} ;
   \node at (1,-2.6) {\scriptsize $1$} ;
 
\draw[fill] (3,1) circle (3pt); 
        \coordinate (a) at (3,1);
        \draw ($(a)$) -- ($(a) + (1,1)$) -- ($(a) + (0,2)$)--($(a) + (-1,1)$) -- ($(a)$);
        \coordinate (a) at (2,2);
        \draw ($(a)$) -- ($(a) + (1,1)$) -- ($(a) + (0,2)$)--($(a) + (-1,1)$) -- ($(a)$);
        \coordinate (a) at (1,3);
        \draw ($(a)$) -- ($(a) + (1,1)$) -- ($(a) + (0,2)$)--($(a) + (-1,1)$) -- ($(a)$);
        \coordinate (a) at (0,4);
        \draw ($(a)$) -- ($(a) + (1,1)$) -- ($(a) + (0,2)$)--($(a) + (-1,1)$) -- ($(a)$);
        \coordinate (a) at (4,2);
        \draw ($(a)$) -- ($(a) + (1,1)$) -- ($(a) + (0,2)$)--($(a) + (-1,1)$) -- ($(a)$);
        \coordinate (a) at (3,3);
        \draw ($(a)$) -- ($(a) + (1,1)$) -- ($(a) + (0,2)$)--($(a) + (-1,1)$) -- ($(a)$);
        \coordinate (a) at (2,4);
        \draw ($(a)$) -- ($(a) + (1,1)$) -- ($(a) + (0,2)$)--($(a) + (-1,1)$) -- ($(a)$);
        \coordinate (a) at (5,3);
        \draw ($(a)$) -- ($(a) + (1,1)$) -- ($(a) + (0,2)$)--($(a) + (-1,1)$) -- ($(a)$);
        \coordinate (a) at (4,4);
        \draw ($(a)$) -- ($(a) + (1,1)$) -- ($(a) + (0,2)$)--($(a) + (-1,1)$) -- ($(a)$);
        \coordinate (a) at (3,5);
        \draw ($(a)$) -- ($(a) + (1,1)$) -- ($(a) + (0,2)$)--($(a) + (-1,1)$) -- ($(a)$);
        \coordinate (a) at (6,4);
        \draw ($(a)$) -- ($(a) + (1,1)$) -- ($(a) + (0,2)$)--($(a) + (-1,1)$) -- ($(a)$);
        \coordinate (a) at (5,5);
        \draw ($(a)$) -- ($(a) + (1,1)$) -- ($(a) + (0,2)$)--($(a) + (-1,1)$) -- ($(a)$);
        \coordinate (a) at (4,6);
        \draw ($(a)$) -- ($(a) + (1,1)$) -- ($(a) + (0,2)$)--($(a) + (-1,1)$) -- ($(a)$);
        \coordinate (a) at (7,5);
        \draw ($(a)$) -- ($(a) + (1,1)$) -- ($(a) + (0,2)$)--($(a) + (-1,1)$) -- ($(a)$);
        \coordinate (a) at (6,6);
        \draw ($(a)$) -- ($(a) + (1,1)$) -- ($(a) + (0,2)$)--($(a) + (-1,1)$) -- ($(a)$);
\draw[fill] (3,9) circle (3pt); 
        \coordinate (a) at (3,9);
        \draw ($(a)$) -- ($(a) + (1,1)$) -- ($(a) + (0,2)$)--($(a) + (-1,1)$) -- ($(a)$);
         \coordinate (a) at (2,10);
        \draw ($(a)$) -- ($(a) + (1,1)$) -- ($(a) + (0,2)$)--($(a) + (-1,1)$) -- ($(a)$);
        \coordinate (a) at (1,11);
        \draw ($(a)$) -- ($(a) + (1,1)$) -- ($(a) + (0,2)$)--($(a) + (-1,1)$) -- ($(a)$);
        \coordinate (a) at (4,10);
        \draw ($(a)$) -- ($(a) + (1,1)$) -- ($(a) + (0,2)$)--($(a) + (-1,1)$) -- ($(a)$);
        \coordinate (a) at (3,11);
        \draw ($(a)$) -- ($(a) + (1,1)$) -- ($(a) + (0,2)$)--($(a) + (-1,1)$) -- ($(a)$);
        \coordinate (a) at (2,12);
        \draw ($(a)$) -- ($(a) + (1,1)$) -- ($(a) + (0,2)$)--($(a) + (-1,1)$) -- ($(a)$);
        \coordinate (a) at (5,11);
        \draw ($(a)$) -- ($(a) + (1,1)$) -- ($(a) + (0,2)$)--($(a) + (-1,1)$) -- ($(a)$);
        \coordinate (a) at (6,12);
        \draw ($(a)$) -- ($(a) + (1,1)$) -- ($(a) + (0,2)$)--($(a) + (-1,1)$) -- ($(a)$);
        \coordinate (a) at (7,13);
        \draw ($(a)$) -- ($(a) + (1,1)$) -- ($(a) + (0,2)$)--($(a) + (-1,1)$) -- ($(a)$);
        \coordinate (a) at (8,14);
        \draw ($(a)$) -- ($(a) + (1,1)$) -- ($(a) + (0,2)$)--($(a) + (-1,1)$) -- ($(a)$);
\draw[fill] (1,17) circle (3pt); 
        \coordinate (a) at (1,17);
        \draw ($(a)$) -- ($(a) + (1,1)$) -- ($(a) + (0,2)$)--($(a) + (-1,1)$) -- ($(a)$);
         \coordinate (a) at (0,18);
        \draw ($(a)$) -- ($(a) + (1,1)$) -- ($(a) + (0,2)$)--($(a) + (-1,1)$) -- ($(a)$);
        \coordinate (a) at (2,18);
        \draw ($(a)$) -- ($(a) + (1,1)$) -- ($(a) + (0,2)$)--($(a) + (-1,1)$) -- ($(a)$);
        \coordinate (a) at (1,19);
        \draw ($(a)$) -- ($(a) + (1,1)$) -- ($(a) + (0,2)$)--($(a) + (-1,1)$) -- ($(a)$);
        \coordinate (a) at (3,19);
        \draw ($(a)$) -- ($(a) + (1,1)$) -- ($(a) + (0,2)$)--($(a) + (-1,1)$) -- ($(a)$);
        \coordinate (a) at (4,20);
        \draw ($(a)$) -- ($(a) + (1,1)$) -- ($(a) + (0,2)$)--($(a) + (-1,1)$) -- ($(a)$);
        \coordinate (a) at (5,21);
        \draw ($(a)$) -- ($(a) + (1,1)$) -- ($(a) + (0,2)$)--($(a) + (-1,1)$) -- ($(a)$);
        \coordinate (a) at (6,22);

\draw [dotted, gray] (0,0.5) -- (0,23);
\draw [dotted, gray] (1,0.5) -- (1,23);
\draw [dotted, gray] (2,0.5) -- (2,23);
\draw [dotted, gray] (3,0.5) -- (3,23);
\draw [dotted, gray] (4,0.5) -- (4,23);
\draw [dotted, gray] (5,0.5) -- (5,23);
\draw [dotted, gray] (6,0.5) -- (6,23);
\draw [dotted, gray] (7,0.5) -- (7,23);
\draw [dotted, gray] (8,0.5) -- (8,23);
 
\end{scope}
  \end{tikzpicture}
\]
\caption{Combinatorial description of the bijection between the torus fixed points of $T^*\!\Fl_{2,4,4}$ and its 3d mirror dual. (``Positions'' on the left are counted from below.)} \label{fig:QFixBijection}
\end{figure}
Namely, according to Section \ref{sec:QFix}, the combinatorial code for a torus fixed point on \eqref{eq:TF} is a length $\lambda_1+\ldots+\lambda_N$ sequence $S$ of partitions, such that $\lambda_i$ of them is $(1^{N-i})$. The code of the fixed points of $X'$ is an $N$-tuple of partitions. The $N-i$th component is then the partition 
\[
(p_{\lambda_i}-\lambda_i, p_{\lambda_i-1}-(\lambda_i-1),\ldots,p_3-3, p_2-2,p_1-1)
\]
where $p_1,p_2,\ldots,p_{\lambda_i}$ are the positions of the partition $1^{N-i}$ in $S$, see Figure \ref{fig:QFixBijection}. For more details see \cite{Sh}.

\section{Characteristic classes of bow varieties} \label{sec:char}

\subsection{Cohomology, K theory, localization}
Our motivation is to study the enumerative geometry of varieties through the characteristic classes of their subvarieties. These characteristic classes can be considered in various extraordinary cohomology theories, typically (equivariant) cohomology, K theory, or elliptic cohomology. %For now, we will restrict our attention to the first two: $H^*_{\T}$ and $K_{\T}$.
In each of these theories, an effective method is equivariant localization. This tool usually has three steps of arguments (here we phrase them for~$H^*_{\T}$):
\begin{itemize}
\item Optimally, the (equivariant) characteristic classes of the tautological bundles over the space $M$ generate $H^*_{\T}(M)$.
\item Optimally, the localization (restriction to the fixed points) map $\Loc: H^*_{\T}(M)\to H^*_{\T}(M^{\T})$ is injective. 
\item Optimally, the image of the localization map is described by `simple' relations among its components $H^*_{\T}(M^{\T})=\oplus H^*_{\T}($connected components of $M^{\T})$. 
\end{itemize}

The first step, under the name of ``Kirwan surjectivity'' is studied extensively: it generally holds for GIT quotients (and more), but may or may not hold for hyperk\"ahler quotients. We do not know if it holds for our class of bow varieties or not, cf. \cite{JKF, MG1, MG2}.

The second step holds in very general topological circumstances, but only up to $H^*_{\T}(\pt)$-torsion, and only under some (generalized, equivariant) compactness (or properness) assumptions, see e.g. \cite[Thm 2.3]{HHH}. Further studies are needed to verify whether bow varieties are covered. 

The third step holds for so-called GKM spaces \cite{GKM}. The assumptions that make a space a GKM space include that there are finitely many torus invariant curves on $M$. For GKM spaces  the image of the localization map is described by relations of the following flavor:  ``neighboring''  fixed point components agree under some substitutions of variables. For non-GKM spaces these coincidences must hold not only for the components themselves, but also for {\em some of their higher derivatives} as well. Bow varieties are typically not GKM spaces (an example is in Figure \ref{fig:GKM2}), hence GKM theory does not apply. This fact is both a disadvantage (Remark \ref{rem:nonGKM}) and an advantage (Remark \ref{rem:resonance}).

\smallskip

To circumvent some of the aforementioned obstacles, we make the following definitions.

\begin{definition}
In the algebra $H^*_{\T}(\Ch(\DD))$ (over $H^*_{\T}(\pt)=\C[u_1,\ldots,u_n,\h]$) consider the subalgebra generated by the Chern classes of the tautological bundles. Define $\HH_{\T}(\Ch(\DD))$ to be the $\Loc$-image of this subalgebra in $H_{\T}^*(\Ch(\DD)^{\T})=\C[u_1,\ldots,u_n,\h]^N$ ($N$ is the number of fixed points).
\end{definition}

\begin{definition}
In the algebra $K^*_{\T}(\Ch(\DD))$ (over $K^*_{\T}(\pt)=\C[\uu_1^{\pm 1},\ldots,\uu_n^{\pm 1},\hb^{\pm 1}]$) consider the subalgebra generated by the classes of the tautological bundles. Define $\KK_{\T}(\Ch(\DD))$ to be the $\Loc$-image of this subalgebra in $K_{\T}^*(\Ch(\DD)^{\T})=\C[\uu_1^{\pm 1},\ldots,\uu_n^{\pm 1},\hb^{\pm 1}]^N$ ($N$ is the number of fixed points).
\end{definition}

\begin{figure}
\[
\begin{tikzpicture}[scale=1.9]

\draw[fill=pink,pink] (1,0) to (1.2,-0.2) to  [out=-150, in=-30]  (.8,-0.2) to (1,0);
\draw[fill=pink,pink] (0,1) to (.2,.8) to  [out=-150, in=-30]  (-.2,.8) to (0,1);
\draw[fill=pink,pink] (2,3) to (2.2,2.8) to  [out=-150, in=-30]  (1.8,2.8) to (2,3);
\draw[fill=pink,pink] (1,2) to (1.1,1.8) to  [out=-150, in=-30]  (.8,1.8) to (1,2);
\draw[fill=pink,pink] (1,4) to (1.2,3.8) to  [out=-150, in=-30]  (.9,3.8) to (1,4);

\draw[] (1,0) -- (0,1);\draw[] (1,0) to [out=60,in=-60] (1,2);\draw[] (0,1) -- (1,2);\draw[] (1,2) -- (2,3);\draw[] (1,2) to [out=120,in=-120] (1,4);\draw[] (2,3) -- (1,4);
\draw[] (1,0)--(1.5,-.5);\draw[] (1,0)--(.5,-.5);
\draw[] (0,1)--(-.5,1.5);\draw[] (0,1)--(-.5,.5);
\draw[] (2,3)--(2.5,2.5);\draw[] (2,3)--(2.5,3.5);
\draw[] (1,4)--(.5,4.5);\draw[] (1,4)--(1.5,4.5);

\node at (1.5,-.3) {\scalebox{.6}[.6]{$\frac{{\uu}_3}{{\uu}_2}\hb$}};
\node at (0.8,-.4) {\scalebox{.6}[.6]{$\frac{{\uu}_3}{{\uu}_1}\hb$}};
\node at (1.2,.2) {\scalebox{.6}[.6]{$\frac{{\uu}_1}{{\uu}_3}$}};
\node at (0.8,.35) {\scalebox{.6}[.6]{$\frac{{\uu}_2}{{\uu}_3}$}};

\node at (.35,.8) {\scalebox{.6}[.6]{$\frac{{\uu}_3}{{\uu}_2}$}};
\node at (-.2,.6) {\scalebox{.6}[.6]{$\frac{{\uu}_2}{{\uu}_1}\hb$}};
\node at (.2,1.35) {\scalebox{.6}[.6]{$\frac{{\uu}_1}{{\uu}_2}$}};
\node at (-.5,1.3) {\scalebox{.6}[.6]{$\frac{{\uu}_2}{{\uu}_3}\hb$}};

\node at (1.3,1.7) {\scalebox{.6}[.6]{$\frac{{\uu}_3}{{\uu}_1}$}};
\node at (.8,1.6) {\scalebox{.6}[.6]{$\frac{{\uu}_2}{{\uu}_1}$}};
\node at (1.2,2.35) {\scalebox{.6}[.6]{$\frac{{\uu}_1}{{\uu}_2}\hb$}};
\node at (.7,2.3) {\scalebox{.6}[.6]{$\frac{{\uu}_1}{{\uu}_3}\hb$}};

\node at (2.5,2.7) {\scalebox{.6}[.6]{$\frac{{\uu}_3}{{\uu}_2}\hb$}};
\node at (1.8,2.55) {\scalebox{.6}[.6]{$\frac{{\uu}_2}{{\uu}_1}\hb^{-1}$}};
\node at (2.2,3.35) {\scalebox{.6}[.6]{$\frac{{\uu}_1}{{\uu}_2}\hb^2$}};
\node at (1.7,3.15) {\scalebox{.6}[.6]{$\frac{{\uu}_2}{{\uu}_3}$}};

\node at (1.45,3.7) {\scalebox{.6}[.6]{$\frac{{\uu}_3}{{\uu}_2}$}};
\node at (.6,3.6) {\scalebox{.6}[.6]{$\frac{{\uu}_3}{{\uu}_1}\hb^{-1}$}};
\node at (1.2,4.35) {\scalebox{.6}[.6]{$\frac{{\uu}_1}{{\uu}_3}\hb^2$}};
\node at (.35,4.4) {\scalebox{.6}[.6]{$\frac{{\uu}_2}{{\uu}_3}\hb$}};

\node[fill,white,draw,circle,minimum size=.5cm,inner sep=0pt] at (1,0) {\footnotesize $1$};
\node[fill,white,draw,circle,minimum size=.5cm,inner sep=0pt] at (0,1) {\footnotesize$2$};
\node[fill,white,draw,circle,minimum size=.5cm,inner sep=0pt] at (1,2) {\footnotesize$3$};
\node[fill,white,draw,circle,minimum size=.5cm,inner sep=0pt] at (2,3) {\footnotesize$4$};
\node[fill,white,draw,circle,minimum size=.5cm,inner sep=0pt] at (1,4) {\footnotesize$5$};
\node[draw,circle,minimum size=.5cm,inner sep=0pt] at (1,0) {\footnotesize $1$};
\node[draw,circle,minimum size=.5cm,inner sep=0pt] at (0,1) {\footnotesize$2$};
\node[draw,circle,minimum size=.5cm,inner sep=0pt] at (1,2) {\footnotesize$3$};
\node[draw,circle,minimum size=.5cm,inner sep=0pt] at (2,3) {\footnotesize$4$};
\node[draw,circle,minimum size=.5cm,inner sep=0pt] at (1,4) {\footnotesize$5$};

\end{tikzpicture}
\]
\caption{Illustration of $\T$ fixed points, invariant curves (with their $\T$ weights), poset structure, $\Leaf$s, $\Slope$s, and $N_f^+$ spaces on $\Ch($\ttt{/1$\backslash$1/2$\backslash$2$\backslash$2/}$)$ $=$ $\Naka((1,2),(1,2))$. See Section \ref{sec:ex1}.}  \label{fig:GKM1}
\end{figure}
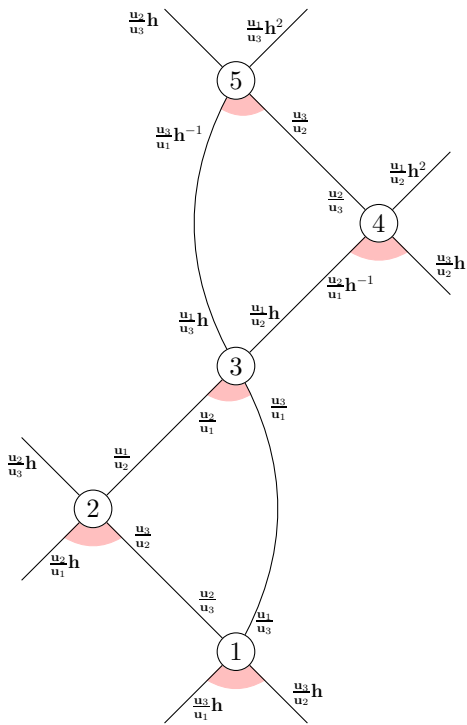

\subsection{Cohomology stable envelopes}
Our goal is to associate a characteristic class to every fixed point on $\Ch(\DD)$, which generalizes the notion of {\em stable envelope} for quiver varieties, and in turn, the {\em Chern-Schwartz-MacPherson class} of Schubert varieties in partial flag varieties (which itself is a generalization of the notion {\em Schubert class}). 
%Stable envelopes depend on a few parameters, here we choose to describe a version with many of the parameter values fixed, cf. Remark~\ref{rem:polariz}.  

Let $\sigma:\C^\times \to \T=\A\times \C^{\times}_{\h}$ be the one-parameter subgroup $\sigma(z)=(z,z^2,z^3,\ldots,z^n,1)$ (cf.~Remark~\ref{rem:polariz}).

\begin{definition} 
\begin{itemize}
\item 
For a fixed point $f\in \Ch(\DD)^{\T}$ define 
\[
\Leaf(f)=\{x\in \Ch(\DD) : \lim_{z\to 0} \sigma(z)x=f\}.
\]
\item Define the partial order on $\Ch(\DD)^{\T}$ by 
\[f'\leq f \qquad\text{if}\qquad f'\in\overline{\Leaf(f)}.\]
\item Define the slope (the ``stable leaf'' or ``full attracting set'') of a fixed point $f$ by
\[
\Slope(f)=\bigcup_{f'\leq f} \Leaf(f').
\]
\item Let $N_f^+\oplus N_f^-$ denote the $\T$ invariant decomposition of $T_f\Ch(\DD)$ to positive and negative $\sigma$ weight spaces. 
\end{itemize}
\end{definition}

At this point it is worth looking at examples, that is, Figures \ref{fig:GKM1}--\ref{fig:GKM3}. In these figures the vertices represent fixed points. The edges represent $\T$ invariant curves, and the decoration on the edge at a vertex is the $\T$ weight on the tangent line of the curve at that fixed point. The $\Leaf$s of the fixed points are cells, the tangent space of the $\Leaf$ at its own fixed point is indicated in the figures by the pink shading. The fixed points are positioned in such a way that the poset structure is illustrated the usual way: in the figures the vertexes are $\leq$-growing from bottom up. Hence, at a fixed point $f$ the weights of the lines spanning $N_f^+$ (respectively $N_f^-$) are the labels on the edges in the pink (respectively non-pink) region.

First let us recall the general axiomatic definition of cohomological stable envelopes of Maulik-Okounkov.  
\begin{definition}[\cite{MO}]
Let $f\in X^{\T}$. The cohomology class $\Stab(f)\in H^*_{\T}(X)$ of homogeneous degree $\dim_{\C}X$ is called the stable envelope of $f$, if it satisfies the axioms:
\begin{enumerate}
\item (support) it is supported on $Slope(f)$;
\item (normalization) $\Stab(f)|_f=e(N^-_f)$;
\item (boundary) $\Stab(f)|_{f'}$ is divisible by $\h$, for $f'\not= f$ .
\end{enumerate}
 \end{definition}

\noindent Since we are working with $\HH_{\T}$ instead of $H^*_{\T}$ we need a local reformulation of the support condition. For that we need the notion of normal bundle of $\Slope(f)$ at $f'$. 
\begin{definition}
We say that a $\T$ invariant line bundle $\ell$ ``belongs to $T_{f'}(\Slope(f))$''  either if it is tangent to an invariant curve connecting $f'$ with $f''$, for an $f''$ satisfying $f'< f'' \leq f$, or, if $\ell \subset N^+_{f'}$. The span of these line bundles is called $T_{f'}(\Slope(f))$. A $\T$ invariant complement of $T_{f'}(\Slope(f))$ in $T_{f'}\Ch(\DD)$ is called $N_{f'}(\Slope(f))$.
\end{definition}
 
\begin{example} \rm \label{ex:strangeN}
In Figure \ref{fig:GKM1} let $f$ be the fixed point denoted by 5. Then for $f'=4$ we have that $N_{f'}(\Slope(f))$ is one dimensional, with $\T$ weight $u_1-u_2+2\h$.  For $f'=3$ we have that $N_{f'}(\Slope(f))$ is zero dimensional.  For $f'=2$ we have that $N_{f'}(\Slope(f))$ is one dimensional, with weight $u_2-u_3+\h$. For $f'=1$ we have that $N_{f'}(\Slope(f))$ is zero dimensional.
\end{example}

\begin{definition}\label{def:SE}
Let $f\in \Ch(\DD)^{\T}$. The cohomology class $\Stab(f)\in \HH_{\T}(\Ch(\DD))$ of homogeneous degree $\dim_{\C}\Ch(\DD)$ is called the stable envelope of $f$, if it satisfies the axioms:
\begin{enumerate}
\item (support-1) $\Stab(f)|_{f'}=0$ if $f'\not\in \Slope(f)$;
\item (support-2) $\Stab(f)|_{f'}$ is divisible by $e(N_{f'}(\Slope(f)))$.
\item (normalization) $\Stab(f)|_f=e(N^-_f)$;
\item (boundary) $\Stab(f)|_{f'}$ is divisible by $\h$, for $f'\not= f$.
\end{enumerate}
 \end{definition}

The relation between the global support condition and the local ``support-2'' condition is the well known argument combining the Gysin sequence argument and a Mayer-Vietoris induction, see eg. \cite[Section 5.2]{RTV_EllK}. The local ``support-1'' condition is in fact a corollary of the rest of the axioms, yet, we listed it for clarity. If a stable envelope exists, then it is unique, the original proof \cite[Section 3.3.4]{MO} carries over to this case (see also \cite[Section 3.1]{RTVtrig}, \cite[Section 7.8]{RTV_EllK}). In the next three sections we show examples of stable envelopes.

\begin{remark} \label{rem:polariz}
The general notion of cohomological stable envelope \cite{MO} depends on other parameters: a `chamber', and a `polarization'. We fixed natural choices of these, e.g. our chamber choice is $u_1>u_2>\ldots>u_n$ (essentially $\sigma$ above). In future studies, the stable envelopes should be considered for different chambers as well. The comparison of stable envelopes for different chambers is what endows  $\HH_{\T}$ with the structure of a quantum group representation. K theory and elliptic cohomology stable envelopes depend on an additional set of parameters, the K\"ahler (or dynamical) parameters (locally constant dependence in $K_{\T}$, meromorphic dependence in Ell${}_{\T}$) \cite{O, AO}, cf. Section \ref{sec:elliptic}.
\end{remark}

\subsection{Stable envelopes for $\Ch($\ttt{\fs 1\bs 1\fs 2\bs 2\bs 2\fs}$)=$ $\Naka((1,2),(1,2))$.} \label{sec:ex1}
The ``skeleton'' of the $\dim_{\C}=4$ bow variety $\Ch(\ttt{\fs 1\bs 1\fs 2\bs 2\bs 2\fs})$ is in Figure \ref{fig:GKM1}, where the fixed points named 1, 2, 3, 4, 5 are 
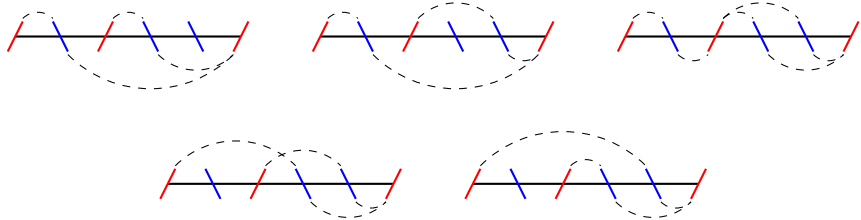
\[
\begin{tikzpicture}[baseline=0pt,scale=.2]
\draw[thick] (0,1)--(15,1) ;
\draw[thick,red] (-.5,0)--(.5,2);
\draw[thick,blue] (3.5,0)--(2.5,2);
\draw[thick,red] (5.5,0)--(6.5,2);
\draw[thick,blue] (9.5,0)--(8.5,2);
\draw[thick,blue] (12.5,0)--(11.5,2);
\draw[thick,red] (14.5,0)--(15.5,2);
\draw [dashed, black](3.5,-.2) to [out=-45,in=225] (14.5,-.2);
\draw [dashed, black](9.5,-.2) to [out=-45,in=225] (14.5,-.2);
\draw [dashed, black](0.5,2.2) to [out=45,in=-225] (2.5,2.2);
\draw [dashed, black](6.5,2.2) to [out=45,in=-225] (8.5,2.2);
\end{tikzpicture}
\qquad
\begin{tikzpicture}[baseline=0pt,scale=.2]
\draw[thick] (0,1)--(15,1) ;
\draw[thick,red] (-.5,0)--(.5,2);
\draw[thick,blue] (3.5,0)--(2.5,2);
\draw[thick,red] (5.5,0)--(6.5,2);
\draw[thick,blue] (9.5,0)--(8.5,2);
\draw[thick,blue] (12.5,0)--(11.5,2);
\draw[thick,red] (14.5,0)--(15.5,2);
\draw [dashed, black](3.5,-.2) to [out=-45,in=225] (14.5,-.2);
\draw [dashed, black](12.5,-.2) to [out=-45,in=225] (14.5,-.2);
\draw [dashed, black](0.5,2.2) to [out=45,in=-225] (2.5,2.2);
\draw [dashed, black](6.5,2.2) to [out=45,in=-225] (11.5,2.2);
\end{tikzpicture}
\qquad
\begin{tikzpicture}[baseline=0pt,scale=.2]
\draw[thick] (0,1)--(15,1) ;
\draw[thick,red] (-.5,0)--(.5,2);
\draw[thick,blue] (3.5,0)--(2.5,2);
\draw[thick,red] (5.5,0)--(6.5,2);
\draw[thick,blue] (9.5,0)--(8.5,2);
\draw[thick,blue] (12.5,0)--(11.5,2);
\draw[thick,red] (14.5,0)--(15.5,2);
\draw [dashed, black](12.5,-.2) to [out=-45,in=225] (14.5,-.2);
\draw [dashed, black](9.5,-.2) to [out=-45,in=225] (14.5,-.2);
\draw [dashed, black](3.5,-.2) to [out=-45,in=225] (5.5,-.2);
\draw [dashed, black](0.5,2.2) to [out=45,in=-225] (2.5,2.2);
\draw [dashed, black](6.5,2.2) to [out=45,in=-225] (8.5,2.2);
\draw [dashed, black](6.5,2.2) to [out=45,in=-225] (11.5,2.2);
\end{tikzpicture}
\]
\[
\begin{tikzpicture}[baseline=0pt,scale=.2]
\draw[thick] (0,1)--(15,1) ;
\draw[thick,red] (-.5,0)--(.5,2);
\draw[thick,blue] (3.5,0)--(2.5,2);
\draw[thick,red] (5.5,0)--(6.5,2);
\draw[thick,blue] (9.5,0)--(8.5,2);
\draw[thick,blue] (12.5,0)--(11.5,2);
\draw[thick,red] (14.5,0)--(15.5,2);
\draw [dashed, black](9.5,-.2) to [out=-45,in=225] (14.5,-.2);
\draw [dashed, black](12.5,-.2) to [out=-45,in=225] (14.5,-.2);
\draw [dashed, black](0.5,2.2) to [out=45,in=-225] (8.5,2.2);
\draw [dashed, black](6.5,2.2) to [out=45,in=-225] (11.5,2.2);
\end{tikzpicture}
\qquad
\begin{tikzpicture}[baseline=0pt,scale=.2]
\draw[thick] (0,1)--(15,1) ;
\draw[thick,red] (-.5,0)--(.5,2);
\draw[thick,blue] (3.5,0)--(2.5,2);
\draw[thick,red] (5.5,0)--(6.5,2);
\draw[thick,blue] (9.5,0)--(8.5,2);
\draw[thick,blue] (12.5,0)--(11.5,2);
\draw[thick,red] (14.5,0)--(15.5,2);
\draw [dashed, black](9.5,-.2) to [out=-45,in=225] (14.5,-.2);
\draw [dashed, black](12.5,-.2) to [out=-45,in=225] (14.5,-.2);
\draw [dashed, black](0.5,2.2) to [out=45,in=-225] (11.5,2.2);
\draw [dashed, black](6.5,2.2) to [out=45,in=-225] (8.5,2.2);
\end{tikzpicture}
\]
respectively. The stable envelopes are in the table (*)
\begin{equation*}
\begin{tabular}{|c||c|c|c|c|c|} 
\hline
& 1 & 2 & 3 & 4 & 5 \\
\hline\hline
1 & \tiny $(u_1\!-\!u_3)(u_2\!-\!u_3)$ & \tiny 0 & \tiny 0 & \tiny0 & \tiny 0  \\
\hline
2 & \tiny $(u_1\!-\!u_3)\h$ & \tiny $(u_1\!-\!u_2)(u_2\!-\!u_3\!+\!\h)$ & \tiny 0 & \tiny0 & \tiny 0  \\
\hline
3 & \tiny $(u_3\!-\!u_2\!+\!\h)\h$ & \tiny $(u_2\!-\!u_3\!+\!\h)\h$ & \tiny $(u_1\!-\!u_3\!+\!\h)(u_1\!-\!u_2\!+\!\h)$ & \tiny 0 & \tiny 0  \\
\hline
4 & \tiny $\h^2$ & \tiny $(u_2\!-\!u_3\!+\!\h)\h$ & \tiny $(u_1\!-\!u_3\!+\!\h)\h$ & \tiny $(u_2\!-\!u_3)(u_1\!-\!u_2\!+\!2\h)$ & \tiny 0  \\
\hline
5 & \tiny $(u_2\!-\!u_3)\h$ & \tiny 0 & \tiny $(u_2\!-\!u_1)\h$ & \tiny$(u_1\!-\!u_2\!+\!2\h)\h$ & \tiny $(u_1\!-\!u_3\!+\!2\h)(u_2\!-\!u_3\!+\!\h)$  \\
\hline
\end{tabular}.
\end{equation*}
In this, and similar tables in the whole paper, each row is a stable envelope for a fixed point. The entries in one row are the fixed point restrictions of that stable envelope.

To verify that this table is correct we need to verify that
\begin{enumerate}[(i)]
\item each row (as a five-tuple) is an element of $\HH_{\T}(\Ch(\DD))$; and that
\item the axioms of Definition \ref{def:SE} are satisfied.
\end{enumerate}

\noindent Property (i) is proved by applying the five fixed point restriction homomorphisms $\Loc_1, \ldots,\Loc_5$ (as described in Section \ref{sec:fixrest})
\[
\begin{tabular}{llllllll}
$x_{11}\mapsto u_1$ & $x_{21}\mapsto u_1$ & $x_{31}\mapsto u_1$ & $x_{32}\mapsto u_2$ & $x_{41} \mapsto u_1$ & $x_{42}\mapsto u_2$ & $x_{51}\mapsto u_1$ & $x_{52}\mapsto u_2$, \\
$x_{11}\mapsto u_1$ & $x_{21}\mapsto u_1$ & $x_{31}\mapsto u_1$ & $x_{32}\mapsto u_3$ & $x_{41} \mapsto u_1$ & $x_{42}\mapsto u_3$ & $x_{51}\mapsto u_1$ & $x_{52}\mapsto u_3$, \\
$x_{11}\mapsto u_1$ & $x_{21}\mapsto u_1$ & $x_{31}\mapsto u_2$ & $x_{32}\mapsto u_3$ & $x_{41} \mapsto u_2$ & $x_{42}\mapsto u_3$ & $x_{51}\mapsto u_2$ & $x_{52}\mapsto u_3$, \\
$x_{11}\mapsto u_2-\h$ & $x_{21}\mapsto u_2-\h$ & $x_{31}\mapsto u_2$ & $x_{32}\mapsto u_3$ & $x_{41} \mapsto u_2$ & $x_{42}\mapsto u_3$ & $x_{51}\mapsto u_2$ & $x_{52}\mapsto u_3$, \\
$x_{11}\mapsto u_3-\h$ & $x_{21}\mapsto u_3-\h$ & $x_{31}\mapsto u_2$ & $x_{32}\mapsto u_3$ & $x_{41} \mapsto u_2$ & $x_{42}\mapsto u_3$ & $x_{51}\mapsto u_2$ & $x_{52}\mapsto u_3$ \\
\end{tabular}
\]

\noindent to the concrete formulas 
\begin{equation}
\begin{aligned} 
F_1=& (x_{31}-u_3)(x_{32}-u_3), \\
F_2=& (x_{31}+x_{32}-u_2-u_3)(u_1+u_2-x_{31}-x_{32}+\h),  \\
F_3=& (x_{11}-x_{31}-x_{32}+u_3+\h)(x_{11}-x_{31}-x_{32}+u_2+\h), \label{eq:polys} \\
F_4=& (u_1-x_{11}+\h)(x_{11}-x_{31}-x_{32}+u_2+\h), \\
F_5=& (u_1-x_{11}+\h)(-x_{11}+x_{31}+x_{32}-u_3), 
\end{aligned}
\end{equation}
where $x_{i1},x_{i2},\ldots,x_{i,\mult_{X_i}}$ are the Chern roots of the $i$'th tautological bundle.

\begin{remark}
The polynomials in \eqref{eq:polys} are not unique, they are only defined up to $\cap_i \ker(\Loc_i)$. We chose `nice' representatives $F_j$, which in this case factor to linear factors. The existence of such nice representatives is not expected for more complicated brane diagrams. If $\Ch(\DD)$ is the cotangent bundle of a partial flag variety, then there are reasonably nice representatives of stable envelopes called ``weight functions,'' see \cite{RTV, RTVtrig, RTV_EllK, RV}.
\end{remark}

\begin{remark}
In fact we could argue differently to prove (i)---because this particular $\Ch(\DD)$ shares properties with GKM spaces. Namely, the five-tuple $(p_1,p_2,p_3,p_4,p_5)$ of polynomials in $\C[u_1,u_2,u_3,\h]$ is the $(\Loc_1,\ldots,\Loc_5)$-image of a polynomial 
\[
F \in \C[u_1,u_2,u_3,\h][x_{11},x_{21},x_{31},x_{32},x_{41},x_{42},x_{51},x_{52}]^{S_1\times S_1 \times S_1 \times S_2 \times S_2 \times S_2},
\]
if and only if 
\begin{equation}\label{eq:GKMconditions}
\begin{tabular}{lllll}
$(p_1-p_2)|_{u_2=u_3}=0$, && $(p_1-p_3)|_{u_1=u_3}=0$, && $(p_2-p_3)|_{u_1=u_2}=0$, 
\\
$(p_3-p_4)|_{u_2=u_1+\h}=0$, && $(p_3-p_5)|_{u_3=u_1+\h}=0$, &&  $(p_4-p_5)|_{u_2=u_3}=0$
\end{tabular}
\end{equation}
(these equations are read from the edges of the graph in Figure \ref{fig:GKM1}). To prove this statement consider a linear space $\C^{12}$ with coordinates 
\[
u_1,u_2,u_3,\h,x_{11},x_{21},x_{31}+x_{32}, x_{31}x_{32}, x_{41}+x_{42}, x_{41}x_{42},x_{51}+x_{52}, x_{51}x_{52}
\]  
and in it the five subvarieties $H_i$ defined by 
\[
x_{11}=\Loc_i(x_{11}), \ldots, x_{51}+x_{52}=\Loc_i(x_{51}+x_{52}), x_{51}x_{52}=\Loc_i(x_{51}x_{52})
\]
for $i=1,\ldots, 5$. The polynomials $p_i$ can be considered to be polynomials on the $H_i$'s, and the existence of $F$ is rephrased as the existence of a polynomial on $\C^{12}$ that {\em restricts} to the  given five-tuple. A necessary condition is, of course, that the $p_i$'s  agree on their pairwise intersections. These conditions are exactly \eqref{eq:GKMconditions}. It can be shown that our varieties $H_i$ intersect in such a general way that guarantees that the named necessary conditions are also sufficient. (For more sophisticated intersections of varieties such a statement does not hold, for example consider the polynomials $c_jx$ on the lines $y=jx$ for $j=1,2,3$, in the $x,y$-plane. They agree on their intersection, but they only extend to a polynomial in $x,y$ if $c_1+c_3=2c_2$.) In fact the point of view of this Remark is used in the {\em definition} of equivariant elliptic cohomology, namely the elliptic counterpart of $\cup H_i$ is called the {\em elliptic cohomology scheme} cf. \cite[\S 2.2-2.3]{AO}, \cite[\S 4]{FRV}, \cite[\S 7]{RTV_EllK}, \cite[\S 2]{RSVZ1}.
\end{remark}

Now that (i) is verified for the table (*) above, it is worth verifying property (ii), that is the axioms of stable envelopes. 

The {\em normalization axiom} is about the diagonal entries: for each vertex on the graph the directions of $N^-$ are those that are {\em not} covered by the pink shading in the figure. Hence the diagonal entries need to be the products of their weights. 

The {\em boundary axiom} holds because all below-diagonal entries are divisible by $\h$. 

The {\em support-1 axiom} holds because the above-diagonal entries are all 0. 

The {\em support-2 axiom} is a divisibility requirement for below-diagonal entries. Continuing Example~\ref{ex:strangeN} we see that the axiom requires that the (5,4), (5,3), (5,2), (5,1) entries of the table (*) are divisible by $u_1-u_2+2\h$, $1$, $u_2-u_3+\h$, $1$, respectively. 

\begin{figure}
\[
 \begin{tikzpicture}[scale=1.6]
 
    \draw[fill=pink,pink] (3,0) to [out=-45, in=135]  (3.3,-.3) to [out=-160, in=-20] (2.7,-.3) to [out=45, in=-135] (3,0);  
    \draw[fill=pink,pink] (2,1) to [out=-45, in=135]  (2.3,0.7) to [out=-160, in=-20] (1.7,0.7) to [out=45, in=-135] (2,1);  
    \draw[fill=pink,pink] (1,2) to [out=-45, in=135]  (1.3,1.7) to [out=-160, in=-20] (0.7,1.7) to [out=45, in=-135] (1,2);  
    \draw[fill=pink,pink] (0,3) to [out=-45, in=135]  (0.3,2.7) to [out=-160, in=-20] (-.3,2.7) to [out=45, in=-135] (0,3);  
    \draw[fill=pink,pink] (3,2) to [out=-45, in=135]  (3.3,1.7) to [out=-160, in=-20] (2.7,1.7) to [out=45, in=-135] (3,2);  
    \draw[fill=pink,pink] (2,3) to [out=-45, in=135]  (2.3,2.7) to [out=-160, in=-20] (1.7,2.7) to [out=45, in=-135] (2,3);  
    \draw[fill=pink,pink] (1,4) to [out=-45, in=135]  (1.3,3.7) to [out=-160, in=-20] (0.7,3.7) to [out=45, in=-135] (1,4);      
    \draw[fill=pink,pink] (3,4) to [out=-45, in=135]  (3.3,3.7) to [out=-160, in=-20] (2.7,3.7) to [out=45, in=-135] (3,4);      
    \draw[fill=pink,pink] (2,5) to [out=-45, in=135]  (2.3,4.7) to [out=-160, in=-20] (1.7,4.7) to [out=45, in=-135] (2,5);      
    \draw[fill=pink,pink] (3,6) to [out=-45, in=135]  (3.3,5.7) to [out=-160, in=-20] (2.7,5.7) to [out=45, in=-135] (3,6);

\draw[] (3.5,-.5) -- (-0.5,3.5);
\draw[] (3.5,1.5) -- (.5,4.5);
\draw[] (3,4) -- (1.5,5.5);     

\draw [](3,0) to [out=40,in=-40] (3,4);    
\draw [](3,2) to [out=40,in=-40] (3,6);   
     
\draw[] (1.5,.5) -- (3,2);
\draw[] (.5,1.5) -- (3.5,4.5);     
\draw[] (-.5,2.5) -- (3.5,6.5); 
\draw[] (3,0) -- (2.5,-.5);
\draw[] (3,6) -- (2.5,6.5);

   \node[] at (2.5,.2) {\scalebox{.6}[.6]{$\frac{\uu_1}{\uu_2}\hb^{-2}$}};
   \node[] at (2.6,-.2) {\scalebox{.6}[.6]{$\frac{\uu_2}{\uu_1}\hb^{2}$}};
   \node[] at (3.5,.2) {\scalebox{.6}[.6]{$\frac{\uu_1}{\uu_2}\hb^{-1}$}};
   \node[] at (3.44,-.2) {\scalebox{.6}[.6]{$\frac{\uu_2}{\uu_1}\hb^3$}};
   
   \node[] at (1.5,1.2) {\scalebox{.6}[.6]{$\frac{\uu_1}{\uu_2}\hb^{-1}$}};
   \node[] at (1.6,.8) {\scalebox{.6}[.6]{$\frac{\uu_2}{\uu_1}\hb^{2}$}};
   \node[] at (2.55,1.2) {\scalebox{.6}[.6]{$\frac{\uu_1}{\uu_2}\hb^{-1}$}};
   \node[] at (2.45,.8) {\scalebox{.6}[.6]{$\frac{\uu_2}{\uu_1}\hb^2$}};

   \node[] at (0.6,2.2) {\scalebox{.6}[.6]{$\frac{\uu_1}{\uu_2}$}};
   \node[] at (0.6,1.8) {\scalebox{.6}[.6]{$\frac{\uu_2}{\uu_1}\hb^2$}};
   \node[] at (1.15,2.32) {\scalebox{.6}[.6]{$\frac{\uu_1}{\uu_2}\hb^{-1}$}};
   \node[] at (1.2,1.6) {\scalebox{.6}[.6]{$\frac{\uu_2}{\uu_1}\hb$}};

   \node[] at (2.9,2.4) {\scalebox{.6}[.6]{$\frac{\uu_1}{\uu_2}\hb^{-1}$}};
   \node[] at (2.8,1.55) {\scalebox{.6}[.6]{$\frac{\uu_2}{\uu_1}\hb$}};
   \node[] at (3.4,2.2) {\scalebox{.6}[.6]{$\frac{\uu_1}{\uu_2}$}};
   \node[] at (3.45,1.8) {\scalebox{.6}[.6]{$\frac{\uu_2}{\uu_1}\hb^2$}};
   
   \node[] at (1.6,3.2) {\scalebox{.6}[.6]{$\frac{\uu_1}{\uu_2}$}};
   \node[] at (1.6,2.8) {\scalebox{.6}[.6]{$\frac{\uu_2}{\uu_1}\hb$}};
   \node[] at (2.4,3.2) {\scalebox{.6}[.6]{$\frac{\uu_1}{\uu_2}$}};
   \node[] at (2.4,2.8) {\scalebox{.6}[.6]{$\frac{\uu_2}{\uu_1}\hb$}};
   
   \node[] at (-.4,3.2) {\scalebox{.6}[.6]{$\frac{\uu_1}{\uu_2}\hb$}};
   \node[] at (-.4,2.8) {\scalebox{.6}[.6]{$\frac{\uu_2}{\uu_1}\hb^2$}};
   \node[] at (.5,3.2) {\scalebox{.6}[.6]{$\frac{\uu_1}{\uu_2}\hb^{-1}$}};
   \node[] at (.4,2.8) {\scalebox{.6}[.6]{$\frac{\uu_2}{\uu_1}$}};
   
   \node[] at (.6,4.2) {\scalebox{.6}[.6]{$\frac{\uu_1}{\uu_2}\hb$}};
   \node[] at (.6,3.8) {\scalebox{.6}[.6]{$\frac{\uu_2}{\uu_1}\hb$}};
   \node[] at (1.2,4.35) {\scalebox{.6}[.6]{$\frac{\uu_1}{\uu_2}$}};
   \node[] at (1.25,3.6) {\scalebox{.6}[.6]{$\frac{\uu_2}{\uu_1}$}};
   
   \node[] at (2.85,4.3) {\scalebox{.6}[.6]{$\frac{\uu_1}{\uu_2}$}};
   \node[] at (2.75,3.55) {\scalebox{.6}[.6]{$\frac{\uu_2}{\uu_1}$}};
   \node[] at (3.45,4.2) {\scalebox{.6}[.6]{$\frac{\uu_1}{\uu_2}\hb$}};
   \node[] at (3.4,3.8) {\scalebox{.6}[.6]{$\frac{\uu_2}{\uu_1}\hb$}};
   
   \node[] at (1.6,5.2) {\scalebox{.6}[.6]{$\frac{\uu_1}{\uu_2}\hb$}};
   \node[] at (1.6,4.8) {\scalebox{.6}[.6]{$\frac{\uu_2}{\uu_1}$}};
   \node[] at (2.45,5.2) {\scalebox{.6}[.6]{$\frac{\uu_1}{\uu_2}\hb$}};
   \node[] at (2.4,4.8) {\scalebox{.6}[.6]{$\frac{\uu_2}{\uu_1}$}};
   
   \node[] at (2.6,6.2) {\scalebox{.6}[.6]{$\frac{\uu_1}{\uu_2}\hb$}};
   \node[] at (2.6,5.8) {\scalebox{.6}[.6]{$\frac{\uu_2}{\uu_1}\hb^{-1}$}};
   \node[] at (3.45,6.2) {\scalebox{.6}[.6]{$\frac{\uu_1}{\uu_2}\hb^2$}};
   \node[] at (3.4,5.8) {\scalebox{.6}[.6]{$\frac{\uu_2}{\uu_1}$}};

   \draw [](2.2,1.2) to [out=53,in=-53] (2.2,2.8);  
   \draw [](2.1,1.12) to [out=76,in=-76] (2.1,2.88);  
   \draw [](2,1.15) to [out=90,in=-90] (2,2.85);  
   \draw [](1.9,1.12) to [out=104,in=-104] (1.9,2.88);  
   \draw [](1.8,1.2) to [out=127,in=-127] (1.8,2.8);  
   
   \draw [](2.2,3.2) to [out=53,in=-53] (2.2,4.8);  
   \draw [](2.1,3.12) to [out=76,in=-76] (2.1,4.88);  
   \draw [](2,3.15) to [out=90,in=-90] (2,4.85);  
   \draw [](1.9,3.12) to [out=104,in=-104] (1.9,4.88);  
   \draw [](1.8,3.2) to [out=127,in=-127] (1.8,4.8);  
   
    \draw [](2,1) to (2.35,0.5);  
    \draw [](2,1) to (2.2,0.5);  
    \draw [](2,1) to (2 ,0.5); 
    \draw [](2,1) to (1.8,0.5);  
    \draw [](2,1) to (1.65,0.5); 
    
    \draw [](2,5) to (2.35,5.5);  
    \draw [](2,5) to (2.2,5.5);  
    \draw [](2,5) to (2 ,5.5); 
    \draw [](2,5) to (1.8,5.5);  
    \draw [](2,5) to (1.65,5.5); 

  \node[fill,white,draw,circle,minimum size=.5cm,inner sep=0pt] at (3,0) {\footnotesize$45$};
  \node[fill,white,draw,circle,minimum size=.5cm,inner sep=0pt] at (2,1) {\footnotesize$35$};
  \node[fill,white,draw,circle,minimum size=.5cm,inner sep=0pt] at (3,2) {\footnotesize$34$};
  \node[fill,white,draw,circle,minimum size=.5cm,inner sep=0pt] at (1,2) {\footnotesize$25$};
  \node[fill,white,draw,circle,minimum size=.5cm,inner sep=0pt] at (2,3) {\footnotesize$24$};
  \node[fill,white,draw,circle,minimum size=.5cm,inner sep=0pt] at (3,4) {\footnotesize$23$};
  \node[fill,white,draw,circle,minimum size=.5cm,inner sep=0pt] at (0,3) {\footnotesize$15$};
  \node[fill,white,draw,circle,minimum size=.5cm,inner sep=0pt] at (1,4) {\footnotesize$14$};
  \node[fill,white,draw,circle,minimum size=.5cm,inner sep=0pt] at (2,5) {\footnotesize$13$};
  \node[fill,white,draw,circle,minimum size=.5cm,inner sep=0pt] at (3,6) {\footnotesize$12$};
  \node[draw,circle,minimum size=.5cm,inner sep=0pt] at (3,0) {\footnotesize$45$};
  \node[draw,circle,minimum size=.5cm,inner sep=0pt] at (2,1) {\footnotesize$35$};
  \node[draw,circle,minimum size=.5cm,inner sep=0pt] at (3,2) {\footnotesize$34$};
  \node[draw,circle,minimum size=.5cm,inner sep=0pt] at (1,2) {\footnotesize$25$};
  \node[draw,circle,minimum size=.5cm,inner sep=0pt] at (2,3) {\footnotesize$24$};
  \node[draw,circle,minimum size=.5cm,inner sep=0pt] at (3,4) {\footnotesize$23$};
  \node[draw,circle,minimum size=.5cm,inner sep=0pt] at (0,3) {\footnotesize$15$};
  \node[draw,circle,minimum size=.5cm,inner sep=0pt] at (1,4) {\footnotesize$14$};
  \node[draw,circle,minimum size=.5cm,inner sep=0pt] at (2,5) {\footnotesize$13$};
  \node[draw,circle,minimum size=.5cm,inner sep=0pt] at (3,6) {\footnotesize$12$};

  \end{tikzpicture}
 \]
\caption{Illustration of $\T$ fixed points, invariant curves (with their $\T$ weights), poset structure, $\Leaf$s, $\Slope$s, and $N_f^+$ spaces on $\Ch(\ttt{/1/2/3/4/5$\backslash$2$\backslash$})$, which is the 3d mirror dual of $T^*\!\Gr(2,5)$. See Section \ref{sec:ex2}.}  \label{fig:GKM2} 
\end{figure}

\subsection{Stable envelopes for $\Ch($\ttt{\fs 1\fs 2\fs 3\fs 4\fs 5\bs 2\bs}$)$.} \label{sec:ex2}
Consider $\DD=\ttt{\fs 1\fs 2\fs 3\fs 4\fs 5\bs 2\bs}$. The corres\-ponding bow variety is (Hanany-Witten isomorphic to) the 3d mirror dual of $T^*\!\Gr(2,5)$. Its $\T$ fixed points are in bijection with the 2-element subsets of  $\{1,\ldots,5\}$. The tie diagram corresponding to $\{k,l\}$ consists of 5 ties: $U_2$ is connected with $V_k$ and $V_l$, and $U_1$ is connected with $V_i$ for $i\not= k,l$. We denote this fixed point by $kl$. Figure \ref{fig:GKM2} illustrates relevant information on the fixed point data.

In the table below we name the stable envelopes in the same manner as in Section \ref{sec:ex1}. We used the following conventions: both horizontally and vertically we used the  45, 35, 34, 25, 24, 23, 15, 14, 13, 12 order of the vertices, and for brevity we write $u_{ij}^{(k)}$ for $(u_i-u_j+k\h)$.

\begin{equation*}
\begin{tabular}{|c|c|c|c|c|c|c|c|c|c|} 
\hline 45 &  35 & 34& 25& 24& 23& 15& 14& 13& 12 \\
\hline\hline
\tiny $u_{12}^{(-1)}u_{12}^{(-2)} $ & \tiny 0 & \tiny 0 & \tiny 0 & \tiny 0 & \tiny 0 & \tiny 0 & \tiny 0  & \tiny 0 & \tiny 0 
\\
\hline
\tiny $u_{12}^{(-1)}\h$ & \tiny $u_{12}^{(-1)}u_{12}^{(-1)}$ & \tiny 0 & \tiny 0 & \tiny 0 & \tiny 0 & \tiny 0 & \tiny 0  & \tiny 0 & \tiny 0 \\
\hline
\tiny $u_{12}^{(-1)}\h$ & \tiny $u_{12}^{(-1)}\h$ & \tiny $u_{12}^{(0)}u_{12}^{(-1)}$ & \tiny 0 & \tiny 0 & \tiny 0 & \tiny 0 & \tiny 0  & \tiny 0 & \tiny 0 
\\
\hline
\tiny $u_{12}^{(-1)}\h$ & \tiny $u_{12}^{(-1)}\h$ & \tiny 0 & \tiny $u_{12}^{(0)}u_{12}^{(-1)}$ &  \tiny 0 & \tiny 0 & \tiny 0 & \tiny 0  & \tiny 0 & \tiny 0 
\\
\hline
\tiny $u_{12}^{(-1)}\h$ & \tiny $\h^2$ & \tiny $u_{12}^{(0)}\h$ & \tiny $u_{12}^{(0)}\h$ &  \tiny $u_{12}^{(0)}u_{12}^{(0)}$ & \tiny 0 & \tiny 0 & \tiny 0  & \tiny 0 & \tiny 0 
\\
\hline
\tiny $2\h^2$ & \tiny $u_{12}^{(0)}\h$ & \tiny $u_{12}^{(0)}\h$ & \tiny $u_{12}^{(0)}\h$ &  \tiny $u_{12}^{(0)}\h$ & \tiny $u_{12}^{(0)}u_{12}^{(1)}$ & \tiny 0 & \tiny 0  & \tiny 0 & \tiny 0 
\\
\hline
\tiny $u_{12}^{(-1)}\h$ & \tiny $u_{12}^{(-1)}\h$ & \tiny $0$ & \tiny $u_{12}^{(-1)}\h$ &  \tiny $0$ & \tiny $0$ & \tiny $u_{12}^{(-1)}u_{12}^{(1)}$ & \tiny 0  & \tiny 0 & \tiny 0 
\\
\hline
\tiny $u_{12}^{(-1)}\h$ & \tiny $\h^2$ & \tiny $u_{12}^{(0)}\h$ & \tiny $\h^2$ &  \tiny $u_{12}^{(0)}\h$ & \tiny $0$ & \tiny $u_{12}^{(1)}\h$ & \tiny $u_{12}^{(0)} u_{12}^{(1)}$  & \tiny 0 & \tiny 0 
\\
\hline
\tiny $2\h^2$ & \tiny $u_{12}^{(0)}\h$ & \tiny $u_{12}^{(0)}\h$ & \tiny $\h^2$ &  \tiny $\h^2$ & \tiny $u_{12}^{(1)}\h$ & \tiny $u_{12}^{(1)}\h$ &\tiny $u_{12}^{(1)}\h$ &\tiny $u_{12}^{(1)}u_{12}^{(1)}$ & \tiny $0$ 
\\
\hline
\tiny $2\h^2$ & \tiny $2\h^2$ & \tiny $2\h^2$ & \tiny $u_{12}^{(1)}\h$ &  \tiny $u_{12}^{(1)}\h$ & \tiny $u_{12}^{(1)}\h$ & \tiny $u_{12}^{(1)}\h$ &\tiny $u_{12}^{(1)}\h$ &\tiny $ u_{12}^{(1)}\h$ & \tiny $u_{12}^{(1)}u_{12}^{(2)}$ 
\\ 
\hline
\end{tabular}
\end{equation*} 

To prove that the values of this table are correct we need to prove the properties (i) and (ii) as in Section \ref{sec:ex1}. Property (ii) is done by observation. Property (i) is more sophisticated: either we use computer to find representatives of the rows as polynomials in the $x_{ij}, u_i, \h$ variables, %(as we did), 
or one can use a formula presented in \cite[\S 5]{RSVZ1} specifically for stable envelopes on 3d mirror duals of $T^*\!\Gr$ spaces.

\begin{remark}\label{rem:nonGKM} 
Note that property (i) can {\em not} be concluded by checking that neighboring components agree up to a substitution. This bow variety has infinitely many invariant curves, so it is not a GKM space. The subvarieties analogous to those called $H_i$ in Section \ref{sec:ex1} do not intersect transversally. Hence the GKM conditions on the components of an element in the image of $\Loc$ must be generalized to some coincidences of higher derivatives.
\end{remark}

\subsection{Stable envelopes for $\Ch($\ttt{\bs 1\fs 2\fs 2\bs 2\bs 1\fs}$)$.} \label{sec:ex3}
Since the examples of Sections \ref{sec:ex1}, \ref{sec:ex2} were both quiver varieties, we present one more example, for which neither $\DD$ nor its 3d mirror dual are Hanany-Witten equivalent to a quiver variety: $\DD=\ttt{\bs 1\fs 2\fs 2\bs 2\bs 1\fs}$. Calculation shows that its fixed point data relevant for stable envelopes is in Figure \ref{fig:GKM3}, where the five fixed points denoted by 1, 2, 3, 4, 5 are
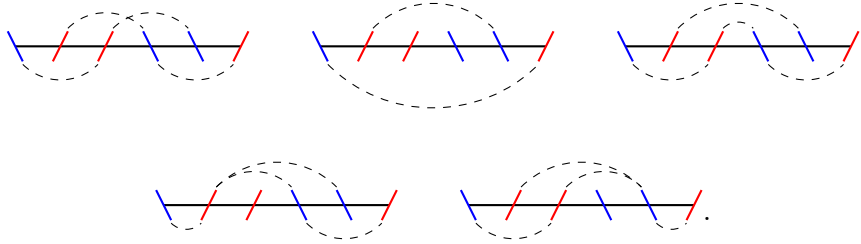
\[
\begin{tikzpicture}[baseline=0pt,scale=.2]
\draw[thick] (0,1)--(15,1) ;
\draw[thick,blue] (.5,0)--(-.5,2);
\draw[thick,red] (2.5,0)--(3.5,2);
\draw[thick,red] (5.5,0)--(6.5,2);
\draw[thick,blue] (9.5,0)--(8.5,2);
\draw[thick,blue] (12.5,0)--(11.5,2);
\draw[thick,red] (14.5,0)--(15.5,2);
\draw [dashed, black](.5,-.2) to [out=-45,in=225] (5.5,-.2);
\draw [dashed, black](9.5,-.2) to [out=-45,in=225] (14.5,-.2);
\draw [dashed, black](3.5,2.2) to [out=45,in=-225] (8.5,2.2);
\draw [dashed, black](6.5,2.2) to [out=45,in=-225] (11.5,2.2);
\end{tikzpicture}
\qquad
\begin{tikzpicture}[baseline=0pt,scale=.2]
\draw[thick] (0,1)--(15,1) ;
\draw[thick,blue] (.5,0)--(-.5,2);
\draw[thick,red] (2.5,0)--(3.5,2);
\draw[thick,red] (5.5,0)--(6.5,2);
\draw[thick,blue] (9.5,0)--(8.5,2);
\draw[thick,blue] (12.5,0)--(11.5,2);
\draw[thick,red] (14.5,0)--(15.5,2);
\draw [dashed, black](.5,-.2) to [out=-45,in=225] (14.5,-.2);
%\draw [dashed, black](9.5,-.2) to [out=-45,in=225] (14.5,-.2);
\draw [dashed, black](3.5,2.2) to [out=45,in=-225] (11.5,2.2);
%\draw [dashed, black](6.5,2.2) to [out=45,in=-225] (11.5,2.2);
\end{tikzpicture}
\qquad
\begin{tikzpicture}[baseline=0pt,scale=.2]
\draw[thick] (0,1)--(15,1) ;
\draw[thick,blue] (.5,0)--(-.5,2);
\draw[thick,red] (2.5,0)--(3.5,2);
\draw[thick,red] (5.5,0)--(6.5,2);
\draw[thick,blue] (9.5,0)--(8.5,2);
\draw[thick,blue] (12.5,0)--(11.5,2);
\draw[thick,red] (14.5,0)--(15.5,2);
\draw [dashed, black](.5,-.2) to [out=-45,in=225] (5.5,-.2);
\draw [dashed, black](9.5,-.2) to [out=-45,in=225] (14.5,-.2);
\draw [dashed, black](3.5,2.2) to [out=45,in=-225] (11.5,2.2);
\draw [dashed, black](6.5,2.2) to [out=45,in=-225] (8.5,2.2);
\end{tikzpicture}
\]
\[
\begin{tikzpicture}[baseline=0pt,scale=.2]
\draw[thick] (0,1)--(15,1) ;
\draw[thick,blue] (.5,0)--(-.5,2);
\draw[thick,red] (2.5,0)--(3.5,2);
\draw[thick,red] (5.5,0)--(6.5,2);
\draw[thick,blue] (9.5,0)--(8.5,2);
\draw[thick,blue] (12.5,0)--(11.5,2);
\draw[thick,red] (14.5,0)--(15.5,2);
\draw [dashed, black](.5,-.2) to [out=-45,in=225] (2.5,-.2);
\draw [dashed, black](9.5,-.2) to [out=-45,in=225] (14.5,-.2);
\draw [dashed, black](3.5,2.2) to [out=45,in=-225] (8.5,2.2);
\draw [dashed, black](3.5,2.2) to [out=45,in=-225] (11.5,2.2);
\end{tikzpicture}
\qquad
\begin{tikzpicture}[baseline=0pt,scale=.2]
\draw[thick] (0,1)--(15,1) ;
\draw[thick,blue] (.5,0)--(-.5,2);
\draw[thick,red] (2.5,0)--(3.5,2);
\draw[thick,red] (5.5,0)--(6.5,2);
\draw[thick,blue] (9.5,0)--(8.5,2);
\draw[thick,blue] (12.5,0)--(11.5,2);
\draw[thick,red] (14.5,0)--(15.5,2);
\draw [dashed, black](.5,-.2) to [out=-45,in=225] (5.5,-.2);
\draw [dashed, black](12.5,-.2) to [out=-45,in=225] (14.5,-.2);
\draw [dashed, black](3.5,2.2) to [out=45,in=-225] (11.5,2.2);
\draw [dashed, black](6.5,2.2) to [out=45,in=-225] (11.5,2.2);
\end{tikzpicture}.
\]
By methods similar to those in Section \ref{sec:ex2}, it can be verified that the stable envelopes are

\begin{equation*}
\begin{tabular}{|c||c|c|c|c|c|} 
\hline
 & 1 & 2 & 3 &4 &5 
 \\
 \hline\hline
1 & \tiny $u_{23}^{(0)}u_{13}^{(2)}u_{23}^{(1)} $ & \tiny 0 & \tiny 0 & \tiny 0 & \tiny 0 
\\
\hline
2 & \tiny 0 & \tiny $u_{13}^{(2)}u_{12}^{(1)}u_{23}^{(1)}$ & \tiny 0 & \tiny 0 & \tiny 0 
\\
\hline
3 & \tiny $u_{13}^{(2)}u_{23}^{(1)}\h$ & \tiny $u_{13}^{(2)}u_{23}^{(1)}\h$ & \tiny $u_{12}^{(2)}u_{23}^{(1)}u_{23}^{(1)}$ & \tiny 0 & \tiny 0 
\\ \hline
4 & \tiny $u_{32}^{(1)}u_{23}^{(1)}\h$ & \tiny $u_{13}^{(2)}u_{23}^{(1)}\h$ & \tiny $u_{23}^{(1)}u_{23}^{(1)} \h $ & \tiny $u_{13}^{(3)}u_{12}^{(3)}u_{23}^{(1)}$ & \tiny 0 
\\ \hline
5 & \tiny $u_{13}^{(2)}u_{23}^{(1)}\h$ & \tiny $u_{21}^{(0)}u_{23}^{(1)}\h$ & \tiny $u_{12}^{(2)}u_{23}^{(1)} \h $ & \tiny 0& \tiny $u_{13}^{(3)}u_{23}^{(1)}u_{23}^{(2)}$  
\\ \hline
\end{tabular}.
\end{equation*}
It is instructive to verify the axioms from Definition \ref{def:SE} just by checking the entries against the structure of Figure \ref{fig:GKM3}. The fact that each line is in fact an element of $\HH_{\T}(\Ch(\DD))$ is verified by computer calculation.

\begin{figure}
\[
\begin{tikzpicture}[scale=2.5]

\draw[fill=pink,pink] (0,0) to (0.2,-0.2) to  [out=-150, in=-30]  (-.2,-0.2) to (0,0);
\draw[fill=pink,pink] (1,1) to (1.2,0.8) to  [out=-150, in=-30]  (0.8,0.8) to (1,1);
\draw[fill=pink,pink] (0,2) to (0.2,1.8) to  [out=-150, in=-30]  (-.2,1.8) to (0,2);
\draw[fill=pink,pink] (2,2) to (2.2,1.8) to  [out=-150, in=-30]  (1.8,1.8) to (2,2);
\draw[fill=pink,pink] (2,0) to (2.2,-0.2) to  [out=-150, in=-30]  (1.8,-0.2) to (2,0);
   
\draw[] (.07,.07) -- (.93,.93);
\draw[] (.07,1.93) -- (.93,1.07);
\draw[] (1.07,0.93) -- (1.93,0.07);
\draw[] (1.07,1.07) -- (1.93,1.93);
\draw[] (0,.1) -- (0,1.9);
\draw[] (2,.1) -- (2,1.9);

\draw[] (-.07,.07) -- (-.5,.5);
\draw[] (-.07,-.07) -- (-.5,-.5);
\draw[] (0,-.1) -- (0,-.5);
\draw[] (.07,-.07) -- (.5,-.5);

\draw[] (2.07,.07) -- (2.5,.5);
\draw[] (2.07,-.07) -- (2.5,-.5);
\draw[] (1.93,-.07) -- (1.5,-.5);
\draw[] (2,-.1) -- (2,-.5);

\draw[] (-.07,2.07) -- (-.5,2.5);
\draw[] (-.07,1.93) -- (-.5,1.5);
\draw[] (0,2.1) -- (0,2.5);
\draw[] (.07,2.07) -- (.5,2.5);

\draw[] (2.07,2.07) -- (2.5,2.5);
\draw[] (2.07,1.93) -- (2.5,1.5);
\draw[] (1.93,2.07) -- (1.5,2.5);
\draw[] (2,2.1) -- (2,2.5);

\draw[] (1,1) -- (1,1.6); \draw[] (1,1) -- (1.1,1.6); \draw[] (1,1) -- (1.2,1.6);   \draw[] (1,1) -- (1.3,1.6);  \draw[] (1,1) -- (1.4,1.6);  \draw[] (1,1) -- (1.5,1.6);
\draw[] (1,1) -- (1,.4);\draw[] (1,1) -- (.9,.4);\draw[] (1,1) -- (.8,.4);\draw[] (1,1) -- (.7,.4);\draw[] (1,1) -- (.6,.4);\draw[] (1,1) -- (.5,.4);

\node[fill,white,draw,circle,minimum size=.5cm,inner sep=0pt] at (0,0) {\footnotesize $1$};
\node[fill,white,draw,circle,minimum size=.5cm,inner sep=0pt] at (1,1) {$3$};
\node[fill,white,draw,circle,minimum size=.5cm,inner sep=0pt] at (2,0) {$2$};
\node[fill,white,draw,circle,minimum size=.5cm,inner sep=0pt] at (0,2) {$4$};
\node[fill,white,draw,circle,minimum size=.5cm,inner sep=0pt] at (2,2) {$5$};
\node[draw,circle,minimum size=.5cm,inner sep=0pt] at (0,0) {\footnotesize $1$};
\node[draw,circle,minimum size=.5cm,inner sep=0pt] at (1,1) {\footnotesize$3$};
\node[draw,circle,minimum size=.5cm,inner sep=0pt] at (2,0) {\footnotesize$2$};
\node[draw,circle,minimum size=.5cm,inner sep=0pt] at (0,2) {\footnotesize$4$};
\node[draw,circle,minimum size=.5cm,inner sep=0pt] at (2,2) {\footnotesize$5$};

\node at (1.38,1.2) {\scalebox{.6}[.6]{$\frac{{\uu}_2}{{\uu}_3}\hb$}};
\node at (.91,1.4) {\scalebox{.6}[.6]{$\frac{{\uu}_2}{{\uu}_3}\hb$}};
\node at (.6,1.22) {\scalebox{.6}[.6]{$\frac{{\uu}_1}{{\uu}_2}\hb^2$}};
\node at (.65,.75) {\scalebox{.6}[.6]{$\frac{{\uu}_3}{{\uu}_2}$}};
\node at (1.05,.6) {\scalebox{.6}[.6]{$\frac{{\uu}_3}{{\uu}_2}$}};
\node at (1.4,.75) {\scalebox{.6}[.6]{$\frac{{\uu}_2}{{\uu}_1}\hb^{-1}$}};

\node at (2.4,2.2) {\scalebox{.6}[.6]{$\frac{{\uu}_2}{{\uu}_3}\hb^2$}};
\node at (1.91,2.4) {\scalebox{.6}[.6]{$\frac{{\uu}_2}{{\uu}_3}\hb$}};
\node at (1.6,2.22) {\scalebox{.6}[.6]{$\frac{{\uu}_1}{{\uu}_3}\hb^3$}};
\node at (1.65,1.75) {\scalebox{.6}[.6]{$\frac{{\uu}_3}{{\uu}_2}\hb^{-1}$}};
\node at (2.13,1.6) {\scalebox{.6}[.6]{$\frac{{\uu}_3}{{\uu}_1}\hb^{-2}$}};
\node at (2.4,1.75) {\scalebox{.6}[.6]{$\frac{{\uu}_3}{{\uu}_2}$}};

\node at (2.4,.2) {\scalebox{.6}[.6]{$\frac{{\uu}_2}{{\uu}_3}\hb$}};
\node at (1.9,.4) {\scalebox{.6}[.6]{$\frac{{\uu}_1}{{\uu}_3}\hb^2$}};
\node at (1.6,.22) {\scalebox{.6}[.6]{$\frac{{\uu}_1}{{\uu}_2}\hb$}};
\node at (1.65,-.25) {\scalebox{.6}[.6]{$\frac{{\uu}_3}{{\uu}_1}\hb^{-1}$}};
\node at (2.05,-.4) {\scalebox{.6}[.6]{$\frac{{\uu}_2}{{\uu}_1} $}};
\node at (2.4,-0.25) {\scalebox{.6}[.6]{$\frac{{\uu}_3}{{\uu}_2}$}};

\node at (.32,.2) {\scalebox{.6}[.6]{$\frac{{\uu}_2}{{\uu}_3}$}};
\node at (-.1,.4) {\scalebox{.6}[.6]{$\frac{{\uu}_1}{{\uu}_3}\hb^2$}};
\node at (-.4,.22) {\scalebox{.6}[.6]{$\frac{{\uu}_2}{{\uu}_3}\hb$}};
\node at (-0.4,-.25) {\scalebox{.6}[.6]{$\frac{{\uu}_3}{{\uu}_2}\hb$}};
\node at (.05,-.4) {\scalebox{.6}[.6]{$\frac{{\uu}_3}{{\uu}_2} $}};
\node at (.41,-0.25) {\scalebox{.6}[.6]{$\frac{{\uu}_3}{{\uu}_1}\hb^{-1}$}};

\node at (.35,2.2) {\scalebox{.6}[.6]{$\frac{{\uu}_1}{{\uu}_3}\hb^3$}};
\node at (-.1,2.4) {\scalebox{.6}[.6]{$\frac{{\uu}_1}{{\uu}_2}\hb^3$}};
\node at (-.4,2.24) {\scalebox{.6}[.6]{$\frac{{\uu}_2}{{\uu}_3}\hb$}};
\node at (-0.35,1.75) {\scalebox{.6}[.6]{$\frac{{\uu}_3}{{\uu}_2}$}};
\node at (-.14,1.5) {\scalebox{.6}[.6]{$\frac{{\uu}_3}{{\uu}_1} \hb^{-2}$}};
\node at (.42,1.75) {\scalebox{.6}[.6]{$\frac{{\uu}_2}{{\uu}_1}\hb^{-2}$}};
\end{tikzpicture}
\]
\caption{Illustration of $\T$ fixed points, invariant curves (with their $\T$ weights), poset structure, $\Leaf$s, $\Slope$s, and $N_f^+$ spaces on $\Ch($\ttt{$\setminus$1/2/2$\setminus$2$\setminus$1/}$)$. See Section \ref{sec:ex3}.}  \label{fig:GKM3} 
\end{figure}
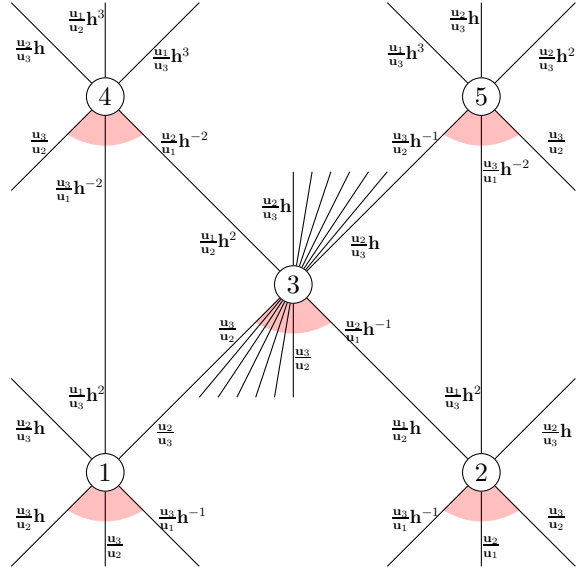

\subsection{Conjectured formula for cohomological stable envelopes for bow varieties}
In the last three sections we presented formulas for stable envelopes that can be verified. % (see details in Section \ref{sec:ex1}). 
The question remains how we came up with these formulas. We used the conjecture we present now.

Let us fix a brane diagram $\DD$, and a fixed point $f \in \Ch(\DD)^{\T}$. Recall that 
\[
\Loc^{K}_f: K_{\T}(\Ch(\DD))\to K_{\T}(f)=\C[\uu_1^{\pm1}, \uu_2^{\pm1}, \ldots, \uu_n^{\pm1},\hb^{\pm1}]
\]
is the restriction map in K theory that we calculated explicitly in Section \ref{sec:fixrest}. The formula for $T\Ch(\DD)$ from Section \ref{sec:tangent} has the form
\[
T\Ch(\DD)=\sum_{x,y,k} \alpha_{x,y,k} \frac{x}{y} \hb^k
\]
where $x$ and $y$ are one of the Grothendieck roots of one of the tautological bundles, or one of $u_i$, and $\alpha_{x,y,k}\in \Z$. Define a term $ \alpha_{x,y,k} \frac{x}{y} \hb^k$ to be {\em $f$-small}, if $\Loc^K_f( \frac{x}{y})=\frac{u_i}{u_j}$ with $i<j$. Let
\[
\tilde{W}_f=e\left( \sum_{f\text{-small}} \alpha_{x,y,k} \frac{x}{y} \hb^k \right).
\]
Now $\tilde{W}_f$ is a rational expression in the Chern roots $x_{ij}$ %($i=1,2,\ldots,s$, $j=1,\ldots,\mult_{X_i}$, where $s$ is the number of D3 branes) 
of the tautological bundles as well as $u_i$ and $\h$. We define
\[
W_f=\frac{1}{N_f} \Sym \left( \tilde{W}_f \right)
\]
where $\Sym$ is the symmetrizing operator $\Sym_1 \Sym_2 \cdots \Sym_s$ and 
\[
\Sym_i (f(x_{i1},\ldots,x_{i,\mult_{X_i}}))=\sum_{\sigma\in S_{\mult_{X_i}}} f( x_{i,\sigma(1)}, x_{i,\sigma(2)}, \ldots, x_{i,\sigma(d_{X_i})}).
\]
The normalizing factor $N_f$ is defined as follows. Consider the tie diagram of the fixed point~$f$. For each D5 brane $U$ the ties whose one end is $U$ cover the D3 branes $(d^U_1, d^U_2,\ldots,d^U_s)$ times. (For example, in Figure \ref{YiyansButterfly}, for $U=U_1$ we get $(1,1,1,1,0,0,0,0,0,0)$, and for $U=U_2$ we get $(1,1,1,3,2,1,1,1,0,0)$.) We define 
\[
N_f=\prod_{U\text{ D5}} \prod_{i=1}^s d^U_i!.
\]

\begin{conjecture} \label{con:Stab}
Cohomological stable envelopes exist for bow varieties. The stable envelope for the fixed point $f$ is represented by $W_f$.
\end{conjecture}

We need to explain what ``represented'' means. The formula $W_f$ is a rational function in the Chern roots of the tautological bundles as well as $u_i,\h$. The cohomological restriction map 
\[
\Loc_{f'}: H^*_{\T}(\Ch(\DD))\to H^*_{\T}(f')=\C[u_1, u_2, \ldots, u_n,\h],
\]
is a ``substitution'' map: we substitute certain $u,\h$ polynomials into the Chern roots. Hence, it can be applied to $W_f$. However, the result might have 0/0 terms. Part of the conjecture is that the {\em limit} of this substitution map exists and is a {\em polynomial}. Thus, the first part of the conjecture is that $W_f$ {\em defines} an element in
$\oplus_{f'} H^*_{\T}(f')$. The second part is that this element is in $\HH_{\T}(\Ch(\DD))$, and the third part is that it satisfies the axioms of Definition \ref{def:SE}. 

\begin{remark}\rm
The phenomenon of naming a rational function whose fixed point restrictions are polynomials is not new in the theory of stable envelopes: the so-called weight functions of \cite{RTV, Rh} are also examples for that. The phenomenon that we need to take limits in the substitution is new.
\end{remark}

Some examples of $W_f$ will be presented in Section \ref{sec:StabEx}

\begin{remark}
The analogous conjecture for K theoretic stable envelopes also holds in many special cases we checked by computer, see more details in \cite{Sh}.
\end{remark}

\begin{remark}
Although the definition of $W_f$ may sound technical, in plain language, it is just the natural formula we obtain if we want an expression symmetric in the Grothendieck roots that satisfies the normalization axiom of Definition \ref{def:SE}. Remarkably, all our computations support the conjecture that the rest of the axioms also hold. The expressions in Sections \ref{sec:ex1}--\ref{sec:ex3} are results of calculating the relevant $W_f$ formulas. 
\end{remark}

\begin{remark} 
Hanany-Witten equivalent brane diagrams have tautological bundles of different ranks. Hence, the symmetrization part of the definition of $W_f$ may be computationally much easier for certain representatives in the HW equivalence class than for others. Our choice e.g. in Section \ref{sec:ex2} was made for this reason.
\end{remark}

\begin{remark}
Let us call a brane diagram {\em separated} if all the $m$ NS5 branes are to the left of all the $n$ D5 branes. Equivalently, if the $m\times n$ table-with-margin code for the diagram has the separating line at the left and bottom edge of the table. 
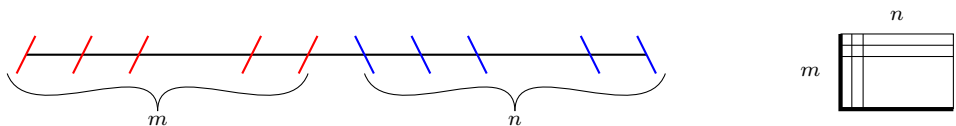
\[
\begin{tikzpicture}[scale=.25]
\draw[thick] (0,1)--(33,1) ;
\draw[thick,red] (-.5,0)--(.5,2);
\draw[thick,red] (2.5,0)--(3.5,2);
\draw[thick,red] (5.5,0)--(6.5,2);
%\draw[thick,red] (8.5,0)--(9.5,2);
\draw[thick,red] (11.5,0)--(12.5,2);
\draw[thick,red] (14.5,0)--(15.5,2);
\draw[thick,blue] (18.5,0)--(17.5,2);
\draw[thick,blue] (21.5,0)--(20.5,2);
\draw[thick,blue] (24.5,0)--(23.5,2);
%\draw[thick,blue] (27.5,0)--(26.5,2);
\draw[thick,blue] (30.5,0)--(29.5,2);
\draw[thick,blue] (33.5,0)--(32.5,2);
%\draw[thick,blue] (36.5,0)--(35.5,2);
\draw  [thin](-1,0) to [out=290,in=90] (7,-2);
\draw [ thin](7,-2) to [out=90,in=250] (15,0);
\node at (7,-2.5) {\tiny $m$};
\draw  [thin](18,0) to [out=290,in=90] (26,-2);
\draw [ thin](26,-2) to [out=90,in=250] (34,0);
\node at (26,-2.5) {\tiny $n$};
\end{tikzpicture}
\qquad\qquad
\begin{tikzpicture}[scale=.5, baseline=-10pt]
\draw[ultra thick] (0,2)--(0,0) --(3,0);
\draw[thin] (0,2)--(3,2) --(3,0);
\draw[thin] (0,1.7)--(3,1.7);
\draw[thin] (0,1.4)--(3,1.4);
\draw[thin] (.3,2) --(.3,0);
\draw[thin] (.6,2) --(.6,0);
\node at (-.8,1) {\tiny $m$};
\node at (1.5,2.5) {\tiny $n$};
\end{tikzpicture}
\]
Every HW equivalence class has exactly one separated brane diagram. These diagrams have many special properties. One is that the $\Loc^{K}_f$ maps restricted to the $\xi_{m+1}, \xi_{m+2}, \ldots, \xi_{m+n}$ tautological bundles do not depend on $f$. This fact has computational consequences. For example, the definition of $W_f$ can be simplified so that we deal only with  the first $m-1$ bundles (see more details in \cite{Sh}). For the separated representative of the HW equivalence class of $T^*\!\cF$ varieties, our conjectured $W_f$ coincides with the weight functions of \cite{RTV}.
\end{remark}

\subsection{An example for Conjecture \ref{con:Stab}} \label{sec:StabEx}
Consider the brane diagram $\DD=\ttt{\bs 1\fs 2\fs 2\bs 2\bs 1\fs}$ of Section \ref{sec:ex3} (see Figure~\ref{fig:GKM3}). In this section we show how Conjecture \ref{con:Stab} produces some entries in the second line of the stable envelope table of Section \ref{sec:ex3}.

According to Section \ref{sec:tangent} the first few terms of the tangent bundle $T\Ch(\DD)$ expressed in Grothen\-dieck roots are
\begin{equation}\label{eq:extan}
\frac{\uu_1}{\xi_1}\hb+\frac{\xi_1}{\xi_2^{(1)}}\hb+\frac{\xi_1}{\xi_2^{(2)}}\hb+\frac{\xi_2^{(1)}}{\xi_1}+\frac{\xi_2^{(2)}}{\xi_1}-1-(1+\hb)\left(\frac{\xi_2^{(1)}}{\xi_2^{(2)}}+\frac{\xi_2^{(2)}}{\xi_2^{(1)}}+2\right)+\ldots,
\end{equation}
where $\xi_1$ is the first tautological bundle, and $\xi_2^{(1)}, \xi_2^{(2)}$ are the Grothendieck roots of the second tautological bundle, etc. According to Section \ref{sec:fixrest} the localization map to the second fixed point maps $\xi_1\mapsto \uu_1\hb, \xi_2^{(1)}\mapsto \uu_1\hb, \xi_2^{(2)}\mapsto \uu_3\hb^{-1}$. Under this substitution only a few terms from \eqref{eq:extan} will be of the form $\uu_{small}/\uu_{large}$, from the displayed ones only $\xi_1/\xi_2^{(2)}\hb$ and $-(1+\hb)\xi_2^{(1)}/\xi_2^{(2)}$; for the whole sum we obtain 
\begin{multline*}
\frac{\xi_1}{\xi_2^{(2)}}\hb-\frac{\xi_2^{(1)}}{\xi_2^{(2)}}-\frac{\xi_2^{(1)}}{\xi_2^{(2)}}\hb-\frac{\xi_3^{(1)}}{\xi_3^{(2)}}+\frac{\xi_3^{(1)}}{u_2}+\frac{\xi_4^{(1)}}{u_3}+\frac{\xi_2^{(1)}}{\xi_3^{(2)}}\hb-\frac{\xi_4^{(1)}}{\xi_4^{(2)}}+\frac{\xi_3^{(1)}}{\xi_4^{(2)}}-\frac{\xi_3^{(1)}}{\xi_4^{(2)}}\hb+\frac{u_2}{\xi_4^{(2)}}\hb+\frac{\xi_3^{(1)}}{\xi_2^{(2)}}
+\frac{\xi_4^{(1)}}{\xi_4^{(2)}}\hb.
\end{multline*}
Hence we have 
\[
\textstyle
\tilde{W}_2=
{\frac{ (x_{11}-x_{22}+\h) (x_{31}-u_{2}) (x_{41}-u_{3}) (x_{21}-x_{32}+ \h) (x_{31}-x_{42})(u_{2}-x_{42}+ \h) (x_{31}-x_{22}) (x_{41}-x_{42}+ \h)}
{(x_{21}-x_{22})(x_{21}-x_{22}+ \h)(x_{31}-x_{32}) (x_{41}-x_{42})  (x_{31}-x_{42}+ \h) }},
\]
and $W_2$ is its $S_2\times S_2 \times S_2$ symmetrization with respect to $(x_{2,1},x_{2,2})$, $(x_{3,1},x_{3,2})$, $(x_{4,1},x_{4,2})$. Hence, e.g. the (2,3) entry of the stable envelope table of Section \ref{sec:ex3} is the 

\centerline{
\begin{tabular}{lllll}
$x_{11}\mapsto u_1+\h$  &  $x_{21}\mapsto u_1+\h$ & $x_{31}\mapsto u_2$ &   $x_{41}\mapsto u_2$ & $x_{51}\mapsto u_2$ \\
                                 &  $x_{22}\mapsto u_3-\h$  & $x_{32}\mapsto u_3$&   $x_{42}\mapsto u_3$ &  
\end{tabular}
}

\noindent substitution into this 8-term rational function (the substitutions are determined in Section \ref{sec:fixrest}). We get termwise 0, hence the (2,3) entry is 0.

The (2,2) entry in the table is the 

\centerline{
\begin{tabular}{lllll}
$x_{11}\mapsto u_1+\h$  &  $x_{21}\mapsto u_1+\h$ & $x_{31}\mapsto u_1+\h$ &   $x_{41}\mapsto u_1+\h$ & $x_{51}\mapsto u_1+\h$ \\
                                 &  $x_{22}\mapsto u_3-\h$  & $x_{32}\mapsto u_3$       &   $x_{42}\mapsto u_3$      &  
\end{tabular}
}

\noindent substitution into the 8-term rational expression. The first term maps to $(u_1-u_3+2\h)(u_1-u_2+\h)(u_2-u_3+\h)$ (a polynomial!) and all the other terms map to 0.  

The (2,5) entry in the table is the 

\centerline{
\begin{tabular}{lllll}
$x_{11}\mapsto u_1+\h$  &  $x_{21}\mapsto u_1+\h$ & $x_{31}\mapsto u_3$           &   $x_{41}\mapsto u_3$ & $x_{51}\mapsto u_3-\h$ \\
                                 &  $x_{22}\mapsto u_3-\h$  & $x_{32}\mapsto u_3-\h$       &   $x_{42}\mapsto u_3-\h$      &  
\end{tabular}
}

\noindent substitution into the 8-term rational expression. Six of those terms map to 0. However, the substitution does not make sense for the remaining two terms, because of the presence of the $(x_{32}-x_{41}+\h)$ factor in the denominator of these two terms. Yet, the sum of the terms {\em has} a limit, and it is 0. Thus we obtain that the (2,5) entry of the table is 0. The other entries follow similarly.

The example of this section illustrates that Conjecture \ref{con:Stab} typically does not provide the simplest representatives for the stable envelope classes. The $F_j$ functions of \eqref{eq:polys} are for example much simpler representatives (they are polynomials to start with, not rational functions) than the formula of Conjecture \ref{con:Stab}.

\subsection{Elliptic stable envelopes, 3d mirror symmetry for characteristic classes} \label{sec:elliptic}
The main goal of the subject---not achieved in this paper yet---is to prove that 3d mirror symmetry is displayed in the stable envelopes, namely, elliptic stable envelopes. In this section we give an example of the sought statement.

Consider the 3d mirror dual brane diagrams 
\[
\DD=
\begin{tikzpicture}[baseline=2pt,scale=.2]
\draw[thick] (0,1)--(12,1) ;
\draw[thick,red] (-.5,0)--(.5,2);
\draw[thick,blue] (3.5,0)--(2.5,2);
\draw[thick,blue] (6.5,0)--(5.5,2);
\draw[thick,blue] (9.5,0)--(8.5,2);
\draw[thick,red] (11.5,0)--(12.5,2);
\node at (1.5,1.6) {\tiny $1$};\node at (4.5,1.6) {\tiny $1$};\node at (7.5,1.6) {\tiny $1$};\node at (10.5,1.6) {\tiny $1$};
\node at (0,-1) {\tiny $V_1$};\node at (3,-1) {\tiny $U_1$};\node at (6,-1) {\tiny $U_2$};\node at (9,-1) {\tiny $U_3$};\node at (12,-1) {\tiny $V_2$};
\end{tikzpicture}
\qquad\qquad
\DD'=
\begin{tikzpicture}[baseline=2pt,scale=.2]
\draw[thick] (0,1)--(12,1) ;
\draw[thick,blue] (.5,0)--(-.5,2);
\draw[thick,red] (2.5,0)--(3.5,2);
\draw[thick,red] (5.5,0)--(6.5,2);
\draw[thick,red] (8.5,0)--(9.5,2);
\draw[thick,blue] (12.5,0)--(11.5,2);
\node at (1.5,1.6) {\tiny $1$};\node at (4.5,1.6) {\tiny $1$};\node at (7.5,1.6) {\tiny $1$};\node at (10.5,1.6) {\tiny $1$};
\node at (0,-1) {\tiny $U'_1$};\node at (3,-1) {\tiny $V'_1$};\node at (6,-1) {\tiny $V'_2$};\node at (9,-1) {\tiny $V'_3$};\node at (12,-1) {\tiny $U'_2$};
\end{tikzpicture}
\]
and let us fix their corresponding fixed points
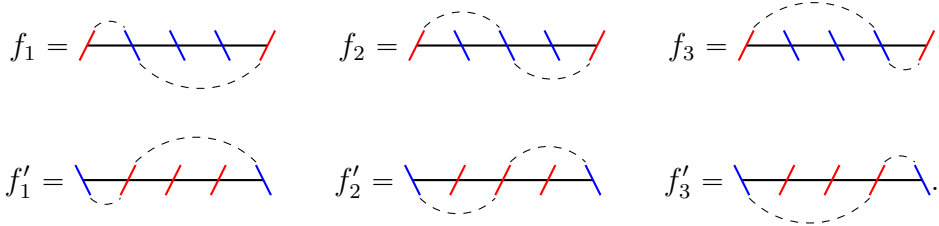
\[
f_1=\begin{tikzpicture}[baseline=2pt,scale=.2]
\draw[thick] (0,1)--(12,1) ;
\draw[thick,red] (-.5,0)--(.5,2);
\draw[thick,blue] (3.5,0)--(2.5,2);
\draw[thick,blue] (6.5,0)--(5.5,2);
\draw[thick,blue] (9.5,0)--(8.5,2);
\draw[thick,red] (11.5,0)--(12.5,2);
\draw [dashed, black](3.5,-.2) to [out=-45,in=225] (11.5,-.2);
\draw [dashed, black](.5,2.2) to [out=45,in=-225] (2.5,2.2);
\end{tikzpicture}
\qquad
f_2=\begin{tikzpicture}[baseline=2pt,scale=.2]
\draw[thick] (0,1)--(12,1) ;
\draw[thick,red] (-.5,0)--(.5,2);
\draw[thick,blue] (3.5,0)--(2.5,2);
\draw[thick,blue] (6.5,0)--(5.5,2);
\draw[thick,blue] (9.5,0)--(8.5,2);
\draw[thick,red] (11.5,0)--(12.5,2);
\draw [dashed, black](6.5,-.2) to [out=-45,in=225] (11.5,-.2);
\draw [dashed, black](.5,2.2) to [out=45,in=-225] (5.5,2.2);
\end{tikzpicture}
\qquad
f_3=\begin{tikzpicture}[baseline=2pt,scale=.2]
\draw[thick] (0,1)--(12,1) ;
\draw[thick,red] (-.5,0)--(.5,2);
\draw[thick,blue] (3.5,0)--(2.5,2);
\draw[thick,blue] (6.5,0)--(5.5,2);
\draw[thick,blue] (9.5,0)--(8.5,2);
\draw[thick,red] (11.5,0)--(12.5,2);
\draw [dashed, black](9.5,-.2) to [out=-45,in=225] (11.5,-.2);
\draw [dashed, black](.5,2.2) to [out=45,in=-225] (8.5,2.2);
\end{tikzpicture}
\]
\[
f'_1=
\begin{tikzpicture}[baseline=2pt,scale=.2]
\draw[thick] (0,1)--(12,1) ;
\draw[thick,blue] (.5,0)--(-.5,2);
\draw[thick,red] (2.5,0)--(3.5,2);
\draw[thick,red] (5.5,0)--(6.5,2);
\draw[thick,red] (8.5,0)--(9.5,2);
\draw[thick,blue] (12.5,0)--(11.5,2);
\draw [dashed, black](.5,-.2) to [out=-45,in=225] (2.5,-.2);
\draw [dashed, black](3.5,2.2) to [out=45,in=-225] (11.5,2.2);
\end{tikzpicture}
\qquad
f'_2=
\begin{tikzpicture}[baseline=2pt,scale=.2]
\draw[thick] (0,1)--(12,1) ;
\draw[thick,blue] (.5,0)--(-.5,2);
\draw[thick,red] (2.5,0)--(3.5,2);
\draw[thick,red] (5.5,0)--(6.5,2);
\draw[thick,red] (8.5,0)--(9.5,2);
\draw[thick,blue] (12.5,0)--(11.5,2);
\draw [dashed, black](.5,-.2) to [out=-45,in=225] (5.5,-.2);
\draw [dashed, black](6.5,2.2) to [out=45,in=-225] (11.5,2.2);
\end{tikzpicture}
\qquad
f'_3=
\begin{tikzpicture}[baseline=2pt,scale=.2]
\draw[thick] (0,1)--(12,1) ;
\draw[thick,blue] (.5,0)--(-.5,2);
\draw[thick,red] (2.5,0)--(3.5,2);
\draw[thick,red] (5.5,0)--(6.5,2);
\draw[thick,red] (8.5,0)--(9.5,2);
\draw[thick,blue] (12.5,0)--(11.5,2);
\draw [dashed, black](.5,-.2) to [out=-45,in=225] (8.5,-.2);
\draw [dashed, black](9.5,2.2) to [out=45,in=-225] (11.5,2.2);
\end{tikzpicture}.
\]

The graphs of their fixed points are 
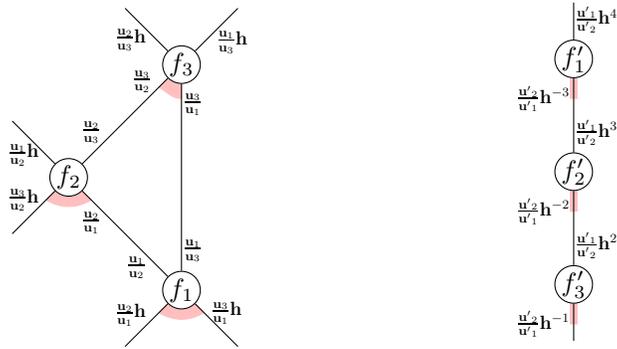
\[
\begin{tikzpicture}[scale=1.5]
\draw[fill=pink,pink] (1,0) to (1.2,-0.2) to  [out=-150, in=-30]  (.8,-0.2) to (1,0);
\draw[fill=pink,pink] (0,1) to (.2,.8) to  [out=-150, in=-30]  (-.2,.8) to (0,1);
\draw[fill=pink,pink] (1,2) to (1,1.7) to  [out=-180, in=-30]  (.8,1.8) to (1,2);
\draw[] (1,0) -- (0,1);
\draw[] (0,1) -- (1,2);
\draw[] (1,0)--(1,2);
\draw[] (1,0)--(1.5,-.5);
\draw[] (1,0)--( .5,-.5);
\draw[] (0,1)--(-.5,0.5);
\draw[] (0,1)--(-.5,1.5);
\draw[] (1,2)--(1.5,2.5);
\draw[] (1,2)--(  .5,2.5);
\node[fill,white,draw,circle,minimum size=.5cm,inner sep=0pt] at (1,0) {\footnotesize $f_1$};
\node[fill,white,draw,circle,minimum size=.5cm,inner sep=0pt] at (0,1) {\footnotesize$f_2$};
\node[fill,white,draw,circle,minimum size=.5cm,inner sep=0pt] at (1,2) {\footnotesize$f_3$};
\node[draw, circle,minimum size=.5cm,inner sep=0pt] at (1,0) {\footnotesize $f_1$};
\node[draw, circle,minimum size=.5cm,inner sep=0pt] at (0,1) {\footnotesize$f_2$};
\node[draw, circle,minimum size=.5cm,inner sep=0pt] at (1,2) {\footnotesize$f_3$};
\node at (1.1,.35) {\scalebox{.6}[.6]{$\frac{{\uu}_1}{{\uu}_3}$}};
\node at (1.1,1.65) {\scalebox{.6}[.6]{$\frac{{\uu}_3}{{\uu}_1}$}};
\node at (0.6,.2) {\scalebox{.6}[.6]{$\frac{{\uu}_1}{{\uu}_2}$}};
\node at (0.2,.6) {\scalebox{.6}[.6]{$\frac{{\uu}_2}{{\uu}_1}$}};
\node at (0.2,1.4) {\scalebox{.6}[.6]{$\frac{{\uu}_2}{{\uu}_3}$}};
\node at (0.65,1.85) {\scalebox{.6}[.6]{$\frac{{\uu}_3}{{\uu}_2}$}};
\node at (1.4,-.2) {\scalebox{.6}[.6]{$\frac{{\uu}_3}{{\uu}_1}\hb$}};
\node at (.55,-.22) {\scalebox{.6}[.6]{$\frac{{\uu}_2}{{\uu}_1}\hb$}};
\node at (-.4,.8) {\scalebox{.6}[.6]{$\frac{{\uu}_3}{{\uu}_2}\hb$}};
\node at (-.4,1.2) {\scalebox{.6}[.6]{$\frac{{\uu}_1}{{\uu}_2}\hb$}};
\node at (1.45,2.2) {\scalebox{.6}[.6]{$\frac{{\uu}_1}{{\uu}_3}\hb$}};
\node at (.55,2.22) {\scalebox{.6}[.6]{$\frac{{\uu}_2}{{\uu}_3}\hb$}};
\end{tikzpicture}
\qquad\qquad\qquad\qquad
\begin{tikzpicture}[scale=1.5]
\draw[fill=pink,pink] (-.03,0) to (-0.03,-0.35) to (.03,-0.35) to (.03,0);
\draw[fill=pink,pink] (-.03,1) to (-0.03,.65) to (.03,.65) to (.03,1);
\draw[fill=pink,pink] (-.03,2) to (-0.03,1.65) to (.03,1.65) to (.03,2);
\draw[] (0,-.5) -- (0,2.5);
\node[fill,white,draw,circle,minimum size=.5cm,inner sep=0pt] at (0,0) {\footnotesize $f'_3$};
\node[fill,white,draw,circle,minimum size=.5cm,inner sep=0pt] at (0,1) {\footnotesize$f'_2$};
\node[fill,white,draw,circle,minimum size=.5cm,inner sep=0pt] at (0,2) {\footnotesize$f'_1$};
\node[draw, circle,minimum size=.5cm,inner sep=0pt] at (0,0) {\footnotesize $f'_3$};
\node[draw, circle,minimum size=.5cm,inner sep=0pt] at (0,1) {\footnotesize$f'_2$};
\node[draw, circle,minimum size=.5cm,inner sep=0pt] at (0,2) {\footnotesize$f'_1$};
\node at (0.2,.35) {\scalebox{.6}[.6]{$\frac{{\uu'}_1}{{\uu'}_2}\hb^2$}};
\node at (0.2,1.35) {\scalebox{.6}[.6]{$\frac{{\uu'}_1}{{\uu'}_2}\hb^3$}};
\node at (0.2,2.35) {\scalebox{.6}[.6]{$\frac{{\uu'}_1}{{\uu'}_2}\hb^4$}};
\node at (-.27,-.35) {\scalebox{.6}[.6]{$\frac{{\uu'}_2}{{\uu'}_1}\hb^{-1}$}};
\node at (-.27,.65) {\scalebox{.6}[.6]{$\frac{{\uu'}_2}{{\uu'}_1}\hb^{-2}$}};
\node at (-.27,1.65) {\scalebox{.6}[.6]{$\frac{{\uu'}_2}{{\uu'}_1}\hb^{-3}$}};
\end{tikzpicture}.
\]
Based on these graphs the cohomological stable envelopes of the two varieties can be calculated:
\begin{equation}\label{eq:Ah}
\begin{tabular}{|c||c|c|c|} 
\hline
 & $f_1$ & $f_2$ & $f_3$ 
 \\
 \hline\hline
$f_1$ & $(u_1-u_2)(u_1-u_3)$ &  0 &  0
\\
\hline
$f_2$ &  $(u_1-u_3)\h$ &  $(u_1-u_2+\h)(u_2-u_3)$ &  0 
\\
\hline
$f_3$ &  $(u_2-u_1+\h)\h$ &  $(u_1-u_2+\h)\h$ &  $(u_2-u_3+\h)(u_1-u_3+\h)$ \\
\hline
\end{tabular}
\end{equation}

\begin{equation}\label{eq:Bh}
\begin{tabular}{|c||c|c|c|} 
\hline
 & $f'_1$ & $f'_2$ & $f'_3$ 
 \\
 \hline\hline
$f'_1$ &  $u'_1-u'_2+4\h$ &  $\h$ &  $\h$
\\
\hline
$f'_2$ &  0 &  $u'_1-u'_2+3\h$ &  $\h$ 
\\
\hline
$f'_3$ & 0 & 0 & $u'_1-u'_2+2\h$ \\
\hline
\end{tabular}.
\end{equation}

There is however an elliptic cohomology extension of the notion of stable envelopes, due to Aganagic-Okounkov \cite{AO} (see also \cite[\S 5.5]{FRV}, \cite[\S 7.7]{RTV_EllK}, \cite[\S 1.2]{RSVZ1}, \cite{Sm}). We do not attempt to recall the axiomatic definition of elliptic stable envelopes, or their relevance in representation theory and physics (the Introduction in \cite{AO} is a great start). What we want to emphasize is that the elliptic stable envelopes {\em necessarily} depend on a new set of variables---named K\"ahler or dynamical variables. In the language of this paper, for bow varieties, these new variables are naturally attached to the NS5 branes. We will denote the K\"ahler variable attached to the NS5 brane $V_i$ by $\vv_i$ (just like we denoted the equivariant parameter attached to the D5 brane $U_i$ by $\uu_i$ in K theory or $u_i$ in cohomology). The elliptic stable envelopes on $\Ch(\DD)$ and $\Ch(\DD')$ are calculated to be
\begin{equation}\label{eq:Ae}
\begin{tabular}{|c||c|c|c|} 
\hline
 & $f_1$ & $f_2$ & $f_3$ 
 \\
 \hline\hline
$f_1$ & {} $\theta(\frac{\uu_1}{\uu_2})\theta(\frac{\uu_1}{\uu_3})\theta(\frac{\vv_2}{\vv_1}\hb^4)$ & {} 0 & {} 0
\\
\hline
$f_2$ & {} $\theta(\hb)\theta(\frac{\uu_1}{\uu_3})\theta(\frac{\uu_2\vv_2}{\uu_1\vv_1}\hb^3)$ & {} $\theta(\frac{\uu_1}{\uu_2}\hb) \theta(\frac{\uu_2}{\uu_3})\theta(\frac{\vv_2}{\vv_1}\hb^3)$ &{}  0 
\\
\hline
$f_3$ & {} $\theta(\hb)\theta(\frac{\uu_2}{\uu_1}\hb)\theta(\frac{\uu_3\vv_2}{\uu_1\vv_1}\hb^2)$ & {} $\theta(\hb)\theta(\frac{\uu_1}{\uu_2}\hb)\theta(\frac{\uu_3\vv_2}{\uu_2\vv_1}\hb^2)$ & {} $\theta(\frac{\uu_2}{\uu_3}\hb) \theta(\frac{\uu_1}{\uu_3}\hb)\theta(\frac{\vv_2}{\vv_1}\hb^2)$ \\
\hline
\end{tabular}
\end{equation}
\begin{equation}\label{eq:Be}
\begin{tabular}{|c||c|c|c|} 
\hline
 & $f'_1$ & $f'_2$ & $f'_3$ 
 \\
 \hline\hline
$f'_1$ & {} $\theta(\frac{\uu'_1}{\uu'_2}\hb^{4})\theta(\frac{\vv'_2}{\vv'_1})\theta(\frac{\vv'_3}{\vv'_1})$ & {} $\theta(\hb)\theta(\frac{\vv'_3}{\vv'_1})\theta(\frac{\vv'_2\uu'_2}{\vv'_1\uu'_1}\hb^{-3})$ & {} $\theta(\hb)\theta(\frac{\vv'_2}{\vv'_1}\hb^{-1})\theta(\frac{\vv'_3\uu'_2}{\vv'_1\uu'_1}\hb^{-2})$
\\
\hline
$f'_2$ & {} 0 & {} $\theta(\frac{\uu'_1}{\uu'_2}\hb^{3})\theta(\frac{\vv'_2}{\vv'_1}\hb) \theta(\frac{\vv'_3}{\vv'_2})$ &{} $\theta(\hb)\theta(\frac{\vv'_2}{\vv'_1}\hb)\theta(\frac{\vv'_3\uu'_2}{\vv'_2\uu'_1}\hb^{-2})$ 
\\
\hline
$f'_3$ &  {} 0 & {} 0 & {} $\theta(\frac{\uu'_1}{\uu'_2}\hb^2)\theta(\frac{\vv'_3}{\vv'_2}\hb) \theta(\frac{\vv'_3}{\vv'_1}\hb)$ \\
\hline
\end{tabular}
\end{equation}
where
\[
\theta(x)=(x^{1/2}-x^{-1/2})\prod_{s=1}^\infty (1-q^sx)(1-q^s/x).
\]
For an intuitive introduction to equivariant elliptic cohomology, as well as the necessity of K\"ahler parameters in the elliptic characteristic classes see \cite[\S 5-6]{Rh}.
\begin{remark} \rm
The cohomology stable envelopes can be calculated from the elliptic ones, by a limiting procedure. The natural prediction, namely $\theta(\frac{\vv_i}{\vv_j}\cdot p(\uu,\h))\to 1$($i>j$) and $\theta(\uu)=\theta(e^u)\to \sin(u)\to u$, coming from comparing tables \eqref{eq:Ae}, \eqref{eq:Be} with \eqref{eq:Ah}, \eqref{eq:Bh} is correct.
\end{remark}

The novelty that is only visible for the elliptic stable envelopes is that the stable envelopes of $\Ch(\DD)$ and those of $\Ch(\DD')$ are related: one table is obtained from the other one (up to sign) by transpose, switch of equivariant and K\"ahler variables, as well as $\h \leftrightarrow \h^{-1}$. That is, let $S$ be table \eqref{eq:Ae} and let $S'$ be table \eqref{eq:Be}. Then we have
\begin{equation}\label{eq:3d}
\begin{tikzcd}
S_{kl} \ar[r,equal, "\begin{matrix} \ \ \h \leftrightarrow \h^{-1} \\ \uu_i \leftrightarrow \vv'_i \\ \vv_j \leftrightarrow \uu'_j \end{matrix}"] &   (-1)^{k+l+1}S'_{lk}.
\end{tikzcd}
\end{equation}
We call this equality the {\em 3d mirror symmetry for characteristic classes}. The expectation is that \eqref{eq:3d} is a general phenomenon. The known results are the following: 
\begin{itemize}
\item elliptic stable envelopes are defined for Nakajima quiver varieties, as well as so-called abelianization formulas are known for them \cite{AO}, \cite[\S 5]{RSVZ1}, \cite{RTV_EllK}, \cite{Sm}; 
\item using those abelianization formulas, the analogue of \eqref{eq:3d} is proved for $T^*\!\Gr(k,n)$ and its 3d mirror dual if $k\leq n/2$ \cite{RSVZ1}, cf. \eqref{eq:quiverGr}; 
\item the analogue of \eqref{eq:3d} is proved for the full flag variety and its 3d dual (which is the Langland dual full flag variety) in \cite{RSVZ2, RW2}, cf. \eqref{eq:fullflag}.
\item the analogue of \eqref{eq:3d} is proved for hypertoric varieties in \cite{SZ}, see also \cite{KS}.
\end{itemize}

The goal of the project starting with this paper is to prove the analogue of \eqref{eq:3d} in the full generality of bow varieties. 

There are indications that bow varieties are the right land where 3d mirror symmetry for stable envelopes should be studied. Namely, the equivariant and K\"ahler variables come on an equal footing for bow varieties: associated with the two kinds of 5-branes. Also, under the 3d mirror symmetry of brane diagrams the two kinds of 5-branes switch spaces, as \eqref{eq:3d} requires. Moreover, the sophisticated and technical proof in \cite{RSVZ1} does not do justice to the simple elegance of the statement. The reason seems to be that in \cite{RSVZ1} two quiver varieties are considered which are not 3d mirror duals of each other on the nose, but rather one is Hanany-Witten equivalent to the mirror dual of the other one. Relating characteristic class formulas for  Hanany-Witten equivalent varieties is not expected to be simple, because HW equivalence involves difference bundles of tautological bundles, see Theorem \ref{thm:HW}(3). Once elliptic stable envelopes are defined for bow varieties, and (say abelianization) formulas are know for them, the comparison between elliptic stable envelopes of 3d mirror dual varieties should be combinatorial, and the comparison between elliptic stable envelopes on the two sides of a single Hanany-Witten transition should be a theta-function identity. We plan to pursue this project in the future.

\begin{remark}\label{rem:resonance} \rm
The 3d mirror symmetry for characteristic classes gives us a new answer to the question: what is the geometric meaning of the K\"ahler variables in elliptic characteristic classes?  K\"ahler variables are the equivariant variables of the 3d mirror. This innocent remark has a nice application. Actual formulas for characteristic classes often play a role in other areas, eg. if partial differential equations. The so-called weight functions are examples. In those roles the {\em resonance} phenomenon is important \cite{FV}: values of certain weight functions (or their derivatives) agree at certain substitutions of their dynamical parameters. By interpreting the dynamical variables as equivariant variables on the 3d mirror dual, the sought after resonance equations translate to (generalized) GKM conditions that ``neighboring'' fixed point restrictions of the same cohomology class must satisfy. The farther the mirror space is from being GKM the more resonance equations we obtain.
\end{remark}

\begin{figure}
\[
\begin{tikzpicture}[baseline=0pt,scale=.2]
\draw[thick] (0,1)--(9,1) ;
\draw[thick,red] (-.5,0)--(.5,2);
\draw[thick,red] (2.5,0)--(3.5,2);
\draw[thick,blue] (6.5,0)--(5.5,2);
\draw[thick,blue] (9.5,0)--(8.5,2);
\draw [dashed, black](0.5,2.2) to [out=45,in=-225] (5.5,2.2);
\draw [dashed, black](3.5,2.2) to [out=45,in=-225] (8.5,2.2);
\node at (11,1) {$\leftrightarrow$};
\draw [ultra thick] (2,9)--(2,5)-- (6,5);   
\node at (3,6) {\footnotesize $0$}; \node at (3,8) {\footnotesize $1$}; \node at (5,6) {\footnotesize $1$}; \node at (5,8) {\footnotesize $0$}; 
\node at (11,7) {$\leftrightarrow$};
\begin{scope}[xshift=13cm]
\draw[thick] (0,1)--(9,1) ;
\draw[thick,red] (-.5,0)--(.5,2);
\draw[thick,red] (2.5,0)--(3.5,2);
\draw[thick,blue] (6.5,0)--(5.5,2);
\draw[thick,blue] (9.5,0)--(8.5,2);
\draw [dashed, black](0.5,2.2) to [out=45,in=-225] (8.5,2.2);
\draw [dashed, black](3.5,2.2) to [out=45,in=-225] (5.5,2.2);
\draw [ultra thick] (2,9)--(2,5)-- (6,5);   
\node at (3,6) {\footnotesize $1$}; \node at (3,8) {\footnotesize $0$}; \node at (5,6) {\footnotesize $0$}; \node at (5,8) {\footnotesize $1$}; 
\end{scope}
\end{tikzpicture}
\qquad\qquad\qquad
\begin{tikzpicture}[baseline=0pt,scale=.2]
\draw[thick] (0,1)--(9,1) ;
\draw[thick,blue] (.5,0)--(-.5,2);
\draw[thick,blue] (3.5,0)--(2.5,2);
\draw[thick,red] (5.5,0)--(6.5,2);
\draw[thick,red] (8.5,0)--(9.5,2);
\draw [dashed, black](0.5,-.2) to [out=-45,in=225] (8.5,-.2);
\draw [dashed, black](3.5,-.2) to [out=-45,in=225] (5.5,-.2);
\node at (11,1) {\footnotesize $\leftrightarrow$};
\draw [ultra thick] (2,9)--(6,9)-- (6,5);   
\node at (3,6) {\footnotesize $0$}; \node at (3,8) {\footnotesize $1$}; \node at (5,6) {\footnotesize $1$}; \node at (5,8) {\footnotesize $0$}; 
\node at (11,7) {$\leftrightarrow$};
\begin{scope}[xshift=13cm]
\draw[thick] (0,1)--(9,1) ;
\draw[thick,blue] (.5,0)--(-.5,2);
\draw[thick,blue] (3.5,0)--(2.5,2);
\draw[thick,red] (5.5,0)--(6.5,2);
\draw[thick,red] (8.5,0)--(9.5,2);
\draw [dashed, black](0.5,-.2) to [out=-45,in=225] (5.5,-.2);
\draw [dashed, black](3.5,-.2) to [out=-45,in=225] (8.5,-.2);
\draw [ultra thick] (2,9)--(6,9)-- (6,5);   
\node at (3,6) {\footnotesize $1$}; \node at (3,8) {\footnotesize $0$}; \node at (5,6) {\footnotesize $0$}; \node at (5,8) {\footnotesize $1$}; 
\end{scope}
\end{tikzpicture}
\]
\[
\begin{tikzpicture}[baseline=0pt,scale=.2]
\draw[thick] (0,1)--(9,1) ;
\draw[thick,red] (-.5,0)--(.5,2);
\draw[thick,blue] (3.5,0)--(2.5,2);
\draw[thick,red] (5.5,0)--(6.5,2);
\draw[thick,blue] (9.5,0)--(8.5,2);
\draw [dashed, black](0.5,2.2) to [out=45,in=-225] (2.5,2.2);
\draw [dashed, black](6.5,2.2) to [out=45,in=-225] (8.5,2.2);
\draw [dashed, black](3.5,-.2) to [out=-45,in=225] (5.5,-.2);
\node at (11,1) {\footnotesize $\leftrightarrow$};
\draw [ultra thick] (2,9)--(2,7) --(4,7)--(4,5)-- (6,5);   
\node at (3,6) {\footnotesize $0$}; \node at (3,8) {\footnotesize $1$}; \node at (5,6) {\footnotesize $1$}; \node at (5,8) {\footnotesize $0$}; 
\node at (11,7) {$\leftrightarrow$};
\begin{scope}[xshift=13cm]
\draw[thick] (0,1)--(9,1) ;
\draw[thick,red] (-.5,0)--(.5,2);
\draw[thick,blue] (3.5,0)--(2.5,2);
\draw[thick,red] (5.5,0)--(6.5,2);
\draw[thick,blue] (9.5,0)--(8.5,2);
\draw [dashed, black](0.5,2.2) to [out=45,in=-225] (8.5,2.2);
\draw [ultra thick] (2,9)--(2,7) --(4,7)--(4,5)-- (6,5);  
\node at (3,6) {\footnotesize $1$}; \node at (3,8) {\footnotesize $0$}; \node at (5,6) {\footnotesize $0$}; \node at (5,8) {\footnotesize $1$}; 
\end{scope}
\end{tikzpicture}
\qquad\qquad\qquad
\begin{tikzpicture}[baseline=0pt,scale=.2]
\draw[thick] (0,1)--(9,1) ;
\draw[thick,blue] (.5,0)--(-.5,2);
\draw[thick,red] (2.5,0)--(3.5,2);
\draw[thick,blue] (6.5,0)--(5.5,2);
\draw[thick,red] (8.5,0)--(9.5,2);
\draw [dashed, black](0.5,-.2) to [out=-45,in=225] (8.5,-.2);
\node at (11,1) {\footnotesize $\leftrightarrow$};
\draw [ultra thick] (2,9)--(4,9) --(4,7)--(6,7)-- (6,5);   
\node at (3,6) {\footnotesize $0$}; \node at (3,8) {\footnotesize $1$}; \node at (5,6) {\footnotesize $1$}; \node at (5,8) {\footnotesize $0$}; 
\node at (11,7) {$\leftrightarrow$};
\begin{scope}[xshift=13cm]
\draw[thick] (0,1)--(9,1) ;
\draw[thick,blue] (.5,0)--(-.5,2);
\draw[thick,red] (2.5,0)--(3.5,2);
\draw[thick,blue] (6.5,0)--(5.5,2);
\draw[thick,red] (8.5,0)--(9.5,2);
\draw [dashed, black](0.5,-.2) to [out=-45,in=225] (2.5,-.2);
\draw [dashed, black](6.5,-.2) to [out=-45,in=225] (8.5,-.2);
\draw [dashed, black](3.5,2.2) to [out=45,in=-225] (5.5,2.2);
\draw [ultra thick] (2,9)--(4,9) --(4,7)--(6,7)-- (6,5);   
\node at (3,6) {\footnotesize $1$}; \node at (3,8) {\footnotesize $0$}; \node at (5,6) {\footnotesize $0$}; \node at (5,8) {\footnotesize $1$}; 
\end{scope}
\end{tikzpicture}
\]
\[
\begin{tikzpicture}[baseline=0pt,scale=.2]
\draw[thick] (0,1)--(9,1) ;
\draw[thick,red] (-.5,0)--(.5,2);
\draw[thick,blue] (3.5,0)--(2.5,2);
\draw[thick,blue] (6.5,0)--(5.5,2);
\draw[thick,red] (8.5,0)--(9.5,2);
\draw [dashed, black](0.5,2.2) to [out=45,in=-225] (2.5,2.2);
\draw [dashed, black](3.5,-.2) to [out=-45,in=225] (8.5,-.2);
\node at (11,1) {\footnotesize $\leftrightarrow$};
\draw [ultra thick] (2,9)--(2,7) -- (6,7)-- (6,5);   
\node at (3,6) {\footnotesize $0$}; \node at (3,8) {\footnotesize $1$}; \node at (5,6) {\footnotesize $1$}; \node at (5,8) {\footnotesize $0$}; 
\node at (11,7) {$\leftrightarrow$};
\begin{scope}[xshift=13cm]
\draw[thick] (0,1)--(9,1) ;
\draw[thick,red] (-.5,0)--(.5,2);
\draw[thick,blue] (3.5,0)--(2.5,2);
\draw[thick,blue] (6.5,0)--(5.5,2);
\draw[thick,red] (8.5,0)--(9.5,2);
\draw [dashed, black](0.5,2.2) to [out=45,in=-225] (5.5,2.2);
\draw [dashed, black](6.5,-.2) to [out=-45,in=225] (8.5,-.2);
\draw [ultra thick] (2,9)--(2,7) -- (6,7)-- (6,5); 
\node at (3,6) {\footnotesize $1$}; \node at (3,8) {\footnotesize $0$}; \node at (5,6) {\footnotesize $0$}; \node at (5,8) {\footnotesize $1$}; 
\end{scope} 
\end{tikzpicture}
\qquad\qquad\qquad
\begin{tikzpicture}[baseline=0pt,scale=.2]
\draw[thick] (0,1)--(9,1) ;
\draw[thick,blue] (.5,0)--(-.5,2);
\draw[thick,red] (2.5,0)--(3.5,2);
\draw[thick,red] (5.5,0)--(6.5,2);
\draw[thick,blue] (9.5,0)--(8.5,2);
\draw [dashed, black](0.5,-.2) to [out=-45,in=225] (5.5,-.2);
\draw [dashed, black](6.5,2.2) to [out=45,in=-225] (8.5,2.2);
\node at (11,1) {\footnotesize $\leftrightarrow$};
\draw [ultra thick] (2,9)--(4,9) --(4,5)-- (6,5);   
\node at (3,6) {\footnotesize $0$}; \node at (3,8) {\footnotesize $1$}; \node at (5,6) {\footnotesize $1$}; \node at (5,8) {\footnotesize $0$}; 
\node at (11,7) {$\leftrightarrow$};
\begin{scope}[xshift=13cm]
\draw[thick] (0,1)--(9,1) ;
\draw[thick,blue] (.5,0)--(-.5,2);
\draw[thick,red] (2.5,0)--(3.5,2);
\draw[thick,red] (5.5,0)--(6.5,2);
\draw[thick,blue] (9.5,0)--(8.5,2);
\draw [dashed, black](0.5,-.2) to [out=-45,in=225] (2.5,-.2);
\draw [dashed, black](3.5,2.2) to [out=45,in=-225] (8.5,2.2);
\draw [ultra thick] (2,9)--(4,9) --(4,5)-- (6,5);    
\node at (3,6) {\footnotesize $1$}; \node at (3,8) {\footnotesize $0$}; \node at (5,6) {\footnotesize $0$}; \node at (5,8) {\footnotesize $1$}; 
\end{scope}
\end{tikzpicture}
\]
\caption{Surgeries on tie diagrams that reflect the %$\begin{matrix} 1 & 0 \\ 0 & 1 \end{matrix} \leftrightarrow \begin{matrix} 0 & 1 \\ 1 & 0 \end{matrix}$ 
move \eqref{eq:1001} for BCTs. There are six possibilities, depending on the relative position between the $2\times 2$ submatrix and the separating line.}
\label{fig:TieSurg}
\end{figure}
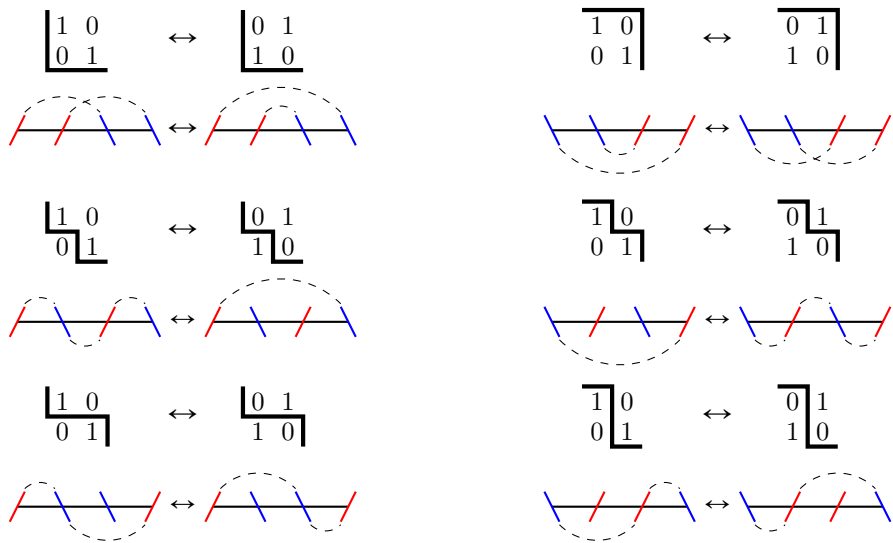

\section{The combinatorics of invariant curves}
Graphs like those in Figures \ref{fig:GKM1}--\ref{fig:GKM3} (generalized GKM graphs) that illustrate fixed point data are important in understanding characteristic classes on bow varieties. Key components of these graphs are the edges, the invariant curves on $\Ch(\DD)$. It would be interesting to find complete combinatorial descriptions of these invariant curves. For a detailed analysis see \cite{Sh}.

Here we only describe the combinatorics of {\em some} of the edges in order to illustrate the similarity with knot diagram (or skein-) combinatorics. Namely, {\em some} of the invariant curves  connect two fixed points whose BCT codes only differ in a $2 \times 2$ submatrix (the two rows or the two columns do {\em not} need to be consecutive) by 
\begin{equation} \label{eq:1001}
\begin{bmatrix} 1 & 0 \\ 0 & 1 \end{bmatrix} \leftrightarrow \begin{bmatrix} 0 & 1 \\ 1 & 0 \end{bmatrix}.
\end{equation}
In this case the tie diagrams encoding the two fixed points are also easily related, but this relation depends on the relative position between this $2\times 2$ submatrix and the separating line. Figure~\ref{fig:TieSurg} illustrates the six different relative positions, and the corresponding surgeries on tie diagrams. 
 
\section{Switching consecutive 5-branes of the same kind} \label{sec:OtherTransitions}
Consider the combinatorial transition of brane diagrams from Remark \ref{rem:OtherTransitions}, that is, for $d_1+d_3=d_2+\tilde{d}_2$ the local changes 
\begin{equation*}
\begin{tikzpicture}[baseline=(current  bounding  box.center), scale=.35]
\draw[thick] (0,1)--(8,1);
\draw[thick,blue] (3,0)--(2,2);
\draw[thick,blue] (6,0)--(5,2);
\node[] at (1.25,1.7) {$d_1$}; \node[] at (4,1.7) {$d_2$}; \node[] at (6.75,1.7) {$d_3$};
\draw[ultra thick, <->] (9,1)--(10.5,1) node[above]{$(TU)$} -- (12,1);
\draw[thick] (13,1)--(21,1);
\draw[thick,blue] (16,0)--(15,2);
\draw[thick,blue] (19,0)--(18,2);
\node[] at (14.25,1.7) {$d_1$}; \node[] at (17,1.75) {$\tilde{d}_2$}; \node[] at (19.75,1.7) {$d_3$};
\begin{scope}[xshift=24cm]
\draw[thick] (0,1)--(8,1);
\draw[thick,red] (2,0)--(3,2);
\draw[thick,red] (5,0)--(6,2);
\node[] at (1.25,1.7) {$d_1$}; \node[] at (4,1.7) {$d_2$}; \node[] at (6.75,1.7) {$d_3$};
\draw[ultra thick, <->] (9,1)--(10.5,1) node[above]{$(TV)$} -- (12,1);
\draw[thick] (13,1)--(21,1);
\draw[thick,red] (15,0)--(16,2);
\draw[thick,red] (18,0)--(19,2);
\node[] at (14.25,1.7) {$d_1$}; \node[] at (17,1.75) {$\tilde{d}_2$}; \node[] at (19.75,1.7) {$d_3$};
\end{scope}
\end{tikzpicture}
\end{equation*}
Under these transitions the charges of the branes do not change, but the branes themselves switch places. Hence, the table-with-margins code for the diagram changes by switching two consecutive components either in $c$ (for $(TU)$ transition) or in $r$ (for $(TV)$ transition). Hence, permitting $(TU)$ transition we may achieve that $c$ is weakly decreasing, and permitting $(TV)$ transition we may achieve that $r$ is weakly increasing. Comparing with Theorem \ref{thm:cobalanced} and Corollary \ref{cor:balanced} we obtain

\begin{proposition}\label{prop:HWTUTV}
Any brane diagram is equivalent to a balanced one using Hanany-Witten and $(TU)$ transitions. Any brane diagram is equivalent to a co-balanced one (ie. whose associated variety is a quiver variety) using Hanany-Witten and $(TV)$ transitions.
\end{proposition}

By permitting $(TV)$ transitions (as well as HW transitions) cotangent bundles of different partial flag varieties become equivalent---which are $C^\infty$ but not algebraically isomorphic. Namely, let $\lambda_1,\lambda_2,\ldots,\lambda_N$ and $\mu_1,\mu_2,\ldots,\mu_N$ be sequences of non-negative integers that are permutations of each other. Then as bow varieties $T^*\!\cF_{\lambda_1,\lambda_1+\lambda_2,\ldots,\lambda_1+\lambda_2+\ldots+\lambda_N}$ and $T^*\!\cF_{\mu_1,\mu_1+\mu_2,\ldots,\mu_1+\mu_2+\ldots+\mu_N}$ are equivalent using HW and $(TV)$ transitions (their tables-with-margins only differ by permuting~$c$).

\smallskip

Figure \ref{fig:TUfix} illustrates a natural bijection between torus fixed points of $\Ch(\DD)$ and torus fixed points of $\Ch(\tilde{\DD})$ for a $(TU)$ transition. It is worth verifying the $d_1+d_3=d_2+\tilde{d}_2$ relation in the figure. An analogous picture (in fact, this one upside down) provides the bijection for a $(TV)$ transition.

\begin{theorem} \label{thm:TU}
Let $\DD$ and $\tilde{\DD}$ be related by $(TU)$ transition. We have
\[
\KK_{\T}(\Ch(\DD)) \cong \KK_{\T}(\Ch(\tDD)),  \qquad\qquad 
\HH_{\T}(\Ch(\DD)) \cong \HH_{\T}(\Ch(\tDD)).
\]
\end{theorem}
 
 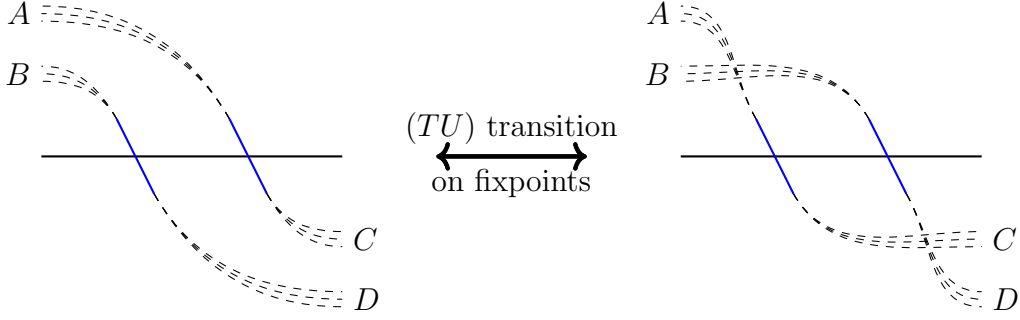
\begin{figure}
\[
\begin{tikzpicture}[scale=.5]
\draw [thick] (0,1) --(8,1);  
\draw [thick, blue] (3,0) -- (2,2);
\draw[thick, blue] (6,0)--(5,2);
\draw [dashed](2,2) to [out=120,in=0] (0,3) ;
\draw [dashed](2,2) to [out=120,in=0] (0,3.2) node [left] {$B$} ;
\draw [dashed](2,2) to [out=120,in=0] (0,3.4) ;
\draw [dashed](6,0) to [out=-60,in=180] (8,-1)  ;
\draw [dashed](6,0) to [out=-60,in=180] (8,-1.2)  node [right] {$C$} ;
\draw [dashed](6,0) to [out=-60,in=180] (8,-1.4);
\draw [dashed](3,0) to [out=-60,in=180] (8,-2.6);
\draw [dashed](3,0) to [out=-60,in=180] (8,-2.8) node [right]{$D$};
\draw [dashed](3,0) to [out=-60,in=180] (8,-3) ;
\draw [dashed](5,2) to [out=120,in=0] (0,4.6) ;
\draw [dashed](5,2) to [out=120,in=0] (0,4.8) node [left]{$A$};
\draw [dashed](5,2) to [out=120,in=0] (0,5);
\draw[ultra thick, <->] (10.5,1)--(12.5,1) node[above]{$(TU)$ transition} -- (14.5,1);
\draw[ultra thick, <->] (10.5,1)--(12.5,1) node[below]{on fixpoints} -- (14.5,1);
\begin{scope}[xshift=17cm]
\draw [thick] (0,1) --(8,1);  
\draw [thick, blue] (3,0) -- (2,2);
\draw[thick, blue] (6,0)--(5,2);
\draw [dashed](5,2) to [out=120,in=0] (0,3) ;
\draw [dashed](5,2) to [out=120,in=0] (0,3.2) node [left] {$B$} ;
\draw [dashed](5,2) to [out=120,in=0] (0,3.4) ;
\draw [dashed](3,0) to [out=-60,in=180] (8,-1)  ;
\draw [dashed](3,0) to [out=-60,in=180] (8,-1.2)  node [right] {$C$} ;
\draw [dashed](3,0) to [out=-60,in=180] (8,-1.4);
\draw [dashed](6,0) to [out=-60,in=180] (8,-2.6);
\draw [dashed](6,0) to [out=-60,in=180] (8,-2.8) node [right]{$D$};
\draw [dashed](6,0) to [out=-60,in=180] (8,-3) ;
\draw [dashed](2,2) to [out=120,in=0] (0,4.6) ;
\draw [dashed](2,2) to [out=120,in=0] (0,4.8) node [left]{$A$};
\draw [dashed](2,2) to [out=120,in=0] (0,5);
\end{scope}
\end{tikzpicture}
\]
\caption{Bijection between fixed point codes for brane diagrams related by a $(TU)$ transition.}\label{fig:TUfix}
\end{figure}
 
\begin{proof} The proof depends on the combinatorics of fixed point restrictions, namely the structure of the butterfly diagrams of Section \ref{sec:fixrest}. Let $U_k$ and $U_{k+1}$ be the two consecutive D5 branes switched at the $(TU)$ transition, and let $X_1, X_2, X_3$ be the D3 branes adjacent to these 5-branes, in this order. We define a map 
\[
s:
\bigoplus_{f\in \Ch(\DD)^{\T}} \C[\uu_1^{\pm 1},\ldots,\uu^{\pm 1}_n,\hb^{\pm 1}]
\to
\bigoplus_{\tilde{f}\in \Ch(\tilde{\DD})^{\T}} \C[\tilde{\uu}_1^{\pm 1},\ldots,\tilde{\uu}^{\pm 1}_n,\hb^{\pm 1}]
\]
as follows. The $f$-component maps to the $\tilde{f}$-component where $f$ and $\tilde{f}$ are related as in Figure~\ref{fig:TUfix}. The map between these components is $\uu_i\mapsto \tilde{\uu}_i$  for $i\not=k, k+1$, and $\uu_k\mapsto \tilde{\uu}_{k+1}$, $\uu_{k+1}\mapsto \tilde{\uu}_k$. This map restricts to a map $s': \KK_{\T}(\Ch(\DD)) \to \KK_{\T}(\Ch(\tDD))$ because we claim that 
\begin{equation} \label{eq:thatone}
s\left( \Loc_f( \xi_{X_2} ) \right) = \Loc_{\tilde{f}}( \xi_{X_1}  \oplus  \xi_{X_3} \ominus \xi_{X_2} ).
\end{equation}
holds for corresponding fixed points $f$ and $\tilde{f}$. Indeed, according to Section \ref{sec:fixrest}, $\Loc^K_f$ maps the relevant Grothendieck roots of $\xi_1, \xi_2,\xi_3$ to 
\[
\begin{tabular}{lll}
$\{\uu_k, \uu_k\hb^{-1},\ldots, \uu_k\hb^{1-b},$ && $\uu_{k+1}, \uu_{k+1}\hb^{-1},\ldots,\uu_{k+1}\hb^{1-a}\}$ \\
$\{\uu_k\hb^{d-b}, \ldots, \uu_k\hb^{1-b},$ && ${\uu}_{k+1}, {\uu}_{k+1}\hb^{-1},\ldots,{\uu}_{k+1}\hb^{1-a}\}$\\
$\{\uu_k\hb^{d-b}, \ldots, \uu_k\hb^{1-b},$ && ${\uu}_{k+1}\hb^{c-a}, \ldots,{\uu}_{k+1}\hb^{1-a}\}$
\end{tabular}
\]
respectively, where $a,b,c,d$ are the number of ties in $A,B,C,D$ in the figure. Similarly $\Loc^K_{\tilde{f}}$ maps the relevant Grothendieck roots of $\xi_1, \tilde{\xi}_2,\xi_3$ to 
\[
\begin{tabular}{lll}
$\{\tilde{\uu}_k, \tilde{\uu}_k\hb^{-1},\ldots, \tilde{\uu}_k\hb^{1-a},$ && $\tilde{\uu}_{k+1}, \tilde{\uu}_{k+1}\hb^{-1},\ldots,\tilde{\uu}_{k+1}\hb^{1-b}\}$ \\
$\{\tilde{\uu}_k\hb^{c-a}, \ldots, \tilde{\uu}_k\hb^{1-a},$ && $\tilde{\uu}_{k+1}, \tilde{\uu}_{k+1}\hb^{-1},\ldots,\tilde{\uu}_{k+1}\hb^{1-b}\}$\\
$\{\tilde{\uu}_k\hb^{c-a}, \ldots, \tilde{\uu}_k\hb^{1-a},$ && $\tilde{\uu}_{k+1}\hb^{d-b}, \ldots,\tilde{\uu}_{k+1}\hb^{1-b}\}$,
\end{tabular}
\]
and \eqref{eq:thatone} indeed holds. The map $s'$ is clearly invertible hence the isomorphism in K theory is proved. The isomorphism in cohomology is proved similarly.
\end{proof} 

For $(TV)$ transition the counterpart of Theorem \ref{thm:TU} does not hold; it holds only after substituting $\hb=1$ ($\h=0$), that is, turning off the $\C^\times_{\h}$ action.

\section*{Appendix: fixed point codes for affine type A bow varieties}

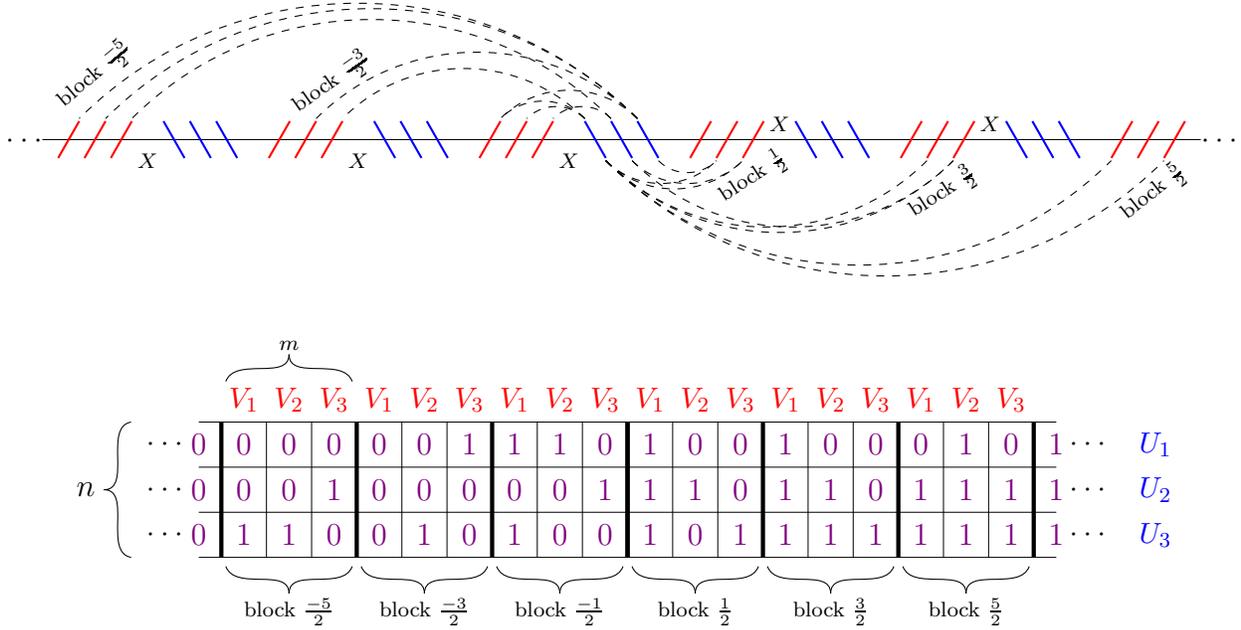
\begin{figure}
\[
\begin{tikzpicture}[scale=.35]
\draw[] (-1,0) -- (43,0);

\coordinate (a) at (0,0); \draw[thick,red]  ($(a) - (.4,.7)$) -- ($(a) + (.4,.7)$);
\coordinate (a) at (1,0); \draw[thick,red]  ($(a) - (.4,.7)$) -- ($(a) + (.4,.7)$);
\coordinate (a) at (2,0); \draw[thick,red]  ($(a) - (.4,.7)$) -- ($(a) + (.4,.7)$);

\coordinate (a) at (4,0); \draw[thick,blue]  ($(a) - (.4,-.7)$) -- ($(a) + (.4,-.7)$);
\coordinate (a) at (5,0); \draw[thick,blue]  ($(a) - (.4,-.7)$) -- ($(a) + (.4,-.7)$);
\coordinate (a) at (6,0); \draw[thick,blue]  ($(a) - (.4,-.7)$) -- ($(a) + (.4,-.7)$);

\coordinate (a) at (8,0); \draw[thick,red]  ($(a) - (.4,.7)$) -- ($(a) + (.4,.7)$);
\coordinate (a) at (9,0); \draw[thick,red]  ($(a) - (.4,.7)$) -- ($(a) + (.4,.7)$);
\coordinate (a) at (10,0); \draw[thick,red]  ($(a) - (.4,.7)$) -- ($(a) + (.4,.7)$);

\coordinate (a) at (12,0); \draw[thick,blue]  ($(a) - (.4,-.7)$) -- ($(a) + (.4,-.7)$);
\coordinate (a) at (13,0); \draw[thick,blue]  ($(a) - (.4,-.7)$) -- ($(a) + (.4,-.7)$);
\coordinate (a) at (14,0); \draw[thick,blue]  ($(a) - (.4,-.7)$) -- ($(a) + (.4,-.7)$);

\coordinate (a) at (16,0); \draw[thick,red]  ($(a) - (.4,.7)$) -- ($(a) + (.4,.7)$);
\coordinate (a) at (17,0); \draw[thick,red]  ($(a) - (.4,.7)$) -- ($(a) + (.4,.7)$);
\coordinate (a) at (18,0); \draw[thick,red]  ($(a) - (.4,.7)$) -- ($(a) + (.4,.7)$);

\coordinate (a) at (20,0); \draw[thick,blue]  ($(a) - (.4,-.7)$) -- ($(a) + (.4,-.7)$);
\coordinate (a) at (21,0); \draw[thick,blue]  ($(a) - (.4,-.7)$) -- ($(a) + (.4,-.7)$);
\coordinate (a) at (22,0); \draw[thick,blue]  ($(a) - (.4,-.7)$) -- ($(a) + (.4,-.7)$);

\coordinate (a) at (24,0); \draw[thick,red]  ($(a) - (.4,.7)$) -- ($(a) + (.4,.7)$);
\coordinate (a) at (25,0); \draw[thick,red]  ($(a) - (.4,.7)$) -- ($(a) + (.4,.7)$);
\coordinate (a) at (26,0); \draw[thick,red]  ($(a) - (.4,.7)$) -- ($(a) + (.4,.7)$);

\coordinate (a) at (28,0); \draw[thick,blue]  ($(a) - (.4,-.7)$) -- ($(a) + (.4,-.7)$);
\coordinate (a) at (29,0); \draw[thick,blue]  ($(a) - (.4,-.7)$) -- ($(a) + (.4,-.7)$);
\coordinate (a) at (30,0); \draw[thick,blue]  ($(a) - (.4,-.7)$) -- ($(a) + (.4,-.7)$);

\coordinate (a) at (32,0); \draw[thick,red]  ($(a) - (.4,.7)$) -- ($(a) + (.4,.7)$);
\coordinate (a) at (33,0); \draw[thick,red]  ($(a) - (.4,.7)$) -- ($(a) + (.4,.7)$);
\coordinate (a) at (34,0); \draw[thick,red]  ($(a) - (.4,.7)$) -- ($(a) + (.4,.7)$);

\coordinate (a) at (36,0); \draw[thick,blue]  ($(a) - (.4,-.7)$) -- ($(a) + (.4,-.7)$);
\coordinate (a) at (37,0); \draw[thick,blue]  ($(a) - (.4,-.7)$) -- ($(a) + (.4,-.7)$);
\coordinate (a) at (38,0); \draw[thick,blue]  ($(a) - (.4,-.7)$) -- ($(a) + (.4,-.7)$);

\coordinate (a) at (40,0); \draw[thick,red]  ($(a) - (.4,.7)$) -- ($(a) + (.4,.7)$);
\coordinate (a) at (41,0); \draw[thick,red]  ($(a) - (.4,.7)$) -- ($(a) + (.4,.7)$);
\coordinate (a) at (42,0); \draw[thick,red]  ($(a) - (.4,.7)$) -- ($(a) + (.4,.7)$);

\draw [dashed, black](20.4,-.8) to [out=-45,in=225] (24.6,-.8);
\draw [dashed, black](20.4,-.8) to [out=-45,in=225] (25.6,-.8);
\draw [dashed, black](21.4,-.8) to [out=-45,in=225] (25.6,-.8);
\draw [dashed, black](22.4,-.8) to [out=-45,in=225] (24.6,-.8);

\draw [dashed, black](20.4,-.8) to [out=-45,in=225] (32.6,-.8);
\draw [dashed, black](20.4,-.8) to [out=-45,in=225] (33.6,-.8);
\draw [dashed, black](21.4,-.8) to [out=-45,in=225] (33.6,-.8);

\draw [dashed, black](20.4,-.8) to [out=-45,in=225] (39.6,-.8);
\draw [dashed, black](20.4,-.8) to [out=-45,in=225] (41.6,-.8);

\draw [dashed, black](19.6,.8) to [out=135,in=45] (16.4,.8);
\draw [dashed, black](19.6,.8) to [out=135,in=45] (17.4,.8);
\draw [dashed, black](20.6,.8) to [out=135,in=45] (18.4,.8);
\draw [dashed, black](21.6,.8) to [out=135,in=45] (16.4,.8);

\draw [dashed, black](19.6,.8) to [out=135,in=45] (10.4,.8);
\draw [dashed, black](21.6,.8) to [out=135,in=45] (9.4,.8);

\draw [dashed, black](21.6,.8) to [out=135,in=45] (0.4,.8);
\draw [dashed, black](21.6,.8) to [out=135,in=45] (1.4,.8);
\draw [dashed, black](20.6,.8) to [out=135,in=45] (2.4,.8);

\node[] at (-1.6,-0.05) {$\cdots$};
\node[] at (43.8,-0.05) {$\cdots$};

\node[rotate=38] at (41.2,-2) {\tiny block $\frac{5}{2}$};
\node[rotate=33] at (33.2,-2) {\tiny block $\frac{3}{2}$};
\node[rotate=33] at (26,-1.4) {\tiny block $\frac{1}{2}$};
\node[rotate=40] at (1,2.5) {\tiny block $\frac{-5}{2}$};
\node[rotate=35] at (10,2.2) {\tiny block $\frac{-3}{2}$};

\node[] at (19,-.8) {\tiny $X$}; \node[] at (11,-.8) {\tiny $X$}; \node[] at (3,-.8) {\tiny $X$}; \node[] at (27,.6) {\tiny $X$}; \node[] at (35,.6) {\tiny $X$};

\end{tikzpicture}
\]
\[
\begin{tikzpicture}[scale=.6]
\draw[ultra thin] (-9.5,0) -- (9.5,0);\draw[ultra thin] (-9.5,1) -- (9.5,1);\draw[ultra thin] (-9.5,2) -- (9.5,2);\draw[ultra thin] (-9.5,3) -- (9.5,3);

\draw[ultra thick] (0,0) -- (0,3); \draw[ultra thin] (1,0) -- (1,3); \draw[ultra thin] (2,0) -- (2,3);
\draw[ultra thick] (3,0) -- (3,3); \draw[ultra thin] (4,0) -- (4,3); \draw[ultra thin] (5,0) -- (5,3);
\draw[ultra thick] (6,0) -- (6,3); \draw[ultra thin] (7,0) -- (7,3); \draw[ultra thin] (8,0) -- (8,3);
\draw[ultra thick] (9,0) -- (9,3);

\draw[ultra thick] (-0,0) -- (-0,3); \draw[ultra thin] (-1,0) -- (-1,3); \draw[ultra thin] (-2,0) -- (-2,3);
\draw[ultra thick] (-3,0) -- (-3,3); \draw[ultra thin] (-4,0) -- (-4,3); \draw[ultra thin] (-5,0) -- (-5,3);
\draw[ultra thick] (-6,0) -- (-6,3); \draw[ultra thin] (-7,0) -- (-7,3); \draw[ultra thin] (-8,0) -- (-8,3);
\draw[ultra thick] (-9,0) -- (-9,3);

\node[blue] at (11.7,.5) {\small $U_3$};\node[blue] at (11.7,1.5) {\small $U_2$};\node[blue] at (11.7,2.5) {\small $U_1$};

\coordinate (a) at (0,0); \node[violet] at ($(a)+(.5,0.5)$) {$1$}; \node[violet] at ($(a)+(.5,1.5)$) {$1$}; \node[violet] at ($(a)+(.5,2.5)$) {$1$}; \node[red] at ($(a)+(.5,3.5)$) {\small $V_1$};
\coordinate (a) at (1,0); \node[violet] at ($(a)+(.5,0.5)$) {$0$}; \node[violet] at ($(a)+(.5,1.5)$) {$1$}; \node[violet] at ($(a)+(.5,2.5)$) {$0$}; \node[red] at ($(a)+(.5,3.5)$) {\small $V_2$};
\coordinate (a) at (2,0); \node[violet] at ($(a)+(.5,0.5)$) {$1$}; \node[violet] at ($(a)+(.5,1.5)$) {$0$}; \node[violet] at ($(a)+(.5,2.5)$) {$0$}; \node[red] at ($(a)+(.5,3.5)$) {\small $V_3$};

\coordinate (a) at (3,0); \node[violet] at ($(a)+(.5,0.5)$) {$1$}; \node[violet] at ($(a)+(.5,1.5)$) {$1$}; \node[violet] at ($(a)+(.5,2.5)$) {$1$}; \node[red] at ($(a)+(.5,3.5)$) {\small $V_1$};
\coordinate (a) at (4,0); \node[violet] at ($(a)+(.5,0.5)$) {$1$}; \node[violet] at ($(a)+(.5,1.5)$) {$1$}; \node[violet] at ($(a)+(.5,2.5)$) {$0$}; \node[red] at ($(a)+(.5,3.5)$) {\small $V_2$};
\coordinate (a) at (5,0); \node[violet] at ($(a)+(.5,0.5)$) {$1$}; \node[violet] at ($(a)+(.5,1.5)$) {$0$}; \node[violet] at ($(a)+(.5,2.5)$) {$0$}; \node[red] at ($(a)+(.5,3.5)$) {\small $V_3$};

\coordinate (a) at (6,0); \node[violet] at ($(a)+(.5,0.5)$) {$1$}; \node[violet] at ($(a)+(.5,1.5)$) {$1$}; \node[violet] at ($(a)+(.5,2.5)$) {$0$}; \node[red] at ($(a)+(.5,3.5)$) {\small $V_1$};
\coordinate (a) at (7,0); \node[violet] at ($(a)+(.5,0.5)$) {$1$}; \node[violet] at ($(a)+(.5,1.5)$) {$1$}; \node[violet] at ($(a)+(.5,2.5)$) {$1$}; \node[red] at ($(a)+(.5,3.5)$) {\small $V_2$};
\coordinate (a) at (8,0); \node[violet] at ($(a)+(.5,0.5)$) {$1$}; \node[violet] at ($(a)+(.5,1.5)$) {$1$}; \node[violet] at ($(a)+(.5,2.5)$) {$0$}; \node[red] at ($(a)+(.5,3.5)$) {\small $V_3$};

\coordinate (a) at (-3,0); \node[violet] at ($(a)+(.5,0.5)$) {$1$}; \node[violet] at ($(a)+(.5,1.5)$) {$0$}; \node[violet] at ($(a)+(.5,2.5)$) {$1$}; \node[red] at ($(a)+(.5,3.5)$) {\small $V_1$};
\coordinate (a) at (-2,0); \node[violet] at ($(a)+(.5,0.5)$) {$0$}; \node[violet] at ($(a)+(.5,1.5)$) {$0$}; \node[violet] at ($(a)+(.5,2.5)$) {$1$}; \node[red] at ($(a)+(.5,3.5)$) {\small $V_2$};
\coordinate (a) at (-1,0); \node[violet] at ($(a)+(.5,0.5)$) {$0$}; \node[violet] at ($(a)+(.5,1.5)$) {$1$}; \node[violet] at ($(a)+(.5,2.5)$) {$0$}; \node[red] at ($(a)+(.5,3.5)$) {\small $V_3$};

\coordinate (a) at (-6,0); \node[violet] at ($(a)+(.5,0.5)$) {$0$}; \node[violet] at ($(a)+(.5,1.5)$) {$0$}; \node[violet] at ($(a)+(.5,2.5)$) {$0$}; \node[red] at ($(a)+(.5,3.5)$) {\small $V_1$};
\coordinate (a) at (-5,0); \node[violet] at ($(a)+(.5,0.5)$) {$1$}; \node[violet] at ($(a)+(.5,1.5)$) {$0$}; \node[violet] at ($(a)+(.5,2.5)$) {$0$}; \node[red] at ($(a)+(.5,3.5)$) {\small $V_2$};
\coordinate (a) at (-4,0); \node[violet] at ($(a)+(.5,0.5)$) {$0$}; \node[violet] at ($(a)+(.5,1.5)$) {$0$}; \node[violet] at ($(a)+(.5,2.5)$) {$1$}; \node[red] at ($(a)+(.5,3.5)$) {\small $V_3$};

\coordinate (a) at (-9,0); \node[violet] at ($(a)+(.5,0.5)$) {$1$}; \node[violet] at ($(a)+(.5,1.5)$) {$0$}; \node[violet] at ($(a)+(.5,2.5)$) {$0$}; \node[red] at ($(a)+(.5,3.5)$) {\small $V_1$};
\coordinate (a) at (-8,0); \node[violet] at ($(a)+(.5,0.5)$) {$1$}; \node[violet] at ($(a)+(.5,1.5)$) {$0$}; \node[violet] at ($(a)+(.5,2.5)$) {$0$}; \node[red] at ($(a)+(.5,3.5)$) {\small $V_2$};
\coordinate (a) at (-7,0); \node[violet] at ($(a)+(.5,0.5)$) {$0$}; \node[violet] at ($(a)+(.5,1.5)$) {$1$}; \node[violet] at ($(a)+(.5,2.5)$) {$0$}; \node[red] at ($(a)+(.5,3.5)$) {\small $V_3$};

\coordinate (a) at (-10,0); \node[violet] at ($(a)+(.5,0.5)$) {$0$}; \node[violet] at ($(a)+(.5,1.5)$) {$0$}; \node[violet] at ($(a)+(.5,2.5)$) {$0$}; %\node[red] at ($(a)+(.5,3.5)$) {\small $V_3$};
\coordinate (a) at (9,0); \node[violet] at ($(a)+(.5,0.5)$) {$1$}; \node[violet] at ($(a)+(.5,1.5)$) {$1$}; \node[violet] at ($(a)+(.5,2.5)$) {$1$}; %\node[red] at ($(a)+(.5,3.5)$) {\small $V_3$};

\coordinate (a) at (0,0); \draw  [thin] ($(a)+(0.1,-0.2)$) to [out=290,in=90] ($(a)+(1.5,-.8)$); \draw [ thin] ($(a)+(1.5,-.8)$) to [out=90,in=250] ($(a)+(2.9,-0.2)$); \node at ($(a)+(1.5,-1.2)$) {\tiny block $\frac{1}{2}$};

\coordinate (a) at (3,0); \draw  [thin] ($(a)+(0.1,-0.2)$) to [out=290,in=90] ($(a)+(1.5,-.8)$); \draw [ thin] ($(a)+(1.5,-.8)$) to [out=90,in=250] ($(a)+(2.9,-0.2)$); \node at ($(a)+(1.5,-1.2)$) {\tiny block $\frac{3}{2}$};

\coordinate (a) at (6,0); \draw  [thin] ($(a)+(0.1,-0.2)$) to [out=290,in=90] ($(a)+(1.5,-.8)$); \draw [ thin] ($(a)+(1.5,-.8)$) to [out=90,in=250] ($(a)+(2.9,-0.2)$); \node at ($(a)+(1.5,-1.2)$) {\tiny block $\frac{5}{2}$};

\coordinate (a) at (-3,0); \draw  [thin] ($(a)+(0.1,-0.2)$) to [out=290,in=90] ($(a)+(1.5,-.8)$); \draw [ thin] ($(a)+(1.5,-.8)$) to [out=90,in=250] ($(a)+(2.9,-0.2)$); \node at ($(a)+(1.5,-1.2)$) {\tiny block $\frac{-1}{2}$};

\coordinate (a) at (-6,0); \draw  [thin] ($(a)+(0.1,-0.2)$) to [out=290,in=90] ($(a)+(1.5,-.8)$); \draw [ thin] ($(a)+(1.5,-.8)$) to [out=90,in=250] ($(a)+(2.9,-0.2)$); \node at ($(a)+(1.5,-1.2)$) {\tiny block $\frac{-3}{2}$};

\coordinate (a) at (-9,0); \draw  [thin] ($(a)+(0.1,-0.2)$) to [out=290,in=90] ($(a)+(1.5,-.8)$); \draw [ thin] ($(a)+(1.5,-.8)$) to [out=90,in=250] ($(a)+(2.9,-0.2)$); \node at ($(a)+(1.5,-1.2)$) {\tiny block $\frac{-5}{2}$};

\draw[thin] (-11,3) to [out=190,in=0] (-11.6,1.5) to [out=0,in=170] (-11,0); \node at (-12,1.5) {$n$}; 
 
\coordinate (a) at (-9,0); \draw [thin] ($(a)+(0.1,3.9)$) to [out=-290,in=-90] ($(a)+(1.5,4.5)$); \draw [ thin] ($(a)+(1.5,4.5)$) to [out=-90,in=-250] ($(a)+(2.9,3.9)$); \node at ($(a)+(1.5,4.7)$) {\tiny $m$}; 
 \node at (-10.2,1.5) {$\cdots$}; \node at (-10.2,2.5) {$\cdots$}; \node at (-10.2,0.5) {$\cdots$}; 
 \node at (10.25,1.5) {$\cdots$};\node at (10.25,2.5) {$\cdots$};\node at (10.25,0.5) {$\cdots$};
\end{tikzpicture}
\]
\caption{Top: a tie diagram of a fixed point in $\Ch(\DD)$ where $\DD$ is a separated brane diagram of affine type~A. Bottom: the corresponding Maya diagram of \cite[Appendix A]{Nakajima_Satake}.} \label{fig:Maya}
\end{figure}

In this appendix we describe the relation between our combinatorial codes and the ones called Maya diagrams in \cite[Appendix A]{Nakajima_Satake}.

Let us consider an affine type A brane diagram, that is, let the 5-branes be arranged around a cycle.  Instead of drawing the diagram on the cycle, we draw it on the universal cover, that is, a periodic brane diagram on an infinite line, see Figure~\ref{fig:Maya}. By applying some Hanany-Witten transitions we may assume that the NS5 branes and the D5 branes are separated on the cycle---we call this brane diagram {\em separated}. In our figure then the D5 branes come in groups, say, around integer positions, and the NS5 branes come in groups positioned at half-integer positions. 

Tie diagrams of fixed point codes for such a brane diagram continue to make sense. However, the notion of brane charge is not defined; hence the table-with-margin code does not make sense. Instead, the analogous code is described in \cite{Nakajima_Satake}, that we sketch now, together with the corresponding tie diagram.

\smallskip 

Consider a representative of a tie diagram of a fixed point where all the ties are attached to the D5 branes in the group {\em at position 0}. Then the ties come in blocks corresponding to half integers, according to the position of the other end of the tie---see the figure. Consider the ``BCTs'' of the ties in block $\frac{k}{2}$ as follows:
\begin{itemize}
\item for $k>0$, the $(U,V)$ entry is 0 if there is a $U$-$V$ tie, otherwise 1;
\item for $k<0$, the $(U,V)$ entry is 1 if there is a $U$-$V$ tie, otherwise 0.
\end{itemize}
Then glue these ``BCTs'' together to form an $n\times \infty$ table, called {\em Maya diagram}, see the bottom table in Figure~\ref{fig:Maya}. 
%(In fact, we transposed the BCT's so that the infinite direction is horizontal, so that we match the convention of \cite{Nakajima_Satake}). 
By abuse of language, the `BCT' corresponding to block $\frac{k}{2}$ will also be called ``block $\frac{k}{2}$''. 

The following properties of the obtained Maya diagram can be read from the corresponding tie diagram. 
\begin{enumerate} 
\item For large enough $k$ all entries of block $\frac{k}{2}$ are 1, and all entries of block $\frac{-k}{2}$ are 0.
\item For $1\leq i \leq n$ we have that $\mult_{U_i^+}-\mult_{U_i^-}=$
\[
\#\{0\text{s in row $i$ of positive blocks}\}-  \#\{1\text{s in row $i$ of negative blocks}\}.
\]  
\item For $1\leq j \leq m$ we have that $\mult_{V_j^+}-\mult_{V_j^-}=$
\[
\#\{1\text{s in a $V_j$ column of negative blocks}\}-  \#\{0\text{s in a $V_j$ column of negative blocks}\}.
\]
\item Let $X$ be the D3 brane for which $X^+$ is a D5 brane and $X^-$ is an NS5 brane. Then
\begin{align*}
\mult_X=&\#\{1s \text{ in block } \frac{-1}{2}\}+2\#\{1s \text{ in block } \frac{-3}{2}\}+3\#\{1s \text{ in block } \frac{-5}{2}\}+\ldots
\\
&+\#\{0s \text{ in block } \frac{3}{2}\}+2\#\{0s \text{ in block } \frac{5}{2}\}+3\#\{0s \text{ in block } \frac{7}{2}\}+\ldots.
\end{align*}
\end{enumerate}
It is proved in \cite{Nakajima_Satake} that Maya diagrams (ie. $n\times \infty$ tables that come as a union of $\infty$ many $n\times m$ tables called blocks, that have properties (1)--(4)) are in bijection with the torus fixed points of a bow variety associated with a separated affine type $A$ brane diagram. 

Since Figure \ref{fig:HWfixpoints} still holds as a proof of bijection between torus fixed points of HW equivalent brane diagrams, this statement describes the fixed points of all affine type A bow varieties.

\end{document}